\documentclass[10pt]{article}
\usepackage{amsthm}
\usepackage{amsmath}
\usepackage{amssymb}
\usepackage{makeidx}

\theoremstyle{definition}
\newtheorem{definition}{Definition}[section]

\newtheorem{example}[definition]{Example}

\newtheorem{remark}[definition]{Remark}

\theoremstyle{plain}
\newtheorem{theorem}[definition]{Theorem}
\newtheorem{lemma}[definition]{Lemma}
\newtheorem{proposition}[definition]{Proposition}
\newtheorem{corollary}[definition]{Corollary}

\def\N{{\mathbb N}}

\def\C{{\cal C}}

\begin{document}
\title{\vspace*{-55pt}Functorial Properties of the Reticulation of a Universal Algebra}
\author{George GEORGESCU$^1$, Leonard KWUIDA$^2$, Claudia MURE\c SAN$^3$\thanks{Corresponding author.}\\ 
{\small $^{1,3}$University of Bucharest, $^2$Bern University of Applied Sciences}\\ 
{\small $^1$georgescu.capreni@yahoo.com; $^2$leonard.kwuida@bfh.c; $^3$cmuresan@fmi.unibuc.ro; c.muresan@yahoo.com}}
\date{\today }
\maketitle \vspace*{-20pt}

\begin{abstract} The {\em reticulation} of an algebra $A$ is a bounded distributive lattice whose prime spectrum of ideals (or filters), endowed with the Stone topology, is homeomorphic to the prime spectrum of congruences of $A$, with its own Stone topology. The reticulation allows algebraic and topological properties to be transferred between the algebra $A$ and bounded distributive lattices, a transfer which is facilitated if we can define a {\em reticulation functor} from a variety containing $A$ to the variety of (bounded) distributive lattices. In this paper, we continue the study of the reticulation of a universal algebra initiated in \cite{retic}, where we have used the notion of a prime congruence introduced through the term condition commutator. We characterize morphisms which admit an image through the reticulation and investigate the kinds of varieties that admit reticulation functors; we prove that these include semi--degenerate congruence--distributive varieties with the Compact Intersection Property and semi--degenerate congruence--distributive varieties with congruence intersection terms, as well as generalizations of these, and additional varietal properties ensure that the reticulation functors preserve the injectivity of morphisms. We also study the property of morphisms of having an image through the reticulation in relation to another property, involving the complemented elements of congruence lattices, exemplify the transfer of properties through the reticulation with conditions Going Up, Going Down, Lying Over and the Congruence Boolean Lifting Property, and illustrate the applicability of such a transfer by using it to derive results for certain types of varieties from properties of bounded distributive lattices.
     

\end{abstract}

\section{Introduction}
\label{introduction}

The {\em reticulation} of an algebra $A$ from a variety ${\cal C}$ is a bounded distributive lattice ${\cal L}(A)$ such that the spectrum of the prime congruences of $A$, endowed with the Stone topology, is homeomorphic to the spectrum of the prime ideals or the prime filters of $A$, endowed with its Stone topology. This construction allows algebraic and topological properties to be transferred between ${\cal C}$ and the variety ${\cal D}{\bf 01}$ of bounded distributive lattices. While a known property of bounded distributive lattices ensures the uniqueness of ${\cal L}(A)$ up to a lattice isomorphism (once we have chosen, for its construction, either its spectrum of prime ideals or that of its prime filters, since the reticulation constructed w.r.t. to one of these prime spectra is dually lattice isomorphic to the one constructed w.r.t. the other), prior to our construction for the setting of universal algebra from \cite{retic}, the existence of the reticulation had only been proven for several concrete varieties ${\cal C}$, out of which we mention: commutative unitary rings \cite{joy,sim}, unitary rings \cite{bell}, MV--algebras \cite{bellsspl}, BL--algebras \cite{dinggll} and (bounded commutative integral) residuated lattices \cite{eu1,eu2,eu5}.

In \cite{retic}, we have constructed the reticulation for any algebra whose one--class congruence is compact, whose term condition commutator is commutative and distributive w.r.t. arbitrary joins and whose set of compact congruences is closed w.r.t. this commutator operation; in particular, our construction can be applied to any algebra from a semi--degenerate congruence--modular variety having the set of the compact congruences closed w.r.t. the modular commutator, hence this construction generalizes all previous constructions of the reticulation for particular varieties.

We recall our construction from \cite{retic} for the reticulation in this universal algebra setting in Section \ref{reticulatia}, after a preliminaries section in which we remind some notions from universal algebra and establish several notations.

A very useful tool for transferring properties through the reticulation between ${\cal C}$ and ${\cal D}{\bf 01}$ is a reticulation functor ${\cal L}:{\cal C}\rightarrow {\cal D}{\bf 01}$, whose preservation properties can be used for such a transfer. In \cite{retic}, we have defined an image through the reticulation for any surjective morphism between algebras satisfying the conditions above for the compact congruences and the term condition commutator. In Section \ref{fret} of the present paper we introduce the {\em functoriality of the reticulation}, which essentially means, for an arbitrary morphism $f:A\rightarrow B$ in $\C $ between algebras $A$ and $B$ from $\C $ having the commutators with the properties above, that $f$ admits an image ${\cal L}(f)$ through the reticulation, that is $f$ induces a $0$ and join--preserving function ${\cal L}(f):{\cal L}(A)\rightarrow {\cal L}(B)$. We have a functor from $\C $ to the variety of (bounded) distributive lattices iff all morphisms in $\C $ satisfy the functoriality of the reticulation and their images through the reticulation also preserve the meet (and the $1$). It turns out that the {\em admissible morphisms} we have studied in \cite{gulo,euadm}, that is the morphisms whose inverse images take prime congruences to prime congruences, are exactly the morphisms satisfying the functoriality of the reticulation and whose images through the reticulation are lattice morphisms. Unfortunately, we have not been able to construct a reticulation functor in the most general case for which we have constructed the reticulation, but we have obtained reticulation functors for remarkable kinds of varieties, such as semi--degenerate congruence--distributive varieties with the Compact Intersection Property (CIP) and semi--degenerate congruence--modular varieties with {\em compact commutator terms}, a notion we have defined by analogy to the more restrictive one of a congruence--distributive variety with compact intersection terms. Varieties with stronger properties, such as semi--degenerate congruence--extensible congruence--distributive varieties with the CIP or semi--degenerate varieties with equationally definable principal congruences (EDPC) and the CIP turn out to have reticulation functors which preserve the injectivity of morphisms. We conclude this section by transferring properties Going Up, Going Down and Lying Over on admissible morphisms through the reticulation, and using this transfer to derive a result on varieties with EDPC, as an illustration of the applicability of the reticulation.

In Section \ref{fbc} we study the functoriality of the reticulation in relation with another property of morphisms, that we call {\em functoriality of the Boolean center}, involving the complemented elements of the congruence lattice of an algebra $A$, which form a Boolean sublattice of the lattice of congruences of $A$, called the {\em Boolean center} of this congruence lattice, whenever $A$ satisfies the conditions above on compact congruences and the term condition commutator and, additionally, has the property that the term condition commutator of any congruence $\alpha $ of $A$ with the one--class congruence of $A$ equals $\alpha $, in particular whenever $A$ is a member of a semi--degenerate congruence--modular variety and has the set of the compact congruences closed w.r.t. the modular commutator. The functoriality of the Boolean center on a morphism $f:A\rightarrow B$ in $\C $ between algebras with the commutators as above essentially means that $f$ induces a Boolean morphism between the Boolean centers of the congruence lattices of $A$ and $B$; if all morphisms in $\C $ have this property, then we can define a functor from $\C $ to the variety of Boolean algebras. We also study another property related to these Boolean centers, namely the Congruence Boolean Lifting Property (CBLP), which turns out to be transferrable through the reticulation in the case when $\C $ is semi--degenerate and congruence--modular.

We conclude our paper with Section \ref{examples}, containing examples for the notions in the previous sections and the relations between these notions.

\section{Preliminaries}
\label{preliminaries}

We refer the reader to \cite{agl}, \cite{bur}, \cite{gralgu}, \cite{koll} for a further study of the following notions from universal algebra, to \cite{bal}, \cite{blyth}, \cite{cwdw}, \cite{gratzer} for the lattice--theoretical ones, to \cite{agl}, \cite{fremck}, \cite{koll}, \cite{ouwe} for the results on commutators and to \cite{agl}, \cite{cze}, \cite{cze2}, \cite{gulo}, \cite{euadm}, \cite{joh} for the Stone topologies.

All algebras will be non--empty and they will be designated by their underlying sets; by {\em trivial algebra} we mean one--element algebra. For brevity, we denote by $A\cong B$ the fact that two algebras $A$ and $B$ of the same type are isomorphic. We abbreviate by {\em CIP} and {\em PIP} the Compact Intersection Property and the principal intersection property, respectively.

$\N $ denotes the set of the natural numbers, $\N ^*=\N \setminus \{0\}$, and, for any $a,b\in \N $, we denote by $\overline{a,b}$ the interval in the lattice $(\N ,\leq )$ bounded by $a$ and $b$, where $\leq $ is the natural order. Let $M$, $N$ be sets and $S\subseteq M$. Then ${\cal P}(M)$ denotes the set of the subsets of $M$ and $({\rm Eq}(M),\vee ,\cap ,\Delta _M=\{(x,x)\ |\ x\in M\},\nabla _M=M^2)$ is the bounded lattice of the equivalences on $M$. We denote by $i_{S,M}:S\rightarrow M$ the inclusion map and by $id_M=i_{M,M}$ the identity map of $M$. For any function $f:M\rightarrow N$, we denote by ${\rm Ker}(f)$ the kernel of $f$, by $f$ the direct image of $f^2=f\times f$ and by $f^*$ the inverse image of $f^2$.

Let $L$ be a lattice. Then ${\rm Cp}(L)$ denotes the set of the compact elements of $L$, and ${\rm Id}(L)$ and ${\rm Spec}_{\rm Id}(L)$ denote the set of the ideals and that of the prime ideals of $L$, respectively. Let $U\subseteq L$ and $u\in L$. Then $[U)$ and $[u)$ denote the filters of $L$ generated by $U$ and by $u$, respectively, while $(U]$ and $(u]$ denote the ideals of $L$ generated by $U$ and by $u$, respectively.

We denote by ${\cal L}_n$ the $n$--element chain for any $n\in \N ^*$, by ${\cal M}_3$ the five--element modular non--distributive lattice and by ${\cal N}_5$ the five--element non--modular lattice. Recall that a {\em frame} is a complete lattice with the meet distributive w.r.t. arbitrary joins.

Throughout this paper, by {\em functor} we mean covariant functor. ${\cal B}$ denotes the functor from the variety of bounded distributive lattices to the variety of Boolean algebras which takes each bounded distributive lattice to its Boolean center and every morphism in the former variety to its restriction to the Boolean centers. If $L$ is a bounded lattice, then we denote by ${\cal B}(L)$ the set of the complemented elements of $L$ even if $L$ is not distributive.

Throughout the rest of this paper, $\tau $ will be a universal algebras signature, $\C $ an equational class of $\tau $--algebras and $A$ an arbitrary member of $\C $. Unless mentioned otherwise, by {\em morphism} we mean $\tau $--morphism.

${\rm Con}(A)$, ${\rm Max}(A)$, ${\rm PCon}(A)$ and ${\cal K}(A)$ denote the sets of the congruences, maximal congruences, principal congruences and finitely generated congruences of $A$, respectively; note that ${\cal K}(A)$ is the set of the compact elements of the lattice ${\rm Con}(A)$. ${\rm Max}(A)$ is called the {\em maximal spectrum} of $A$. For any $X\subseteq A^2$ and any $a,b\in A$, $Cg_A(X)$ will be the congruence of $A$ generated by $X$ and we shall denote by $Cg_A(a,b)=Cg_A(\{(a,b)\})$.

For any $\theta \in {\rm Con}(A)$, $p_{\theta }:A\rightarrow A/\theta $ will be the canonical surjective morphism; given any $X\in A\cup A^2\cup {\cal P}(A)\cup {\cal P}(A^2)$, we denote by $X/\theta =p_{\theta }(X)$. If $L$ is a distributive lattice, so that we have the canonical lattice embedding $\iota _L:{\rm Id}(L)\rightarrow {\rm Con}(L)$, then we will denote, for every $I\in {\rm Id}(L)$, by $\pi _I=p_{\iota _L(I)}:L\rightarrow L/I$.

Recall that, if $B$ is a member of $\C $ and $f:A\rightarrow B$ is a morphism, then, for any $\alpha \in {\rm Con}(A)$ and any $\beta \in {\rm Con}(B)$, we have $f^*(\beta )\in [{\rm Ker}(f))\subseteq {\rm Con}(A)$, $f(f^*(\beta ))=\beta \cap f(A^2)\subseteq \beta $ and $\alpha \subseteq f^*(f(\alpha ))$; if $\alpha \in [{\rm Ker}(f))$, then 
$f(\alpha )\in {\rm Con}(f(A))$ and $f^*(f(\alpha ))=\alpha $. Hence $\theta \mapsto f(\theta )$ is a lattice isomorphism from $[{\rm Ker}(f))$ to ${\rm Con}(f(A))$ and thus it sets an order isomorphism from ${\rm Max}(A)\cap [{\rm Ker}(f))$ to ${\rm Max}(f(A))$. For the next lemma, note that ${\rm Ker}(p_{\theta })=\theta $ for any $\theta \in {\rm Con}(A)$, and that $Cg_A(Cg_S(X))=Cg_A(X)$ for any subalgebra $S$ of $A$ and any $X\subseteq S^2$.

\begin{lemma}{\rm \cite[Lemma $1.11$]{bak}, \cite[Proposition $1.2$]{urs5}} If $B$ is a member of $\C $ and $f:A\rightarrow B$ is a morphism, then, for any $X\subseteq A^2$ and any $\alpha ,\theta \in {\rm Con}(A)$:\begin{itemize}
\item $f(Cg_A(X)\vee {\rm Ker}(f))=Cg_{f(A)}(f(X))$, so $Cg_B(f(Cg_A(X)))=Cg_B(f(X))$ and $(Cg_A(X)\vee \theta )/\theta =$\linebreak $Cg_{A/\theta }(X/\theta )$;
\item in particular, $f(\alpha \vee {\rm Ker}(f))=Cg_{f(A)}
(f(\alpha ))$, so $(\alpha \vee \theta )/\theta =Cg_{A/\theta }(\alpha /\theta )$.\end{itemize}\label{fcg}\end{lemma}

If $B$ is a member of $\C $ and $f:A\rightarrow B$ is a morphism, then, for any non--empty family $(\alpha _i)_{i\in I}\subseteq [{\rm Ker}(f))$, we have, in ${\rm Con}(f(A))$: $\displaystyle f(\bigvee _{i\in I}\alpha _i)=\bigvee _{i\in I}f(\alpha _i)$. Indeed, by Lemma \ref{fcg}, $\displaystyle f(\bigvee _{i\in I}\alpha _i)=f(Cg_A(\bigcup _{i\in I}\alpha _i))=$\linebreak $\displaystyle Cg_{f(A)}(f(\bigcup _{i\in I}\alpha _i))=Cg_{f(A)}(\bigcup _{i\in I}f(\alpha _i))=\bigvee _{i\in I}f(\alpha _i)$.

We use the following definition from \cite{mcks} for the {\em term condition commutator}: let $\alpha ,\beta \in {\rm Con}(A)$. For any $\mu \in {\rm Con}(A)$, by $C(\alpha ,\beta ;\mu )$ we denote the fact that the following condition holds: for all $n,k\in \N $ and any term $t$ over $\tau $ of arity $n+k$, if $(a_i,b_i)\in \alpha $ for all $i\in \overline{1,n}$ and $(c_j,d_j)\in \beta $ for all $j\in \overline{1,k}$, then $(t^A(a_1,\ldots ,a_n,c_1,\ldots ,c_k),t^A(a_1,\ldots ,a_n,d_1,\ldots ,d_k))\in \mu $ iff $(t^A(b_1,\ldots ,b_n,c_1,\ldots ,c_k),t^A(b_1,\ldots ,b_n,d_1,\ldots ,d_k))\in \mu $. We denote by $[\alpha ,\beta ]_A=\bigcap \{\mu \in {\rm Con}(A)\ |\ C(\alpha ,\beta ;\mu )\}$; we call $[\alpha ,\beta ]_A$ the {\em commutator of $\alpha $ and $\beta $} in $A$. The operation $[\cdot ,\cdot ]_A:{\rm Con}(A)\times {\rm Con}(A)\rightarrow {\rm Con}(A)$ is called the {\em commutator of $A$}.

By \cite{fremck}, if $\C $ is congruence--modular, then, for each member $M$ of $\C $, $[\cdot ,\cdot ]_M$ is the unique binary operation on ${\rm Con}(M)$ such that, for all $\alpha ,\beta \in {\rm Con}(M)$, $[\alpha ,\beta ]_M=\min \{\mu \in {\rm Con}(M)\ |\ \mu \subseteq \alpha \cap \beta $ and, for any member $N$ of $\C $ and any surjective morphism $h:M\rightarrow N$ in $\C $, $\mu \vee {\rm Ker}(h)=h^*([h(\alpha \vee {\rm Ker}(h)),h(\beta \vee {\rm Ker}(h))]_N)\}$. Therefore, if $\C $ is congruence--modular, $\alpha ,\beta ,\theta \in {\rm Con}(A)$ and $f$ is surjective, then $[f(\alpha \vee {\rm Ker}(f)),f(\beta \vee {\rm Ker}(f))]_B=f([\alpha ,\beta ]_A\vee {\rm Ker}(f))$, thus $[(\alpha \vee \theta )/\theta ,(\beta \vee \theta )/\theta]_B=([\alpha ,\beta ]_A\vee \theta )/\theta $, hence, if $\theta \subseteq \alpha \cap \beta $, then $[\alpha /\theta ,\beta /\theta ]_{A/\theta }=([\alpha ,\beta ]_A\vee \theta )/\theta $, and, if, moreover, $\theta \subseteq [\alpha ,\beta ]_A$, then $[\alpha /\theta ,\beta /\theta ]_{A/\theta }=[\alpha ,\beta ]_A/\theta $.

By \cite[Lemma 4.6,Lemma 4.7,Theorem 8.3]{mcks}, the commutator is smaller than the intersection and increasing in both arguments. If $\C $ is congruence--modular, then the commutator is also commutative and distributive in both arguments with respect to arbitrary joins. By \cite{bj}, if $\C $ is congruence--distributive, then, in each member of $\C $, the commutator coincides to the intersection of congruences. Clearly, if the commutator of $A$ coincides to the intersection of congruences, then ${\rm Con}(A)$ is a frame, in particular it is congruence--distributive. Recall, however, that, since the lattice ${\rm Con}(A)$ is complete and algebraic, thus upper continuous, ${\rm Con}(A)$ is a frame whenever it is distributive.

By \cite[Theorem 8.5, p. 85]{fremck}, if $\C $ is congruence--modular, then the following are equivalent:\begin{itemize}
\item for any algebra $M$ from $\C $, $[\nabla _M,\nabla _M]_M=\nabla _M$;
\item for any algebra $M$ from $\C $ and any $\theta \in {\rm Con}(M)$, $[\theta ,\nabla _M]_M=\theta $;
\item $\C $ has no skew congruences, that is, for any algebras $M$ and $N$ from $\C $, ${\rm Con}(M\times N)=\{\theta \times \zeta \ |\ \theta \in {\rm Con}(M),\zeta \in {\rm Con}(N)\}$.\end{itemize}

Recall that $\C $ is said to be {\em semi--degenerate} iff no non--trivial algebra in $\C $ has one--element subalgebras. By \cite{koll}, $\C $ is semi--degenerate iff, for all members $M$ of $\C $, $\nabla _M\in {\cal K}(M)$. By \cite[Lemma 5.2]{agl} and the fact that, in congruence--distributive varieties, the commutator coincides to the intersection, we have: if $\C $ is either congruence--distributive or both congruence--mo\-du\-lar and semi--degenerate, then $\C $ has no skew congruences.

If $[\cdot ,\cdot ]_A$ is commutative and distributive w.r.t. the join (in particular if $\C $ is congruence--modular), then, if $A$ has principal commutators, that is $[{\rm PCon}(A),{\rm PCon}(A)]_A\subseteq {\rm PCon}(A)$, then $[{\cal K}(A),{\cal K}(A)]_A\subseteq {\cal K}(A)$.

We denote the set of the {\em prime congruences} of $A$ by ${\rm Spec}(A)$. As defined in \cite{fremck}, ${\rm Spec}(A)=\{\phi \in {\rm Con}(A)\setminus \{\nabla _A\}\ |\ (\forall \, \alpha ,\beta \in {\rm Con}(A))\, ([\alpha ,\beta ]_A\subseteq \phi \Rightarrow \alpha \subseteq \phi \mbox{ or }\beta \subseteq \phi )\}$. ${\rm Spec}(A)$ is called the {\em (prime) spectrum} of $A$. Recall that ${\rm Spec}(A)$ is not necessarily non--empty. However, by \cite[Theorem $5.3$]{agl}, if $\C $ is congruence--modular and semi--de\-ge\-ne\-rate, then any proper congruence of $A$ is included in a maximal congruence of $A$, and any maximal congruence of $A$ is prime. Recall, also, that, if $\C $ is congruence--modular, $B$ is a member of $\C $ and $f:A\rightarrow B$ is a morphism, then the map $\alpha \mapsto f(\alpha )$ is an order isomorphism from ${\rm Spec}(A)\cap [{\rm Ker}(f))$ to ${\rm Spec}(f(A))$, thus to ${\rm Spec}(B)$ if $f$ is surjective, case in which its inverse is $f^*\mid _{{\rm Spec}(B)}:{\rm Spec}(B)\rightarrow {\rm Spec}(A)$. In \cite{gulo},\cite{euadm}, we have called $f$ an {\em admissible morphism} iff $f^*({\rm Spec}(B))\subseteq {\rm Spec}(A)$.

\begin{remark} By the above, if $f$ is surjective, then $f$ is admissible.\end{remark}

Assume that $[\cdot ,\cdot ]_A$ is commutative and distributive w.r.t. arbitrary joins and that ${\rm Spec}(A)$ is non--empty, which hold if $\C $ is congruence--modular and semi--degenerate and $A$ is non--trivial. For each $\theta \in {\rm Con}(A)$, we denote by $V_A(\theta )={\rm Spec}(A)\cap [\theta )$ and by $D_A(\theta )={\rm Spec}(A)\setminus V_A(\theta )$. Then, by \cite{agl} and \cite{retic}, $({\rm Spec}(A),\{D_A(\theta )\ |\ \theta \in {\rm Con}(A)\})$ is a topological space in which, for all $\alpha ,\beta \in {\rm Con}(A)$ and any family $(\alpha _i)_{i\in I}\subseteq {\rm Con}(A)$, the following hold:\begin{itemize}
\item $V_A([\alpha ,\beta ]_A)=V_A(\alpha \cap \beta )=V_A(\alpha )\cup V_A(\beta )$ and $\displaystyle V_A(\bigvee _{i\in I}\alpha _i)=\bigcap _{i\in I}V_A(\alpha _i)$;
\item if $\C $ is congruence--modular and semi--degenerate, then: $V_A(\alpha )=\emptyset $ iff $\alpha =\nabla _A$.\end{itemize}

$\{D_A(\theta )\ |\ \theta \in {\rm Con}(A)\}$ is called the {\em Stone topology} on ${\rm Spec}(A)$ and it has $\{D_A(Cg_A(a,b))\ |\ a,b\in A\}$ as a basis. In the same way, but replacing congruences with ideals, one defines the Stone topology on the set of prime ideals of a bounded distributive lattice.

\section{The Construction of the Reticulation of a Universal Algebra and Related Results}
\label{reticulatia}

In this section, we recall the construction for the reticulation of $A$ from \cite{retic} and point out its basic properties. Throughout this section, we shall assume that $[\cdot ,\cdot ]_A$ is commutative and distributive w.r.t. arbitrary joins and that $\nabla _A\in {\cal K}(A)$, which hold in the particular case when $\C $ is congruence--modular and semi--degenerate.

For every $\theta \in {\rm Con}(A)$, we denote by $\rho _A(\theta )$ the {\em radical} of $\theta $: $\displaystyle \rho _A(\theta )=\bigcap \{\phi \in {\rm Spec}(A)\ |\ \theta \subseteq \phi \}=\bigcap _{\phi \in V_A(\theta )}\phi $. We denote by ${\rm RCon}(A)$ the set of the {\em radical congruences} of $A$: ${\rm RCon}(A)=\{\rho _A(\theta )\ |\ \theta \in {\rm Con}(A)\}=\{\theta \in {\rm Con}(A)\ |\ \theta =\rho _A(\theta )\}=\{\bigcap M\ |\ M\subseteq {\rm Spec}(A)\}$. If the commutator of $A$ equals the intersection (so that $A$ is congruence--distributive), in particular if $\C $ is congruence--distributive, then ${\rm Spec}(A)$ is the set of the prime elements of the lattice ${\rm Con}(A)$, thus its set of meet--irreducible elements, hence ${\rm RCon}(A)={\rm Con}(A)$ since the lattice ${\rm Con}(A)$ is algebraic.

Note that, for any $\alpha ,\beta ,\theta \in {\rm Con}(A)$, the following equivalences hold: $\alpha \subseteq \rho _A(\beta )$ iff $\rho _A(\alpha )\subseteq \rho _A(\beta )$ iff $V_A(\alpha )\supseteq V_A(\beta )$; thus $\rho _A(\alpha )=\rho _A(\beta )$ iff $V_A(\alpha )=V_A(\beta )$. By the above and the properties of the Stone topology on ${\rm Spec}(A)$ recalled in Section \ref{preliminaries}, we have proven, in \cite{retic}, that, for any $n\in \N ^*$, any $\alpha ,\beta \in {\rm Con}(A)$ and any $(\alpha _i)_{i\in I}\subseteq {\rm Con}(A)$, we have:\begin{itemize}
\item $\rho _A(\rho _A(\alpha ))=\rho _A(\alpha )$; $\alpha \subseteq \rho _A(\beta )$ iff $\rho _A(\alpha )\subseteq \rho _A(\beta )$; $\rho _A(\alpha )=\alpha $ iff $\alpha \in {\rm RCon}(A)\supseteq {\rm Spec}(A)$;
\item $\displaystyle \rho _A(\bigvee _{i\in I}\alpha _i)=\rho _A(\bigvee _{i\in I}\rho _A(\alpha _i))=\bigvee _{i\in I}\rho _A(\alpha _i)$; $\rho _A([\alpha ,\beta ]_A^n)=\rho _A([\alpha ,\beta ]_A)=\rho _A(\alpha \wedge \beta )=\rho _A(\alpha )\wedge \rho _A(\beta )$;
\item $\rho _A(\nabla _A)=\nabla _A$; if $\C $ is congruence--modular and semi--degenerate, then: $\rho _A(\alpha )=\nabla _A$ iff $\alpha =\nabla _A$;
\item $\rho _{A/\theta }((\alpha \vee \theta )/\theta )=\rho _A(\alpha \vee \theta )/\theta $.\end{itemize}

If we define $\equiv _A=\{(\alpha ,\beta )\in {\rm Con}(A)\times {\rm Con}(A)\ |\ \rho _A(\alpha )=\rho _A(\beta )\}$, then, by the above, $\equiv _A$ is a lattice congruence of ${\rm Con}(A)$ that preserves arbitrary joins and fulfills $[\alpha ,\beta ]_A\equiv _A\alpha \cap \beta $ for all $\alpha ,\beta \in {\rm Con}(A)$. By the above, if the commutator of $A$ equals the intersection, in particular if $\C $ is congruence--distributive, then $\rho _A(\theta )=\theta $ for all $\theta \in {\rm Con}(A)$, hence $\equiv _A=\Delta _{{\rm Con}(A)}$. Recall that $A$ is called a {\em semiprime algebra} iff $\Delta _A\in {\rm RCon}(A)$, that is iff $\rho _A(\Delta _A)=\Delta _A$. Therefore, if the commutator of $A$ equals the intersection, then $A$ is semiprime, and, if $\C $ is congruence--distributive, then all members of $\C $ are semiprime. Of course, $\theta \subseteq \rho _A(\theta )$ for all $\theta \in {\rm Con}(A)$, so $\rho _A(\theta )=\Delta _A$ implies $\theta =\Delta _A$, hence, if $A$ is semiprime, then $\Delta _A/\!\!\equiv _A=\{\Delta _A\}$. By the above, if $\C $ is congruence--modular and semi--degenerate, then $\nabla _A/\!\!\equiv _A=\{\nabla _A\}$.

\begin{remark} Assume that $A$ is semiprime and let $\alpha ,\beta \in {\rm Con}(A)$. Then $\rho _A([\alpha ,\beta ]_A)=\rho _A(\alpha \cap \beta )$, hence, by the above: $[\alpha ,\beta ]_A=\Delta _A$ iff $\alpha \cap \beta =\Delta _A$.\end{remark}

We will often use the remarks in this paper without referencing them.

By the properties of the commutator, the quotient bounded lattice, $({\rm Con}(A)/\!\!\equiv _A,\vee ,\wedge ,{\bf 0},{\bf 1})$, is a frame. We denote by $\lambda _A:{\rm Con}(A)\rightarrow {\rm Con}(A)/\!\!\equiv _A$ the canonical surjective lattice morphism. The intersection $\equiv _A\!\!\cap ({\cal K}(A))^2\in {\rm Eq}({\cal K}(A))$ will also be denoted $\equiv _A$; ${\cal L}(A)={\cal K}(A)/\!\!\equiv _A$ will be its quotient set and we will use the same notation for the canonical surjection: $\lambda _A:{\cal K}(A)\rightarrow {\cal L}(A)$.

Throughout the rest of this section, we shall assume that ${\cal K}(A)$ is closed w.r.t. the commutator of $A$. Then, by \cite[Proposition $9$]{retic}, ${\cal L}(A)$ is a bounded sublattice of ${\rm Con}(A)/\!\equiv _A$, thus it is a bounded distributive lattice. Note that, in the particular case when the commutator of $A$ coincides to the intersection, the fact that ${\cal K}(A)$ is closed w.r.t. the commutator means that ${\cal K}(A)$ is a sublattice of ${\rm Con}(A)$. So, if $\C $ is congruence--distributive, then: $\C $ has the CIP iff ${\cal K}(M)$ is a sublattice of ${\rm Con}(M)$ in each member $M$ of $\C $.

Note from the above that, for any $\theta \in {\rm Con}(A)$, we have: $\lambda _A(\theta )={\bf 1}$ iff $\theta =\nabla _A$.

Let $\theta \in {\rm Con}(A)$. Then we denote by $\theta ^*=\{\lambda _A(\alpha )\ |\ \alpha \in {\cal K}(A),\alpha \subseteq \theta \}$. Of course, ${\bf 0}=\lambda _A(\Delta _A)\in \theta ^*$. Let $\alpha ,\beta \in {\cal K}(A)$. Then clearly $\alpha \vee \beta \in {\cal K}(A)$, $\lambda _A(\alpha \vee \beta )=\lambda _A(\alpha )\vee \lambda _A(\beta )$ and, if $\alpha \subseteq \theta $ and $\beta \subseteq \theta $, then $\alpha \vee \beta \subseteq \theta $. Since ${\cal K}(A)$ is closed w.r.t. the commutator of $A$, we have $[\alpha ,\beta ]_A\in {\cal K}(A)$, and, if $\alpha \subseteq \theta $ and $\lambda _A(\beta )\leq \lambda _A(\alpha )$, then $[\alpha ,\beta ]_A\subseteq \alpha \subseteq \theta $ and $\lambda _A(\beta )=\lambda _A(\alpha )\wedge \lambda _A(\beta )=\lambda _A([\alpha ,\beta ]_A)$. Hence $\theta ^*\in {\rm Id}({\cal L}(A))$.

\begin{proposition}{\rm \cite[Proposition $10$, (ii)]{retic}} The map $\theta \mapsto \theta ^*$ from ${\rm Con}(A)$ to ${\rm Id}({\cal L}(A))$ is surjective.\label{mapstarsurj}\end{proposition}

\begin{proposition}{\rm \cite[Proposition $11$]{retic}} If $\theta \in {\rm Spec}(A)$, then $\theta ^*\in {\rm Spec}_{\rm Id}({\cal L}(A))$, and the map $\phi \mapsto \phi ^*$ is an order isomorphism from ${\rm Spec}(A)$ to ${\rm Spec}_{\rm Id}({\cal L}(A))$ and a homeomorphism w.r.t. the Stone topologies.\label{homeo}\end{proposition}

The previous proposition allows us to define:

\begin{definition} ${\cal L}(A)$ is called the {\em reticulation} of $A$.\end{definition}

By the above, if the commutator of $A$ equals the intersection, in particular if $\C $ is congruence--distributive, then $\lambda _A:{\rm Con}(A)\rightarrow {\rm Con}(A)/\equiv _A$ is a lattice isomorphism, ${\cal K}(A)$ is a bounded sublattice of ${\rm Con}(A)$ (recall that we are under the hypotheses that $[{\cal K}(A),{\cal K}(A)]_A\subseteq {\cal K}(A)$ and $\nabla _A\in {\cal K}(A)$) and $\lambda _A:{\cal K}(A)\rightarrow {\cal L}(A)$ is a lattice isomorphism, therefore we may take ${\cal L}(A)={\cal K}(A)$, hence, if, additionally, $A$ is finite, so that ${\cal K}(A)={\rm Con}(A)$, then we may take ${\cal L}(A)={\rm Con}(A)$.

\section{Functoriality of the Reticulation}
\label{fret}

Throughout this section, $B$ will be an arbitrary member of $\C $ and $f:A\rightarrow B$ shall be an arbitrary morphism in $\C $. We define $f^{\bullet }:{\rm Con}(A)\rightarrow {\rm Con}(B)$ by: $f^{\bullet }(\alpha )=Cg_B(f(\alpha ))$. Let us note that $f^{\bullet }$ and $f^*$ are order--preserving and, of course, so is the direct image of $f$. Notice, also, that, for all $\alpha \in {\rm Con}(A)$, $f(\alpha )\subseteq f^{\bullet }(\alpha )$, and, if $f$ is surjective and $\alpha \in [{\rm Ker}(f))$, then $f(\alpha )=f^{\bullet }(\alpha )$. Of course, $f^{\bullet }(\Delta _A)=\Delta _B$.

\begin{remark}$(i)\quad f^{\bullet }$ is the unique left adjoint of $f^*$, that is, for all $\alpha \in {\rm Con}(A)$ and all $\beta \in {\rm Con}(B)$: $f^{\bullet }(\alpha )\subseteq \beta \mbox{ iff }\alpha \subseteq f^*(\beta )$.

\noindent Indeed, for the direct implication, notice that $f(\alpha )\subseteq f^{\bullet }(\alpha )\subseteq \beta $ implies $\alpha \subseteq f^*(f(\alpha ))\subseteq f^*(\beta )$. For the converse, note that $\alpha \subseteq f^*(\beta )$ implies $f(\alpha )\subseteq f(f^*(\beta ))\subseteq \beta \in {\rm Con}(B)$, hence $f^{\bullet }(\alpha )=Cg_B(f(\alpha ))\subseteq \beta $. Therefore $f^{\bullet }$ is a left adjoint of $f^*$, and it is unique by the properties of adjoint pairs of morphisms between posets.

$(ii)\quad f^{\bullet }$ preserves arbitrary joins of congruences of $A$.

\noindent This follows from Lemma \ref{fcg}, but also from the properties of adjoint pairs of lattice morphisms between complete lattices and the fact that $f^*$ preserves arbitrary intersections, since it is the inverse image of $f^2$.

$(iii)\quad $If $C$ is a member of ${\cal V}$ and $g:B\rightarrow C$ is a morphism in ${\cal V}$, then $(g\circ f)^{\bullet }=g^{\bullet }\circ f^{\bullet }$.

\noindent It is immediate that $g^{\bullet }\circ f^{\bullet }$ is the unique left adjoint of $(g\circ f)^*=f^*\circ g^*$, so the equality above follows by $(i)$.\label{fbullet}\end{remark}

By Lemma \ref{fcg}, we may consider the restrictions: $f^{\bullet }\mid _{{\rm PCon}(A)}:{\rm PCon}(A)\rightarrow {\rm PCon}(B)$ and $f^{\bullet }\mid _{{\cal K}(A)}:{\cal K}(A)\rightarrow {\cal K}(B)$.

We recall the following definition from \cite{badvag}: $\C $ is called a {\em variety with $\vec{0}$ and $\vec{1}$} iff there exists an $N\in \N ^*$ and constants $0_1,\ldots ,0_N,1_1,\ldots ,1_N$ from $\tau $ such that, if we denote by $\vec{0}=(0_1,\ldots ,0_N)$ and $\vec{1}=(1_1,\ldots ,1_N)$, then $\C \vDash \vec{0}\approx \vec{1}\Rightarrow x\approx y$, that is, for any member $M$ of $\C $, if $0_i^M=1_i^M$ for all $i\in \overline{1,N}$, then $M$ is the trivial algebra. For instance, any variety of bounded ordered structures is a variety with $\vec{0}$ and $\vec{1}$, with $N=1$. Clearly, any variety with $\vec{0}$ and $\vec{1}$ is semi--degenerate.

\begin{remark} If $\C $ is a variety with $\vec{0}$ and $\vec{1}$ (with $N\in \N ^*$ as in the definition above), then, for all $i\in \overline{1,N}$, $(0_i^B,1_i^B)=(f(0_i^A),f(1_i^A))\in f(\nabla _A)\subseteq f^{\bullet }(\nabla _A)=Cg_B(f(\nabla _A))$, hence $B/f^{\bullet }(\nabla _A)\vDash \vec{0}\approx \vec{1}$, thus $f^{\bullet }(\nabla _A)=\nabla _B$.\end{remark}

\begin{remark} As shown in \cite{euadm}, $(f^*)^{-1}(\{\nabla _B\})=\{\nabla _A\}$, otherwise written $f^*(\theta )\neq \nabla _A$ for all $\theta \in {\rm Con}(B)\setminus \{\nabla _B\}$, holds if $\C $ is semi--degenerate, in particular it holds if $\C $ is a variety with $\vec{0}$ and $\vec{1}$.\end{remark}

Throughout the rest of this section, we shall assume that $[\cdot ,\cdot ]_A$ and $[\cdot ,\cdot ]_B$ are commutative and distributive w.r.t. arbitrary joins and that $\nabla _A\in {\cal K}(A)$ and $\nabla _B\in {\cal K}(B)$, all of which hold in the particular case when $\C $ is congruence--modular and semi--degenerate. We will also assume that ${\cal K}(A)$ and ${\cal K}(B)$ are closed w.r.t. the commutator.

\begin{proposition} There exists at most one function $\varphi :{\cal L}(A)\rightarrow {\cal L}(B)$ that closes the following diagram commutatively, and such a function preserves the ${\bf 0}$ and the join. Additionally:\begin{enumerate} 
\item\label{closediagr1} if $f$ is surjective or $\C $ is a variety with $\vec{0}$ and $\vec{1}$, then $\varphi $ preserves the ${\bf 1}$;
\item\label{closediagr2} if $f$ is surjective and $\C $ is congruence--modular, then $\varphi $ is a bounded lattice morphism.\end{enumerate}\vspace*{-10pt}\begin{center}
\begin{picture}(120,40)(0,0)
\put(0,30){${\cal K}(A)$}
\put(82,30){${\cal K}(B)$}
\put(0,0){${\cal L}(A)$}
\put(82,0){${\cal L}(B)$}
\put(25,33){\vector(1,0){55}}
\put(36,38){$f^{\bullet }\mid _{{\cal K}(A)}$}
\put(25,3){\vector(1,0){55}}
\put(30,6){$\varphi ={\cal L}(f)$}
\put(12,27){\vector(0,-1){18}}
\put(-1,17){$\lambda _A$}
\put(94,27){\vector(0,-1){18}}
\put(96,17){$\lambda _B$}
\end{picture}\end{center}\label{closediagr}\vspace*{-10pt}\end{proposition}

\begin{proof}Let $\alpha ,\beta \in {\cal K}(A)$. By the surjectivity of $\lambda _A$, if $\varphi $ exists, then it is uniquely defined by: $\varphi (\lambda _A(\theta ))=\lambda _B(f^{\bullet }(\theta ))$ for all $\theta \in {\cal K}(A)$. Assume that this function is well defined. Then $\varphi ({\bf 0})=\varphi (\lambda _A(\Delta _A))=\lambda _B(f^{\bullet }(\Delta _A))=\lambda _B(Cg_B(f(\Delta _A)))=\lambda _B(\Delta _B)={\bf 0}$ and $\varphi (\lambda _A(\alpha )\vee \lambda _A(\beta ))=\varphi (\lambda _A(\alpha \vee \beta ))=\lambda _B(f^{\bullet }(\alpha \vee \beta ))=\lambda _B(f^{\bullet }(\alpha )\vee f^{\bullet }(\beta ))=\lambda _B(f^{\bullet }(\alpha ))\vee \lambda _B(f^{\bullet }(\beta ))=\varphi (\lambda _A(\alpha ))\vee \varphi (\lambda _A(\beta ))$.

\noindent (\ref{closediagr1}) If $f$ is surjective or $\C $ is a variety with $\vec{0}$ and $\vec{1}$, then $\varphi ({\bf 1})=\varphi (\lambda _A(\nabla _A))=\lambda _B(f^{\bullet }(\nabla _A))=\lambda _B(f(\nabla _A))=\lambda _B(\nabla _B)={\bf 1}$.

\noindent (\ref{closediagr2}) If $f$ is surjective and $\C $ is congruence--modular, then, by Lemma \ref{fcg}, $\varphi (\lambda _A(\alpha )\wedge \lambda _A(\beta ))=\varphi (\lambda _A([\alpha ,\beta ]_A)=\lambda _B(f^{\bullet }([\alpha ,\beta ]_A))=\lambda _B(Cg_B(f([\alpha ,\beta ]_A)))=\lambda _B(f([\alpha ,\beta ]_A\vee {\rm Ker}(f)))=\lambda _B([f(\alpha \vee {\rm Ker}(f)),f(\beta \vee {\rm Ker}(f))]_B)=\lambda _B(f(\alpha \vee {\rm Ker}(f)))\wedge \lambda _B(f(\beta \vee {\rm Ker}(f)))=\lambda _B(Cg_B(f(\alpha )))\wedge \lambda _B(Cg_B(f(\beta )))=\lambda _B(f^{\bullet }(\alpha ))\wedge \lambda _B(f^{\bullet }(\beta ))=\varphi (\lambda _A(\alpha ))\wedge \varphi (\lambda _A(\beta ))$.\end{proof}

\begin{definition} We will say that $f$ satisfies the {\em functoriality of the reticulation} (abbreviated {\em FRet}) iff there exists a function that closes the diagram above commutatively, that is iff the function $\varphi $ in Proposition \ref{closediagr} is well defined.\end{definition}

If $f$ satisfies FRet, then we will denote by ${\cal L}(f)=\varphi $, that is: ${\cal L}(f):{\cal L}(A)\rightarrow {\cal L}(B)$, for all $\alpha \in {\cal K}(A)$, ${\cal L}(f)(\lambda _A(\alpha ))=\lambda _B(f^{\bullet }(\alpha ))$.

\begin{remark} Obviously, if $f$ is an isomorphism, then $f$ satisfies FRet and ${\cal L}(f)$ is a lattice isomorphism (in particular ${\cal L}(f)$ preserves the meet and the ${\bf 1}$), but the converse does not hold, as shown by the case of the morphism $l:Q\rightarrow P$ in Example \ref{mnex4}. Note that, in particular, $id_A^{\bullet }=id_{{\rm Con}(A)}$, thus ${\cal L}(id_A)=id_{{\cal L}(A)}$.\end{remark}

\begin{remark} As shown by the morphism $v:V\rightarrow V$ in Example \ref{tip20}, $f$ may fail FRet, while $f^{\bullet }$ preserves the meet and the commutator and $f^{\bullet }(\nabla _A)\equiv _B\nabla _B$.\end{remark}

\begin{lemma}\begin{itemize}
\item If the commutator of $A$ coincides to the intersection, then $f$ fulfills FRet.
\item In particular, if $\C $ is congruence--distributive and semi--degenerate and has the CIP, then all morphisms in $\C $ fulfill FRet.
\item If the commutators of $A$ and $B$ coincide to the intersection, in particular if $\C $ is congruence--distributive, then $f$ fulfills FRet and the following equivalences hold: ${\cal L}(f)$ preserves the meet iff $f^{\bullet }(\alpha \cap \beta )=f^{\bullet }(\alpha )\cap f^{\bullet }(\beta )$ for all $\alpha ,\beta \in {\cal K}(A)$, ${\cal L}(f)$ preserves the ${\bf 1}$ iff $f^{\bullet }(\nabla _A)=\nabla _B$, ${\cal L}(f)$ is injective or surjective iff $f^{\bullet }\mid _{{\cal K}(A)}:{\cal K}(A)\rightarrow {\cal K}(B)$ is injective or surjective, respectively.\end{itemize}\label{distribsuffret}\end{lemma}

\begin{proof} If the commutator of $A$ coincides to the intersection, then $\rho _A=id _{{\rm Con}(A)}$, so, for all $\alpha ,\beta \in {\rm Con}(A)$, $\lambda _A(\alpha )=\lambda _A(\beta )$ iff $\alpha =\beta $, thus, trivially, $f$ fulfills FRet.

If, additionally, the commutator of $B$ coincides to the intersection, then both $\lambda _A:{\cal K}(A)\rightarrow {\cal L}(A)$ and $\lambda _B:{\cal K}(B)\rightarrow {\cal L}(B)$ are lattice isomorphisms, so the equality ${\cal L}(f)\circ \lambda _A=\lambda _B\circ f^{\bullet }$ proves the equivalences in the enunciation. In fact, we may take ${\cal L}(A)={\cal K}(A)$ and ${\cal L}(B)={\cal K}(B)$, so that $\lambda _A$ and $\lambda _B$ become $id_{{\cal K}(A)}:{\cal K}(A)\rightarrow {\cal L}(A)$ and $id_{{\cal K}(B)}:{\cal K}(B)\rightarrow {\cal L}(B)$, respectively, and ${\cal L}(f)=f^{\bullet }$.\end{proof}

\begin{remark} If $f$ fulfills FRet and $f^{\bullet }:{\rm Con}(A)\rightarrow {\rm Con}(B)$ preserves the intersection, then, clearly, ${\cal L}(f)$ preserves the meet. As shown by Example \ref{tip20}, the converse does not hold.\end{remark}

\begin{proposition} Let $C$ be a member of $\C $ such that $[\cdot ,\cdot ]_C$ is commutative and distributive w.r.t. arbitrary joins, $\nabla _C\in {\cal K}(C)$ and ${\cal K}(C)$ is closed w.r.t. the commutator, and let $g:B\rightarrow C$ be a morphism. If $f$ and $g$ satisfy FRet, then $g\circ f$ satisfies FRet and ${\cal L}(g\circ f)={\cal L}(g)\circ {\cal L}(f)$. Also:\begin{itemize}
\item if, additionally, ${\cal L}(f)$ and ${\cal L}(g)$ preserve the ${\bf 1}$, then ${\cal L}(g\circ f)$ preserves the ${\bf 1}$;
\item if, additionally, ${\cal L}(f)$ and ${\cal L}(g)$ preserve the meet, then ${\cal L}(g\circ f)$ preserves the meet.\end{itemize}\label{compfret}\end{proposition}

\begin{proof}$\lambda _C\circ (g\circ f)^{\bullet }=\lambda _C\circ g^{\bullet }\circ f^{\bullet }={\cal L}(g)\circ \lambda _B\circ f^{\bullet }={\cal L}(g)\circ {\cal L}(f)\circ \lambda _A$, therefore $g\circ f$ satisfies FRet and, by the uniqueness stated in Proposition \ref{closediagr}, ${\cal L}(g\circ f)={\cal L}(g)\circ {\cal L}(f)$, hence the statements on the preservation of the ${\bf 1}$ and the meet.\end{proof}

By Propositions \ref{closediagr} and \ref{compfret}, if all morphisms in ${\cal C}$ satisfy FRet and are such that their images through the map ${\cal L}$ preserve the meet, so that these images are lattice morphisms, then ${\cal L}$ becomes a covariant functor from ${\cal C}$ to the variety of distributive lattices, and, if, additionally, these images preserve the ${\bf 1}$, then ${\cal L}$ is a functor from ${\cal C}$ to the variety of bounded distributive lattices. In either of these cases, we call ${\cal L}$ the {\em reticulation functor} for ${\cal C}$.

\begin{lemma}{\rm \cite{gulo},\cite{euadm}} If $\phi \in {\rm Con}(A)\setminus \{\nabla _A\}$, then the following are equivalent:\begin{enumerate}
\item\label{charspec0} $\phi \in {\rm Spec}(A)$;
\item\label{charspec1} for all $\alpha ,\beta \in {\rm PCon}(A)$, $[\alpha ,\beta ]_A\subseteq \phi $ implies $\alpha \subseteq \phi $ or $\beta \subseteq \phi $;
\item\label{charspec2} for all $\alpha ,\beta \in {\cal K}(A)$, $[\alpha ,\beta ]_A\subseteq \phi $ implies $\alpha \subseteq \phi $ or $\beta \subseteq \phi $.\end{enumerate}\label{charspec}\end{lemma}

\begin{lemma}For all $\alpha ,\beta \in {\rm Con}(A)$, $\rho _B(f^{\bullet }([\alpha ,\beta ]_A))\subseteq \rho _B([f^{\bullet }(\alpha ),f^{\bullet }(\beta )]_B)$.\label{fcomm}\end{lemma}

\begin{proof}Let $\psi \in {\rm Spec}(B)$ such that $[f^{\bullet }(\alpha ),f^{\bullet }(\beta )]_B\subseteq \psi $, so that $f^{\bullet }(\alpha )\subseteq \psi $ or $f^{\bullet }(\beta )\subseteq \psi $, so that $f^{\bullet }([\alpha ,\beta ]_A)\subseteq \psi $. Hence $V_B([f^{\bullet }(\alpha ),f^{\bullet }(\beta )]_B)\subseteq V_B(f^{\bullet }([\alpha ,\beta ]_A))$, therefore $\rho _B(f^{\bullet }([\alpha ,\beta ]_A))\subseteq \rho _B([f^{\bullet }(\alpha ),f^{\bullet }(\beta )]_B)$.\end{proof}

\begin{theorem} The following are equivalent:\begin{enumerate}
\item\label{admfret0} $f$ is admissible;
\item\label{admfret1} $f$ satisfies FRet and ${\cal L}(f)$ preserves the meet (so that ${\cal L}(f)$ is a lattice morphism);
\item\label{admfret2} for all $\alpha ,\beta \in {\cal K}(A)$, $\lambda _B(f^{\bullet }([\alpha ,\beta ]_A))=\lambda _B([f^{\bullet }(\alpha ),f^{\bullet }(\beta )]_B)$;
\item\label{admfret3} for all $\alpha ,\beta \in {\cal K}(A)$, $\rho _B(f^{\bullet }([\alpha ,\beta ]_A))=\rho _B([f^{\bullet }(\alpha ),f^{\bullet }(\beta )]_B)$;
\item\label{admfret4} for all $\alpha ,\beta \in {\cal K}(A)$, $\rho _B(f^{\bullet }([\alpha ,\beta ]_A))\supseteq [f^{\bullet }(\alpha ),f^{\bullet }(\beta )]_B$.
\end{enumerate}\label{admfret}\end{theorem}

\begin{proof}(\ref{admfret2})$\Leftrightarrow $(\ref{admfret3}): By the definition of $\equiv _B$.

\noindent (\ref{admfret3})$\Leftrightarrow $(\ref{admfret4}): By Lemma \ref{fcomm} and the fact that $\rho _B(f^{\bullet }([\alpha ,\beta ]_A))\supseteq [f^{\bullet }(\alpha ),f^{\bullet }(\beta )]_B$ iff $\rho _B(f^{\bullet }([\alpha ,\beta ]_A))\supseteq \rho _B([f^{\bullet }(\alpha ),$\linebreak $f^{\bullet }(\beta )]_B)$.

\noindent (\ref{admfret0})$\Rightarrow $(\ref{admfret2}): Let $\alpha ,\beta \in {\cal K}(A)$ and $\psi \in {\rm Spec}(B)$, so that $f^*(\psi )\in {\rm Spec}(A)$ since $f$ is admissible, thus, since $(f^{\bullet },f^*)$ is an adjoint pair: $f^{\bullet }([\alpha ,\beta ]_A)\subseteq \psi $ iff $[\alpha ,\beta ]_A\subseteq f^*(\psi )$ iff $\alpha \subseteq f^*(\psi )$ or $\beta \subseteq f^*(\psi )$ iff $f^{\bullet }(\alpha )\subseteq \psi $ or $f^{\bullet }(\beta )\subseteq \psi $ iff $[f^{\bullet }(\alpha ),f^{\bullet }(\beta )]_B\subseteq \psi $. Therefore $V_B(f^{\bullet }([\alpha ,\beta ]_A))=V_B([f^{\bullet }(\alpha ),f^{\bullet }(\beta )]_B)$, so $\rho _B(f^{\bullet }([\alpha ,\beta ]_A))=\rho _B([f^{\bullet }(\alpha ),f^{\bullet }(\beta )]_B)$, thus $\lambda _B(f^{\bullet }([\alpha ,\beta ]_A))=\lambda _B([f^{\bullet }(\alpha ),f^{\bullet }(\beta )]_B)$.

\noindent (\ref{admfret0}),(\ref{admfret2})$\Rightarrow $(\ref{admfret1}): Let $\alpha ,\beta \in {\cal K}(A)$ such that $\lambda _A(\alpha )=\lambda _A(\beta )$, so that $\rho _A(\alpha )=\rho _A(\beta )$, thus $V_A(\alpha )=V_A(\beta )$.

Let $\psi \in {\rm Spec}(B)$, so that $f^*(\psi )\in {\rm Spec}(A)$ since $f$ is admissible, thus, by the above and the fact that $(f^{\bullet },f^*)$ is an adjoint pair: $f^{\bullet }(\alpha )\subseteq \psi $ iff $\alpha \subseteq f^*(\psi )$ iff $\beta \subseteq f^*(\psi )$ iff $f^{\bullet }(\beta )\subseteq \psi $, therefore $V_B(f^{\bullet }(\alpha ))=V_B(f^{\bullet }(\beta ))$, so that $\rho _B(f^{\bullet }(\alpha ))=\rho _B(f^{\bullet }(\beta ))$, thus ${\cal L}(f)(\lambda _A(\alpha ))=\lambda _B(f^{\bullet }(\alpha ))=\lambda _B(f^{\bullet }(\beta ))={\cal L}(f)(\lambda _A(\beta ))$, hence ${\cal L}(f)$ is well defined, that is $f$ fulfills FRet.

Now let $\gamma ,\delta \in {\cal K}(A)$, arbitrary. Then ${\cal L}(f)(\lambda _A(\gamma )\wedge \lambda _A(\delta  ))={\cal L}(f)(\lambda _A([\gamma ,\delta ]_A))=\lambda _B(f^{\bullet }([\gamma ,\delta ]_A))=\lambda _B([f^{\bullet }(\gamma ),f^{\bullet }(\delta )]_B)=\lambda _B(f^{\bullet }(\gamma ))\wedge \lambda _B(f^{\bullet }(\delta ))={\cal L}(f)(\lambda _A(\gamma ))\wedge {\cal L}(f)(\lambda _A(\delta ))$.

\noindent (\ref{admfret1})$\Rightarrow $(\ref{admfret2}): Let $\alpha ,\beta \in {\cal K}(A)$, so that $[\alpha ,\beta ]_A\in {\cal K}(A)$ and $\lambda _B(f^{\bullet }([\alpha ,\beta ]_A))={\cal L}(f)(\lambda _A([\alpha ,\beta ]_A))={\cal L}(f)(\lambda _A(\alpha )\wedge \lambda _A(\beta ))={\cal L}(f)(\lambda _A(\alpha ))\wedge {\cal L}(f)(\lambda _A(\beta ))=\lambda _B(f^{\bullet }(\alpha ))\wedge \lambda _B(f^{\bullet }(\beta ))=\lambda _B([f^{\bullet }(\alpha ),f^{\bullet }(\beta )]_B)$.

\noindent (\ref{admfret2})$\Rightarrow $(\ref{admfret0}): Let $\alpha ,\beta \in {\cal K}(A)$ and $\psi \in {\rm Spec}(B)$. Then $\lambda _B(f^{\bullet }([\alpha ,\beta ]_A))=\lambda _B([f^{\bullet }(\alpha ),f^{\bullet }(\beta )]_B)$, thus $\rho _B(f^{\bullet }([\alpha ,\beta ]_A))$\linebreak $=\rho _B([f^{\bullet }(\alpha ),f^{\bullet }(\beta )]_B)$, so that $V_B(f^{\bullet }([\alpha ,\beta ]_A))=V_B([f^{\bullet }(\alpha ),f^{\bullet }(\beta )]_B)$, therefore, since $(f^{\bullet },f^*)$ is an adjoint pair: $[\alpha ,\beta ]_A\subseteq f^*(\psi )$ iff $f^{\bullet }([\alpha ,\beta ]_A)\subseteq \psi $ iff $[f^{\bullet }(\alpha ),f^{\bullet }(\beta )]_B\subseteq \psi $ iff $f^{\bullet }(\alpha )\subseteq \psi $ or $f^{\bullet }(\beta )\subseteq \psi $ iff $\alpha \subseteq f^*(\psi )$ or $\beta \subseteq f^*(\psi )$. By Lemma \ref{charspec}, it follows that $f^*(\psi )\in {\rm Spec}(A)$, hence $f$ is admissible.\end{proof}

\begin{corollary} If $f^{\bullet }([\alpha ,\beta ]_A)=[f^{\bullet }(\alpha ),f^{\bullet }(\beta )]_B$ for all $\alpha ,\beta \in {\cal K}(A)$, then $f$ satisfies FRet and ${\cal L}(f)$ is a lattice morphism. The converse does not hold.\label{suffret}\end{corollary}

\begin{proof} By Theorem \ref{admfret}, the direct implication holds. Example \ref{tip20} disproves the converse.\end{proof}

\begin{lemma}{\rm \cite[Corollary $7.4$]{gulo}} If $\C $ is congruence--distributive and has the CIP, in particular if $\C $ is congruence--distributive and has the PIP, then every morphism in $\C $ is admissible.\label{cipalladm}\end{lemma}

\begin{proposition} If $\C $ is congruence--distributive and has the CIP, in particular if $\C $ is congruence--distributive and has the PIP, then $f$ fulfills FRet and $f^{\bullet }:{\cal K}(A)\rightarrow {\cal K}(B)$ and ${\cal L}(f):{\cal L}(A)\rightarrow {\cal L}(B)$ are lattice morphisms, so that, if $\C $ is also semi--degenerate, then ${\cal L}$ is a functor from $\C $ to the variety of distributive lattices.

If, moreover, $\C $ is a congruence--distributive variety with $\vec{0}$ and $\vec{1}$ and the CIP, then ${\cal L}$ is a functor from $\C $ to the variety of bounded distributive lattices.\label{cipnedpc}\end{proposition}

\begin{proof} By Lemma \ref{cipalladm} and Theorem \ref{admfret}, $f$ fulfills FRet and ${\cal L}(f):{\cal L}(A)\rightarrow {\cal L}(B)$ is a lattice morphism, so that $f^{\bullet }:{\cal K}(A)\rightarrow {\cal K}(B)$ is a lattice morphism since, in this particular case, ${\cal K}(A)$ and ${\cal K}(B)$ are sublattices of ${\rm Con}(A)$ and ${\rm Con}(B)$, respectively, and $\lambda _A:{\cal K}(A)\rightarrow {\cal L}(A)$ and $\lambda _B:{\cal K}(B)\rightarrow {\cal L}(B)$ are lattice isomorphisms.\end{proof}

\begin{remark} If $f$ satisfies FRet and $f^{\bullet }\mid 
_{{\cal K}(A)}:{\cal K}(A)\rightarrow {\cal K}(B)$ is surjective, then, by the surjectivity of $\lambda _B:{\cal K}(B)\rightarrow {\cal L}(B)$, it follows that ${\cal L}(f)\circ \lambda _A=\lambda _B\circ f^{\bullet }$ is surjective, hence ${\cal L}(f):{\cal L}(A)\rightarrow {\cal L}(B)$ is surjective.\end{remark}

\begin{lemma}\begin{enumerate}
\item\label{surjfret} If $f$ is surjective, then $f$ satisfies FRet and ${\cal L}(f)$ is a bounded lattice morphism.
\item\label{surjcg} If $f$ is surjective, then $f^{\bullet }:{\rm Con}(A)\rightarrow {\rm Con}(B)$, $f^{\bullet }\mid 
_{{\cal K}(A)}:{\cal K}(A)\rightarrow {\cal K}(B)$ and $f^{\bullet }\mid _{{\rm PCon}(A)}:{\rm PCon}(A)\rightarrow {\rm PCon}(B)$ are surjective.
\item\label{cgsurjk1} If $f^{\bullet }:{\rm Con}(A)\rightarrow {\rm Con}(B)$ is surjective, then $f^{\bullet }\mid 
_{{\cal K}(A)}:{\cal K}(A)\rightarrow {\cal K}(B)$ is surjective, so, if, additionally, $f$ satisfies FRet, then ${\cal L}(f):{\cal L}(A)\rightarrow {\cal L}(B)$ is surjective.
\item\label{cgsurjk2} If ${\cal C}$ is congruence--distributive and $f^{\bullet }:{\rm Con}(A)\rightarrow {\rm Con}(B)$ is surjective, then $f$ satisfies FRet and ${\cal L}(f):{\cal L}(A)\rightarrow {\cal L}(B)$ is surjective.\end{enumerate}\label{cgsurjk}\end{lemma}

\begin{proof} (\ref{surjfret}) By Proposition \ref{closediagr}, (\ref{closediagr1}), Theorem \ref{admfret} and the fact that all surjective morphisms are admissible.

\noindent (\ref{surjcg}) By Lemma \ref{fcg}, for all $a,b\in A$ and any $\beta \in {\rm Con}(B)$, we have $f^{\bullet }(Cg_A(a,b))=Cg_B(f(a),f(b))$ and $\displaystyle \beta =\bigvee _{(x,y)\in \beta }Cg_B(x,y)$, which, along with the fact that $f^{\bullet }$ preserves arbitrary joins and the surjectivity of $f$, proves that $f^{\bullet }({\rm Con}(A))={\rm Con}(B)$, $f^{\bullet }({\cal K}(A))={\cal K}(B)$ and $f^{\bullet }({\rm PCon}(A))={\rm PCon}(B)$.

\noindent (\ref{cgsurjk1}) Let $\beta \in {\cal K}(B)$. Since $f^{\bullet }:{\rm Con}(A)\rightarrow {\rm Con}(B)$ is surjective, it follows that there exists an $\alpha \in {\rm Con}(A)$ such that $\displaystyle \beta =f^{\bullet }(\alpha )=f^{\bullet }(\bigvee _{(a,b)\in \alpha }Cg_A(a,b))=\bigvee _{(a,b)\in \alpha }f^{\bullet }(Cg_A(a,b))$, hence, for some $n\in \N ^*$ and some $(a_1,b_1),\ldots ,(a_n,b_n)\in \alpha $, $\displaystyle \beta =\bigvee _{i=1}^nf^{\bullet }(Cg_A(a_i,b_i))=f^{\bullet }(Cg_A(\{(a_1,b_1),\ldots ,(a_n,b_n)\}))\in f^{\bullet }({\cal K}(A))$. Therefore $f^{\bullet }\mid 
_{{\cal K}(A)}:{\cal K}(A)\rightarrow {\cal K}(B)$ is surjective.

\noindent (\ref{cgsurjk2}) By (\ref{cgsurjk1}) and Lemma \ref{distribsuffret}.\end{proof}

\begin{remark} By Lemma \ref{cgsurjk}, (\ref{surjcg}), if $f$ is surjective, then, if ${\cal K}(A)={\rm Con}(A)$ or ${\rm PCon}(A)={\rm Con}(A)$ or ${\rm PCon}(A)={\cal K}(A)$ or $A$ is simple, then ${\cal K}(B)={\rm Con}(B)$ or ${\rm PCon}(B)={\rm Con}(B)$ or ${\rm PCon}(B)={\cal K}(B)$ or $B$ is simple, respectively.

Indeed, if ${\cal K}(A)={\rm Con}(A)$, then ${\cal K}(B)=f({\cal K}(A))=f({\rm Con}(A))={\rm Con}(B)$, and analogously for the next two statements. The fact that $f^{\bullet }(\Delta _A)=\Delta _B$ and, since $f$ is surjective, $f^{\bullet }(\nabla _A)=\nabla _B$, gives us the last statement.\end{remark}

\begin{remark} Recall that a complete lattice has all elements compact iff it satisfies the Ascending Chain Condition (ACC). Thus ${\cal K}(A)={\rm Con}(A)$ iff ${\rm Cp}({\rm Con}(A))={\rm Con}(A)$ iff ${\rm Con}(A)$ satisfies the Ascending Chain Condition, which holds, in particular, if ${\rm Con}(A)$ has finite height, in particular if ${\rm Con}(A)$ is finite, for instance if $A$ is finite or simple.

If the commutator of $A$ equals the intersection, in particular if $\C $ is congruence--distributive, then ${\cal K}(A)={\rm Cp}({\rm Con}(A))$ is a sublattice of ${\rm Con}(A)$ with all elements compact and ${\cal L}(A)\cong {\cal K}(A)$, thus ${\cal L}(A)={\rm Cp}({\cal L}(A))$, i.e. ${\cal L}(A)$ has all elements compact, that is ${\cal L}(A)$ satisfies the ACC, according to the above.\label{reticacc}\end{remark}

\begin{proposition} ${\cal L}$ preserves surjectivity; more precisely, if $f$ is surjective, then $f$ fulfills FRet and ${\cal L}(f):{\cal L}(A)\rightarrow {\cal L}(B)$ is a surjective lattice morphism.\end{proposition}

\begin{proof} By Lemma \ref{cgsurjk}, (\ref{surjfret}), (\ref{surjcg}) and (\ref{cgsurjk1}).\end{proof}

\begin{remark} If the commutator of $A$ equals the intersection, ${\rm Con}(A)$ is a chain and $(f^*)^{-1}(\{\nabla _B\})=\{\nabla _A\}$, then $f$ satisfies FRet and ${\cal L}(f)$ is a lattice morphism.

Indeed, this follows from Theorem \ref{admfret} and the fact that, in this case, $f$ is admissible, since ${\rm Spec}(A)={\rm Con}(A)\setminus \{\nabla _A\}$. See also Lemma \ref{distribsuffret} and \cite{euadm}.\label{corfret}\end{remark}

Let $I$ be a non--empty set and, for each $i\in I$, $p_i$ and $q_i$ be terms over $\tau $ of arity $4$.

Recall that $(p_i,q_i)_{i\in I}$ is a {\em system of congruence intersection terms for $\C $} iff, for any member $M$ of $\C $ and any $a,b,c,d\in M$, $\displaystyle Cg_M(a,b)\cap Cg_M(c,d)=\bigvee _{i\in I}Cg_M(p_i^M(a,b,c,d),q_i^M(a,b,c,d))$ {\rm \cite{agl}}.

By analogy to the previous definition, let us introduce:

\begin{definition} $(p_i,q_i)_{i\in I}$ is a {\em system of congruence commutator terms for $\C $} iff, for any member $M$ of $\C $ and any $a,b,c,d\in M$, $\displaystyle [Cg_M(a,b),Cg_M(c,d)]_M=\bigvee _{i\in I}Cg_M(p_i^M(a,b,c,d),q_i^M(a,b,c,d))$.\end{definition}

\begin{remark} Clearly, if $\C $ is congruence--distributive and admits a finite system of congruence intersection terms, then, in each member $M$ of $\C $, ${\cal K}(M)$ is closed w.r.t. the intersection.

More generally, if $\C $ admits a finite system of congruence commutator terms, then, in each member $M$ of $\C $, ${\cal K}(M)$ is closed w.r.t. the commutator.\label{cgcomm}\end{remark}

\begin{proposition} If $\C $ admits a system of congruence commutator terms, then $f^{\bullet }([\alpha ,\beta ]_A)=[f^{\bullet }(\alpha ),f^{\bullet }(\beta )]_B$ for all $\alpha ,\beta \in {\rm Con}(A)$, in particular $f$ fulfills FRet and ${\cal L}(f)$ is a lattice morphism.\end{proposition}

\begin{proof}Let $(p_i,q_i)_{i\in I}$ be a system of congruence commutator terms for $\C $.

We first prove that $f^{\bullet }$ preserves the commutator applied to principal congruences. Let $a,b,c,d\in A$. Then, since $f^{\bullet }$ preserves arbitrary joins: $\displaystyle f^{\bullet }([Cg_A(a,b),Cg_A(c,d)]_A)=f^{\bullet }(\bigvee _{i\in I}Cg_A(p_i^A(a,b,c,d),q_i^A(a,b,c,d)))=\bigvee _{i\in I}f^{\bullet }(Cg_A(p_i^A(a,b,c,d),q_i^A(a,b,c,d)))=\bigvee _{i\in I}Cg_B(f(p_i^A(a,b,c,d)),f(q_i^A(a,b,c,d)))=\bigvee _{i\in I}Cg_B(p_i^B(f(a),f(b),f(c),$\linebreak $f(d))),q_i^B(f(a),f(b),f(c),f(d))))=[Cg_B(f(a),f(b)),Cg_B(f(c),f(d))]_B=[f^{\bullet }(Cg_A(a,b)),f^{\bullet }(Cg_A(c,d))]_B$.

Now let $\alpha ,\beta \in {\rm Con}(A)$. Then $\displaystyle \alpha =\bigvee _{j\in J}\alpha _j$ and $\displaystyle \beta =\bigvee _{k\in K}\beta _k$ for some non--empty families $(\alpha _j)_{j\in J}\subseteq {\rm PCon}(A)$ and $(\beta _k)_{k\in K}\subseteq {\rm PCon}(A)$. From the above and the fact that $f^{\bullet }$ preserves arbitrary joins, we obtain: $\displaystyle f^{\bullet }([\alpha ,\beta ]_A)=f^{\bullet }([\bigvee _{j\in J}\alpha _j,\bigvee _{k\in K}\beta _k]_A)=f^{\bullet }(\bigvee _{j\in J}\bigvee _{k\in K}[\alpha _j,\beta _k]_A)=\bigvee _{j\in J}\bigvee _{k\in K}f^{\bullet }([\alpha _j,\beta _k]_A)=\bigvee _{j\in J}\bigvee _{k\in K}[f^{\bullet }(\alpha _j),$\linebreak $\displaystyle f^{\bullet }(\beta _k)]_B=[\bigvee _{j\in J}f^{\bullet }(\alpha _j),\bigvee _{k\in K}f^{\bullet }(\beta _k)]_B=[f^{\bullet }(\bigvee _{j\in J}\alpha _j),f^{\bullet }(\bigvee _{k\in K}\beta _k)]_B=[f^{\bullet }(\alpha ),f^{\bullet }(\beta )]_B$.

Apply Theorem \ref{admfret} for the last statement.\end{proof}

In view of Remark \ref{cgcomm}, we obtain:

\begin{corollary}\begin{itemize}
\item If $\C $ is semi--degenerate and admits a system of congruence commutator terms, then ${\cal L}$ is a functor from $\C $ to the variety of distributive lattices.
\item If $\C $ is a variety with $\vec{0}$ and $\vec{1}$ that admits a system of congruence commutator terms, then ${\cal L}$ is a functor from $\C $ to the variety of bounded distributive lattices.\end{itemize}\end{corollary}

\begin{corollary}\begin{itemize}
\item If $\C $ is semi--degenerate and congruence--distributive and admits a system of congruence intersection terms, then ${\cal L}$ is a functor from $\C $ to the variety of distributive lattices.
\item If $\C $ is a congruence--distributive variety with $\vec{0}$ and $\vec{1}$ that admits a system of congruence intersection terms, then ${\cal L}$ is a functor from $\C $ to the variety of bounded distributive lattices.\end{itemize}\end{corollary}

Recall that a join--semilattice with smallest element $(L,\vee ,0)$ is said to be {\em dually Brouwerian} iff there exists a binary operation $\dot{-}$ on $L$, called {\em dual relative pseudocomplementation}, such that, for all $a,b,c\in L$, $a\dot{-}b\leq c$ iff $a\leq b\vee c$. In particular, in a dually Brouwerian join--semilattice $(L,\vee ,0)$, we have, for all $a,b\in L$: $a\dot{-}b=0$ iff $a\leq b$.

Following \cite{bj}, we say that $\C $ has {\em equationally definable principal congruences} (abbreviated {\em EDPC}) iff there exist an $n\in \N ^*$ and terms $p_1,\ldots ,p_n,q_1,\ldots ,q_n$ of arity $4$ over $\tau $ such that, for all members $M$ of $\C $ and all $a,b\in M$, $Cg_M(a,b)=\{(c,d)\in M^2\ |\ (\forall \, i\in \overline{1,n})\, (p_i^M(a,b,c,d)=q_i^M(a,b,c,d))\}$.

\begin{theorem}\textup{\cite{blkpgz},\cite{kp}}\begin{enumerate}
\item\label{edpc1} If $\C $ has EDPC, then $\C $ is congruence--distributive.
\item\label{edpc2} $\C $ has EDPC if and only if, for any member $M$ of $\C $, the semilattice $({\cal K}(M),\vee ,\Delta _M)$ is dually Browerian. In this case, if $n\in \N ^*$ and $p_1,q_1,\ldots ,p_n,q_n$ are as above, then, for any member $M$ of $\C $, the operation $\dot{-}$ of the dually Brouwerian semilattice ${\cal K}(M)$ is defined on ${\rm PCon}(M)$ by: $\displaystyle Cg_M(c,d)\dot{-}Cg_M(a,b)=\bigvee _{i=1}^nCg_M(p_i^M(a,b,c,d),q_i^M(a,b,c,d))$ for any $a,b,c,d\in M$.\end{enumerate}\label{edpc}\end{theorem}

\begin{lemma} If $\C $ has EDPC, then, for all $\alpha ,\beta \in {\rm PCon}(A)$, $f^{\bullet }(\alpha \dot{-}\beta )=f^{\bullet }(\alpha )\dot{-}f^{\bullet }(\beta )$.\label{prezdifpcg}\end{lemma}

\begin{proof}Let $n\in \N ^*$ and $p_1,q_1,\ldots ,p_n,q_n$ be as in Theorem \ref{edpc}, and $a,b,c,d\in A$. Then, by Theorem \ref{edpc} and Lemma \ref{fcg}, $\displaystyle f^{\bullet }(Cg_A(c,d)\dot{-}Cg_A(a,b))=f^{\bullet }(\bigvee _{i=1}^nCg_A(p_i^A(a,b,c,d),q_i^A(a,b,c,d)))=\bigvee _{i=1}^nf^{\bullet }(Cg_A(p_i^A(a,b,c,d),$\linebreak $\displaystyle q_i^A(a,b,c,d)))=\bigvee _{i=1}^nCg_B(f(p_i^A(a,b,c,d)),f(q_i^A(a,b,c,d)))=\bigvee _{i=1}^nCg_B(p_i^B(f(a),f(b),f(c),f(d)),q_i^B(f(a),f(b),$\linebreak $f(c),f(d)))=Cg_B(f(c),f(d))\dot{-}Cg_B(f(a),f(b))=f^{\bullet }(Cg_A(c,d))\dot{-}f^{\bullet }(Cg_A(a,b))$.\end{proof}

\begin{remark}{\rm \cite{ezeh}} If $\C $ has EDPC, then, for all $\alpha ,\beta ,\gamma \in {\cal K}(A)$:\begin{itemize}
\item $(\alpha \vee \beta )\dot{-}\gamma =(\alpha \dot{-}\gamma )\vee (\beta \dot{-}\gamma )$;
\item $\alpha \dot{-}(\beta \vee \gamma )=(\alpha \dot{-}\beta )\dot{-}\gamma $.\end{itemize}\end{remark}

\begin{proposition} If $\C $ has EDPC, then, for all $\alpha ,\beta \in {\cal K}(A)$, $f^{\bullet }(\alpha \dot{-}\beta )=f^{\bullet }(\alpha )\dot{-}f^{\bullet }(\beta )$.\label{fprezdif}\end{proposition}

\begin{proof}Let $\theta \in {\rm PCon}(A)$ and $\alpha \in {\cal K}(A)$, so that $\displaystyle \alpha =\bigvee _{i=1}^r\alpha _i$ for some $r\in \N ^*$ and some $\alpha _1,\ldots ,\alpha _r\in {\rm PCon}(A)$. Then, by Lemma \ref{prezdifpcg}, $\displaystyle f^{\bullet }(\alpha \dot{-}\theta )=f^{\bullet }((\bigvee _{i=1}^r\alpha _i)\dot{-}\theta )=f^{\bullet }(\bigvee _{i=1}^r(\alpha _i\dot{-}\theta ))=\bigvee _{i=1}^rf^{\bullet }(\alpha _i\dot{-}\theta )=\bigvee _{i=1}^r(f^{\bullet }(\alpha _i)\dot{-}f^{\bullet }(\theta ))=(\bigvee _{i=1}^rf^{\bullet }(\alpha _i))\dot{-}f^{\bullet }(\theta )=f^{\bullet }(\bigvee _{i=1}^r\alpha _i)\dot{-}f^{\bullet }(\theta )=f^{\bullet }(\alpha )\dot{-}f^{\bullet }(\theta )$.

Now let $\beta \in {\cal K}(A)$, so that $\displaystyle \beta =\bigvee _{j=1}^s\beta _j$ for some $s\in \N ^*$ and some $\beta _1,\ldots ,\beta _s\in {\rm PCon}(A)$. We apply induction on $t\in \overline{1,s}$. By the above, $f^{\bullet }(\alpha \dot{-}\beta _1)=f^{\bullet }(\alpha )\dot{-}f^{\bullet }(\beta _1)$. Now assume that, for some $t\in \overline{1,s-1}$, $\displaystyle f^{\bullet }(\alpha \dot{-}(\bigvee _{j=1}^t\beta _j))=f^{\bullet }(\alpha )\dot{-}f^{\bullet }(\bigvee _{j=1}^t\beta _j)$. Then, since $\displaystyle \alpha \dot{-}(\bigvee _{j=1}^t\beta _j)\in {\cal K}(A)$, $\displaystyle f^{\bullet }(\alpha \dot{-}(\bigvee _{j=1}^{t+1}\beta _j))=f^{\bullet }((\alpha \dot{-}(\bigvee _{j=1}^{t}\beta _j))\dot{-}\beta _{t+1})=f^{\bullet }(\alpha \dot{-}(\bigvee _{j=1}^{t}\beta _j))\dot{-}$\linebreak $f^{\bullet }(\beta _{t+1})=(f^{\bullet }(\alpha )\dot{-}f^{\bullet }(\bigvee _{j=1}^{t}\beta _j))\dot{-}f^{\bullet }(\beta _{t+1})=f^{\bullet }(\alpha )\dot{-}(f^{\bullet }(\bigvee _{j=1}^{t}\beta _j)\vee f^{\bullet }(\beta _{t+1}))=f^{\bullet }(\alpha )\dot{-}f^{\bullet }(\bigvee _{j=1}^{t+1}\beta _j)$. Thus $\displaystyle f^{\bullet }(\alpha \dot{-}\beta )=f^{\bullet }(\alpha \dot{-}(\bigvee _{j=1}^s\beta _j))=f^{\bullet }(\alpha )\dot{-}f^{\bullet }(\bigvee _{j=1}^s\beta _j)=f^{\bullet }(\alpha )\dot{-}f^{\bullet }(\beta )$.\end{proof}

Let $L$ and $M$ be dually Brouwerian join--semilattices. We call $h:L\rightarrow M$ a {\em dually Brouwerian join--semilattice morphism} iff $h$ preserves the $0$, the join and the dual relative pseudocomplementation; if $L$ and $M$ are lattices and $h$ also preserves the meet, then we call $h$ a {\em dually Brouwerian lattice morphism}. Note that, if $L$ is a lattice, then $L$ is distributive, as one can easily derive from \cite[Lemma $4.4$]{bj}.

\begin{corollary} If ${\cal C}$ has EDPC, then ${\cal L}(f)=f^{\bullet }:{\cal L}(A)={\cal K}(A)\rightarrow {\cal L}(B)={\cal K}(B)$ is a dually Brouwerian join--semilattice morphism.\end{corollary}

\begin{proof} By Remark \ref{fbullet}, Proposition \ref{fprezdif} and Theorem \ref{edpc}, (\ref{edpc1}).\end{proof}

\begin{remark} If $\C $ is a discriminator variety, then, by \cite[Theorem $5.5$]{bj}, ${\rm PCon}(A)={\cal K}(A)\cong {\cal L}(A)$ is a relatively complemented sublattice of ${\rm Con}(A)$; we set ${\cal K}(A)={\cal L}(A)$, and the same for $B$. From \cite[Lemma $5.3$]{bj} it follows that ${\cal L}(f)=f^{\bullet }\mid _{{\rm PCon}(A)}:{\rm PCon}(A)\rightarrow {\rm PCon}(B)$ is a relatively complemented lattice morphism.\end{remark}

\begin{remark} ${\cal L}$ reflects neither injectivity, nor surjectivity, as shown by the case of the morphism $l:Q\rightarrow P$ from Example \ref{mnex4}. ${\cal L}$ does not preserve injectivity and does not reflect surjectivity even for congruence--distributive varieties, as shown by the case of the morphism $i_{{\cal L}_2^2,{\cal M}_3}:{\cal L}_2^2\rightarrow {\cal M}_3$ from Example \ref{pentagon}.

If the commutators of $A$ and $B$ coincide to the intersection, ${\cal K}(A)={\rm Con}(A)$ and $f$ is surjective, then $f^{\bullet }:{\rm Con}(A)\rightarrow {\rm Con}(B)$ is surjective, thus ${\cal K}(B)={\rm Con}(B)$ and $f^{\bullet }:{\cal K}(A)\rightarrow {\cal K}(B)$ is surjective, hence ${\cal L}(f):{\cal L}(A)\rightarrow {\cal L}(B)$ is surjective. In particular, in congruence--distributive varieties, the functor ${\cal L}$ preserves the surjectivity of morphisms defined on finite algebras.\end{remark}

\begin{remark} If $f$ is injective, then, for all $\theta \in {\rm Con}(A)$, we have: $f^{\bullet }(\theta )=\Delta _B$ iff $\theta =\Delta _A$. Indeed, $f(\Delta _A)\subseteq \Delta _B$, so $f^{\bullet }(\Delta _A)=\Delta _B$, while, since $f(\theta )\subseteq f^{\bullet }(\theta )$, $f^{\bullet }(\theta )=\Delta _B$ implies $f(\theta )\subseteq \Delta _B$, which implies $\theta =\Delta _A$ if $f$ is injective.\end{remark}

\begin{proposition} If $\C $ is semi--degenerate and has EDPC and the CIP, then ${\cal L}$ is a functor from $\C $ to the variety of distributive lattices which preserves injectivity.\label{edpclinj}\end{proposition}

\begin{proof} Assume that $\C $ has EDPC and the CIP, so that every morphism in $\C $ satisfies FRet and ${\cal L}$ is a functor from $\C $ to the variety of distributive lattices by Theorem \ref{edpc}, (\ref{edpc1}), and Proposition \ref{cipnedpc}, and also assume that $f$ is injective. Let $\alpha ,\beta \in {\cal K}(A)$. Then, by Theorem \ref{edpc}, (\ref{edpc2}), Proposition \ref{fprezdif} and the injectivity of $f$: $f^{\bullet }(\alpha )\subseteq f^{\bullet }(\beta )$ iff $f^{\bullet }(\alpha )\dot{-}f^{\bullet }(\beta )=\Delta _B$ iff $f^{\bullet }(\alpha \dot{-}\beta )=\Delta _B$ iff $\alpha \dot{-}\beta =\Delta _A$ iff $\alpha \subseteq \beta $. Hence: $f^{\bullet }(\alpha )=f^{\bullet }(\beta )$ iff $\alpha =\beta $, therefore $f^{\bullet }$ is injective, thus so is ${\cal L}(f):{\cal L}(A)\rightarrow {\cal L}(B)$, since $\C $ is congruence--distributive.\end{proof}

\begin{remark} Assume that $f$ is injective and the canonical embedding of $f(A)$ into $B$ satisfies the Congruence Extension Property. Then, for $\alpha \in {\rm Con}(A)$, $f^{\bullet }(\alpha )\cap f(A)^2=f(\alpha )$, hence the map $f^{\bullet }:{\rm Con}(A)\rightarrow {\rm Con}(B)$ is injective, thus so are its restrictions $f^{\bullet }\mid _{{\cal K}(A)}:{\cal K}(A)\rightarrow {\cal K}(B)$ and $f^{\bullet }\mid _{{\rm PCon}(A)}:{\rm PCon}(A)\rightarrow {\rm PCon}(B)$.

Thus, if, additionally, the commutators of $A$ and $B$ coincide to the intersection, so that ${\cal K}(A)$ and ${\cal K}(B)$ are sublattices of ${\rm Con}(A)$ and ${\rm Con}(B)$, respectively, $\lambda _A:{\cal K}(A)\rightarrow {\cal L}(A)$ and $\lambda _B:{\cal K}(B)\rightarrow {\cal L}(B)$ are lattice isomorphisms and, as noted in Lemma \ref{distribsuffret}, $f$ satisfies FRet, it follows that ${\cal L}(f)$ is injective.\end{remark}

Therefore, in view of Proposition \ref{cipnedpc}, we have:

\begin{proposition} If $\C $ is semi--degenerate, congruence--distributive and congruence--extensible and it has the CIP, then ${\cal L}$ is a functor from $\C $ to the variety of distributive lattices which preserves injectivity.\end{proposition}

In what follows we apply the functoriality of the reticulation to the study of properties Going Up, Going Down and Lying Over in algebras whose semilattices of compact congruences and commutators are as above.

\begin{definition} We say that $f$ fulfills property {\em Going Up} (abbreviated {\em GU}) if and only if, for any $\phi ,\psi \in \textup{Spec}(A)$ and any $\phi _1\in \textup{Spec}(B)$ such that $\phi \subseteq \psi $ and $f^*(\phi _1)=\phi $, there exists a $\psi _1\in \textup{Spec}(B)$ such that $\phi _1\subseteq \psi _1$ and $f^*(\psi _1)=\psi $.

We say that $f$ fulfills property {\em Going Down} (abbreviated {\em GD}) if and only if, for any $\phi ,\psi \in \textup{Spec}(A)$ and any $\phi _1\in \textup{Spec}(B)$ such that $\phi \supseteq \psi $ and $f^*(\phi _1)=\phi $, there exists a $\psi _1\in \textup{Spec}(B)$ such that $\phi _1\supseteq \psi _1$ and $f^*(\psi _1)=\psi $.

We say that $f$ fulfills property {\em Lying Over} (abbreviated {\em LO}) if and only if, for any $\phi \in \textup{Spec}(A)$ such that $\textup{Ker}(f)\subseteq \phi $, there exists a $\phi _1\in \textup{Spec}(B)$ such that $f^*(\phi _1)=\phi $.\end{definition}

\begin{definition} Let $L$, $M$ be bounded lattices and $h:L\rightarrow M$ be a bounded lattice morphism.

We say that $h$ fulfills property {\em Id--Going Up} (abbreviated {\em Id--GU}) if and only if, for any $P,Q\in \textup{Spec}_{\textup{Id}}(L)$ and any $P_1\in \textup{Spec}_{\textup{Id}}(M)$ such that $P\subseteq Q$ and $h^{-1}(P_1)=P$, there exists a $Q_1\in \textup{Spec}_{\textup{Id}}(M)$ such that $P_1\subseteq Q_1$ and $h^{-1}(Q_1)=Q$.

We say that $h$ fulfills property {\em Id--Going Down} (abbreviated {\em Id--GD}) if and only if, for any $P,Q\in \textup{Spec}_{\textup{Id}}(L)$ and any $P_1\in \textup{Spec}_{\textup{Id}}(M)$ such that $P\supseteq Q$ and $h^{-1}(P_1)=P$, there exists a $Q_1\in \textup{Spec}_{\textup{Id}}(M)$ such that $P_1\supseteq Q_1$ and $h^{-1}(Q_1)=Q$.

We say that $h$ fulfills property {\em Id--Lying Over} (abbreviated {\em Id--LO}) if and only if, for any $P\in \textup{Spec}_{\textup{Id}}(L)$ such that $h^{-1}(\{0\})\subseteq P$, there exists a $P_1\in \textup{Spec}_{\textup{Id}}(M)$ such that $h^{-1}(Q_1)=Q$.\end{definition}

\begin{remark} If $L$ and $M$ are bounded distributive lattices and $h:L\rightarrow M$ is a bounded lattice morphism, then $h^{-1}(\textup{Spec}_{\textup{Id}}(M))\subseteq \textup{Spec}_{\textup{Id}}(L)$.\end{remark}

For the sake of completeness, we include here the argument for the next lemma from \cite{retic}:

\begin{lemma} For any $\alpha \in {\cal K}(A)$ and any $\phi \in {\rm Spec}(A)$, we have: $\lambda _A(\alpha )\in \phi ^*$ iff $\alpha \subseteq \phi $.\label{recallemma}\end{lemma}

\begin{proof} If $\alpha \subseteq \phi $, then $\alpha \in {\cal K}(A)\cap (\phi ]$, hence $\lambda _A(\alpha )\in \lambda _A({\cal K}(A)\cap (\phi ])=\phi ^*$.

If $\lambda _A(\alpha )\in \phi ^*=\lambda _A({\cal K}(A)\cap (\phi ])$, then, for some $\beta \in {\cal K}(A)$ such that $\beta \subseteq \phi $, we have $\lambda _A(\alpha )=\lambda _A(\beta )$, that is $\rho _A(\alpha )=\rho _A(\beta )$, so that $\phi \in V_A(\beta )=V_A(\alpha )$, thus $\alpha \subseteq \phi $.\end{proof}

\begin{lemma} For any $\phi \in {\rm Spec}(A)$, we have: ${\rm Ker}(f)\subseteq \phi $ iff ${\cal L}(f)^{-1}(\{{\bf 0}\})\subseteq \phi ^*$.\label{nuissance}\end{lemma}

\begin{proof} Note that ${\cal L}(f)^{-1}(\{{\bf 0}\})={\cal L}(f)^{-1}(\{\lambda _B(\Delta _B)\})=\{\lambda _A(\alpha )\ |\ \alpha \in {\cal K}(A),{\cal L}(f)(\lambda _A(\alpha ))=\lambda _B(\Delta _B)\}=\{\lambda _A(\alpha )\ |\ \alpha \in {\cal K}(A),\lambda _B(f^{\bullet }(\alpha ))=\lambda _B(\Delta _B)\}=\{\lambda _A(\alpha )\ |\ \alpha \in {\cal K}(A),\lambda _B(f^{\bullet }(\alpha ))=\lambda _B(\Delta _B)\}=\{\lambda _A(\alpha )\ |\ \alpha \in {\cal K}(A),\rho _B(f^{\bullet }(\alpha ))=\rho _B(\Delta _B)\}=\{\lambda _A(\alpha )\ |\ \alpha \in {\cal K}(A),f^{\bullet }(\alpha )\subseteq \rho _B(\Delta _B)\}=\{\lambda _A(\alpha )\ |\ \alpha \in {\cal K}(A),\alpha \subseteq f^*(\rho _B(\Delta _B))\}=\lambda _A({\cal K}(A)\cap (f^*(\rho _B(\Delta _B))])$.

Now let $\phi \in {\rm Spec}(A)$, and recall that $\phi ^*=\lambda _A({\cal K}(A)\cap (\phi ])$. Notice that, for any  $\alpha \in {\cal K}(A)$, $\lambda _A(\alpha )\in \lambda _A({\cal K}(A)\cap (\phi ])$ implies that, for some $\beta \in {\cal K}(A)\cap (\phi ]$, we have $\lambda _A(\alpha )=\lambda _A(\beta )$, so that $\alpha \subseteq \rho _A(\alpha )=\rho _A(\beta )\subseteq \rho _A(\phi )=\phi $, thus $\alpha \subseteq \phi $; hence: $\lambda _A(\alpha )\in \lambda _A({\cal K}(A)\cap (\phi ])$ iff $\alpha \in {\cal K}(A)\cap (\phi ]$.

Therefore: ${\cal L}(f)^{-1}(\{{\bf 0}\})\subseteq \phi ^*$ iff $\lambda _A({\cal K}(A)\cap (f^*(\rho _B(\Delta _B))])\subseteq \lambda _A({\cal K}(A)\cap (\phi ])$ iff ${\cal K}(A)\cap (f^*(\rho _B(\Delta _B))]\subseteq {\cal K}(A)\cap (\phi ]$ iff ${\cal K}(A)\cap (f^*(\rho _B(\Delta _B))]\subseteq (\phi ]$ iff every $\alpha \in {\cal K}(A)$ such that $\alpha \subseteq f^*(\rho _B(\Delta _B))$ satisfies $\alpha \subseteq \phi $ iff $\bigvee ({\cal K}(A)\cap (f^*(\rho _B(\Delta _B))])\subseteq \phi $, that is $f^*(\rho _B(\Delta _B))\subseteq \phi $.

Since $f^*(\Delta _B)\subseteq f^*(\rho _B(\Delta _B))$, by the above ${\cal L}(f)^{-1}(\{{\bf 0}\})\subseteq \phi ^*$ implies $f^*(\Delta _B)\subseteq \phi $, that is ${\rm Ker}(f)\subseteq \phi $.

On the other hand, again since $\rho _A(\phi )=\phi $, we have: $f^*(\Delta _B)={\rm Ker}(f)\subseteq \phi $ iff $\rho _A(f^*(\Delta _B))\subseteq \phi $, that is $\bigcap ({\rm Spec}(A)\cap [f^*(\Delta _B)))\subseteq \phi $, which, since $f^*({\rm Spec}(B))\subseteq {\rm Spec}(A)\cap [f^*(\Delta _B))$, implies that $f^*(\rho _B(\Delta _B))=f^*(\bigcap {\rm Spec}(B))=\bigcap f^*({\rm Spec}(B))\subseteq \bigcap ({\rm Spec}(A)\cap [f^*(\Delta _B)))\subseteq \phi $, so that ${\cal L}(f)^{-1}(\{{\bf 0}\})\subseteq \phi ^*$ by the above.\end{proof}

\begin{proposition} If $f$ is admissible, then: $f$ satisfies property GU, GD, respectively LO iff ${\cal L}(f)$ satisfies Id--GU, Id--GD, respectively Id--LO.\label{gugdloid}\end{proposition}

\begin{proof} By Proposition \ref{homeo}, the maps $u_A:{\rm Spec}(A)\rightarrow {\rm Spec}_{\rm Id}({\cal L}(A))$ and $u_B:{\rm Spec}(B)\rightarrow {\rm Spec}_{\rm Id}({\cal L}(B))$ defined by $u_A(\phi )=\phi ^*$ and $u_B(\psi )=\psi ^*$ for any $\phi \in {\rm Spec}(A)$ and any $\psi \in {\rm Spec}(B)$ are order isomorphisms.

The following diagram is commutative:

\begin{center}\begin{picture}(120,33)(0,0)
\put(-10,30){${\rm Spec}(A)$}
\put(70,30){$\textup{Spec}_{\textup{Id}}({\cal L}(A))$}
\put(-10,0){${\rm Spec}(B)$}
\put(70,0){$\textup{Spec}_{\textup{Id}}({\cal L}(B))$}
\put(26,33){\vector(1,0){42}}
\put(40,35){$u_A$}
\put(26,3){\vector(1,0){42}}
\put(40,5){$u_B$}
\put(7,7){\vector(0,1){21}}
\put(-3,14){$f^*$}
\put(86,7){\vector(0,1){21}}
\put(87,14){${\cal L}(f)^*$}
\end{picture}\end{center}\vspace*{-3pt}

Indeed, by Lemma \ref{recallemma} and the fact that $f^{\bullet}({\cal K}(A))\subseteq {\cal K}(B)$, for any $\psi \in {\rm Spec}(B)$, we have: ${\cal L}(f)^*(u_B(\psi ))={\cal L}(f)^*(\psi ^*)=\{\lambda _A(\alpha )\ |\ \alpha \in {\cal K}(A),{\cal L}(f)(\lambda _A(\alpha ))\in \psi ^*\}=\{\lambda _A(\alpha )\ |\ \alpha \in {\cal K}(A),\lambda _B(f^{\bullet}(\alpha ))\in \psi ^*\}=\{\lambda _A(\alpha )\ |\ \alpha \in {\cal K}(A),f^{\bullet}(\alpha )\subseteq \psi \}=\{\lambda _A(\alpha )\ |\ \alpha \in {\cal K}(A),\alpha \subseteq f^*(\psi )\}=\lambda _A({\cal K}(A)\cap (f^*(\psi )])=f^*(\psi )^*=u_A(f^*(\psi ))$.

Hence the statements in the enunciation on GU and GD versus Id--GU and Id--GD, respectively. By Lemma \ref{nuissance}, we have, for every $\phi \in {\rm Spec}(A)$: ${\cal L}(f)^{-1}(\{{\bf 0}\})\subseteq u_A(\phi )$ iff ${\rm Ker}(f)\subseteq \phi $, which, along with the commutativity of the diagram above, yields the statement on LO versus Id--LO in the enunciation.\end{proof}

\begin{proposition} Any dually Brouwerian lattice morphism satisfies Id--GU.\label{brougu}\end{proposition}

\begin{proof} Let $L$ and $M$ be lattices with smallest element such that $(L,\vee ,0)$ and $(M,\vee ,0)$ are dually Brouwerian join--semilattices, and $h:L\rightarrow M$ be a dually Brouwerian lattice morphism.

Let $P,Q\in {\rm Spec}_{\rm Id}(L)$ and $P_1\in {\rm Spec}_{\rm Id}(M)$ such that $P\subseteq Q$ and $h^{-1}(P_1)=P$.

Let us denote by $S=L\setminus P$ and $T=L\setminus Q$, so that $T\subseteq S$, so that $h^{-1}(P_1)\cap T=P\cap T=\emptyset $ and thus $P_1\cap h(T)=\emptyset $. By Zorn's Lemma, it follows that there exists an ideal $Q_1$ of $M$ such that $Q_1\cap h(T)=\emptyset $ and $Q_1$ is maximal w.r.t. this property, so that $P_1\subseteq Q_1$. Since $Q\in {\rm Spec}_{\rm Id}(L)$, it follows that $T$ is closed w.r.t. the meet, thus $h(T)$ is closed w.r.t. the meet, from which it immediately follows that $Q_1\in {\rm Spec}_{\rm Id}(M)$.

$h^{-1}(Q_1)\cap T\subseteq h^{-1}(Q_1)\cap h^{-1}(h(T))=h^{-1}(Q_1\cap h(T))=\emptyset $, thus $h^{-1}(Q_1)\setminus Q=h^{-1}(Q_1)\cap (L\setminus Q)=\emptyset $, therefore $h^{-1}(Q_1)\subseteq Q$.

Now let $x\in Q$ and assume by absurdum that $x\notin h^{-1}(Q_1)$, that is $h(x)\notin Q_1$, so that $Q_1\subsetneq Q_1\vee (h(x)]$ and thus $(Q_1\vee (h(x)])\cap h(T)\neq \emptyset $ by the choice of $Q_1$, so that, for some $t\in T$ and some $a\in Q_1$, $h(t)\leq h(x)\vee a$, thus $h(t\dot{-}x)=h(t)\dot{-}h(x)\leq a$, hence $h(t\dot{-}x)\in Q_1$, thus $t\dot{-}x\in h^{-1}(Q_1)\subseteq Q$, so that, since $t\dot{-}x\leq t\dot{-}x$, we have $t\leq (t\dot{-}x)\vee x\in Q$, thus $t\in Q=L\setminus T$, and we have a contradiction. Hence $Q\subseteq h^{-1}(Q_1)$, therefore $h^{-1}(Q_1)=Q$.\end{proof}

The proof of the proposition above follows the lines of analogous results for MV--algebras and BL--algebras from \cite{bel} and \cite{rada}, respectively. The two previous propositions yield the following result from \cite{gulo} as a corollary:

\begin{corollary} If ${\cal C}$ has EDPC and $f$ is admissible, then $f$ satisfies GU.\end{corollary}

\section{Functoriality of the Boolean Center}
\label{fbc}

Throughout this section, $B$ will be a member of $\C $, $f:A\rightarrow B$ will be a morphism, and we will assume that $\nabla _A\in {\cal K}(A)$, $\nabla _B\in {\cal K}(B)$, the commutators of $A$ and $B$ are commutative and distributive w.r.t. arbitrary joins, all of which hold in the particular case when $\C $ is congruence--modular and semi--degenerate. We will also assume that ${\cal K}(A)$ and ${\cal K}(B)$ are closed w.r.t. the commutators of $A$ and $B$, respectively.

If ${\cal B}({\rm Con}(A))$ and ${\cal B}({\rm Con}(B))$ are Boolean sublattices of ${\rm Con}(A)$ and ${\rm Con}(B)$, respectively, then we say that $f$ satisfies the {\em functoriality of the Boolean center} (abbreviated {\em FBC}) iff:

\begin{flushleft}\begin{tabular}{ll}
(FBC1) & $f^{\bullet }({\cal B}({\rm Con}(A)))\subseteq {\cal B}({\rm Con}(B))$;\\ 
(FBC2) & $f^{\bullet }\mid _{{\cal B}({\rm Con}(A))}:{\cal B}({\rm Con}(A))\rightarrow {\cal B}({\rm Con}(B))$ is a Boolean morphism.\end{tabular}\end{flushleft}

Throughout the rest of this section, we will also assume that $[\alpha ,\nabla _A]_A=\alpha $ for all $\alpha \in {\rm Con}(A)$ and $[\beta ,\nabla _B]_B=\beta $ for all $\beta \in {\rm Con}(B)$, which also hold in the particular case when $\C $ is congruence--modular and semi--degenerate.

Under the conditions above, by \cite[Lemma $24$]{retic}, ${\cal B}({\rm Con}(A))$ is a Boolean sublattice of ${\rm Con}(A)$, on which the commutator coincides with the intersection; moreover, by \cite[Lemma $18$, (iv)]{retic}, for all $\sigma \in {\cal B}({\rm Con}(A))$ and all $\theta \in {\rm Con}(A)$, we have $[\sigma ,\theta ]_A=\sigma \cap \theta $; also, for all $\alpha ,\beta \in {\rm Con}(A)$ such that $\alpha \vee \beta =\nabla _A$, we have $[\alpha ,\beta ]_A=\alpha \cap \beta $. By \cite[Proposition $19$, (iv)]{retic}, ${\cal B}({\rm Con}(A))\subseteq {\cal K}(A)$, so that $\lambda _A({\cal B}({\rm Con}(A)))\subseteq {\cal B}({\cal L}(A))$ and $\lambda _A\mid _{{\cal B}({\rm Con}(A))}:{\cal B}({\rm Con}(A))\rightarrow {\cal B}({\cal L}(A))$ is a Boolean morphism.

\begin{lemma}{\rm \cite[Theorem $5$, $(i)$]{retic}} If $\C $ is congruence--modular and semi--degenerate, then the Boolean morphism $\lambda _A\mid _{{\cal B}({\cal L}(A))}:{\cal B}({\cal L}(A))\rightarrow {\cal B}({\cal L}(B))$ is injective. If, furthermore, $A$ is semiprime or its commutator is associative, then this restriction of $\lambda _A$ is a Boolean isomorphism.\label{boolinj}\end{lemma}

\begin{lemma}{\rm \cite[Lemma $25$]{retic}} If $\C $ is congruence--modular and semi--degenerate and $A$ is semiprime, then, for all $\alpha \in {\rm Con}(A)$: $\lambda _A(\alpha )\in {\cal B}({\cal L}(A))$ iff $\alpha \in {\cal B}({\rm Con}(A))$.\label{booliff}\end{lemma}

\begin{remark}  Since ${\cal B}({\rm Con}(A))\subseteq {\cal K}(A)\subseteq {\rm Con}(A)$, it follows that, if ${\rm Con}(A)$ is a Boolean lattice, in particular if $A$ is simple, then ${\cal B}({\rm Con}(A))={\cal K}(A)={\rm Con}(A)$.

Since the same holds for $B$, we may notice that: $f$ satisfies (FBC1) if ${\cal B}({\rm Con}(B))={\cal K}(B)$, in particular if ${\rm Con}(B)$ is a Boolean lattice, in particular if $B$ is simple.\end{remark}

\begin{remark} If $f$ satisfies (FBC1), $f^{\bullet }\mid _{{\cal K}(A)}:{\cal K}(A)\rightarrow {\cal K}(B)$ preserves the commutator and $f^{\bullet }(\nabla _A)=\nabla _B$, the latter holding if $f$ is surjective or $\C $ is a variety with $\vec{0}$ and $\vec{1}$, then, since the commutators of $A$ and $B$ coincide to the intersection on ${\cal B}({\rm Con}(A))$ and ${\cal B}({\rm Con}(B))$, respectively, it follows that $f$ satisfies FBC.

In particular, $f$ satisfies FBC if $f^{\bullet }:{\rm Con}(A)\rightarrow {\rm Con}(B)$ is a bounded lattice morphism, that is if:\begin{itemize}
\item $f^{\bullet }(\nabla _A)=\nabla _B$, in particular if $f$ is surjective or $\C $ is a variety with $\vec{0}$ and $\vec{1}$, and:
\item $f^{\bullet }$ preserves the intersection, in particular if $f$ is surjective and the commutators of $A$ and $B$ coincide to the intersection, in particular if $f$ is surjective and $\C $ is congruence--distributive.\end{itemize}\end{remark}

\begin{remark} If $f$ fulfills FRet and ${\cal L}(f):{\cal L}(A)\rightarrow {\cal L}(B)$ is a bounded lattice morphism, then $f$ fulfills FBC and the image of ${\cal L}(f)$ through the functor ${\cal B}$ is ${\cal B}({\cal L}(f))={\cal L}(f)\mid _{{\cal B}({\cal L}(A))}:{\cal B}({\cal L}(A))\rightarrow {\cal B}({\cal L}(B))$.

If all morphisms in $\C $ fulfill FRet and ${\cal L}$ is a functor from $\C $ to the variety of bounded distributive lattices, then ${\cal B}\circ {\cal L}$ is a functor from $\C $ to the variety of Boolean algebras.\end{remark}

Thus, in view of Proposition \ref{cipnedpc}:

\begin{corollary} If $\C $ is a congruence--distributive variety with $\vec{0}$ and $\vec{1}$ and the CIP, then every morphism in $\C $ fulfills FBC.\end{corollary}

\begin{remark} ${\cal B}\circ {\cal L}$ does not preserve surjectivity, as shown by the example of the surjective morphism $h:{\cal N}_5\rightarrow {\cal L}_2^2$ from Example \ref{pentagon}. Note, also, that the bounded lattice morphism ${\cal L}(h)$ is surjective, but the Boolean morphism ${\cal B}({\cal L}(h))$ is not surjective.

On the other hand, notice the bounded lattice embedding $i_{{\cal L}_2,{\cal N}_5}$ from Example \ref{pentagon}, in whose case the Boolean morphism ${\cal B}({\cal L}(i_{{\cal L}_2,{\cal N}_5}))$ is surjective, while the bounded lattice morphism ${\cal L}(i_{{\cal L}_2,{\cal N}_5})$ is not surjective.\end{remark}

\begin{proposition} If:\begin{itemize}
\item $\C $ is congruence--modular and semi--degenerate,
\item $f$ fulfills FRet and ${\cal L}(f)$
preserves the ${\bf 1}$,
\item ${\cal L}(f)\mid _{{\cal B}({\cal L}(A))}$ preserves the meet, in particular if ${\cal L}(f)$ preserves the meet,
\item and $B$ is semiprime,\end{itemize}

\noindent then $f$ fulfills FBC.\label{suffbc}\end{proposition}

\begin{proof} Since $f^{\bullet }$ preserves the join and thus so does ${\cal L}(f)$, it follows that ${\cal L}(f)\mid _{{\cal B}({\cal L}(A))}:{\cal B}({\cal L}(A))\rightarrow {\cal L}(B)$ is a bounded lattice morphism, hence ${\cal L}(f)({\cal B}({\cal L}(A)))\subseteq {\cal B}({\cal L}(B))$ and so ${\cal L}(f)\mid _{{\cal B}({\cal L}(A))}:{\cal B}({\cal L}(A))\rightarrow {\cal B}({\cal L}(B))$ is a bounded lattice morphism, thus a Boolean morphism.

Let $\alpha \in {\cal B}({\rm Con}(A))$. Then $\lambda _A(\alpha )\in {\cal B}({\cal L}(A))$, thus, by the above, $\lambda _B(f^{\bullet }(\alpha ))={\cal L}(f)(\lambda _A(\alpha ))\in {\cal B}({\cal L}(B))$, so that $f^{\bullet }(\alpha )\in {\cal B}({\rm Con}(B))$ by Lemma \ref{booliff}. Hence $f^{\bullet }({\cal B}({\rm Con}(A)))\subseteq {\cal B}({\rm Con}(B))$.

Trivially, $f^{\bullet }(\Delta _A)=\Delta _B$. We have $\lambda _B(f^{\bullet }(\nabla _A))={\cal L}(f)(\lambda _A(\nabla _A))={\cal L}(f)({\bf 1})={\bf 1}=\lambda _B(\nabla _B)$, thus $f^{\bullet }(\nabla _A)=\nabla _B$ by Lemma \ref{boolinj}. Let $\alpha ,\beta \in {\cal B}({\rm Con}(A))\subseteq {\cal K}(A)$. Then $\lambda _B(f^{\bullet }(\alpha \cap \beta ))={\cal L}(f)(\lambda _A(\alpha \cap \beta ))={\cal L}(f)(\lambda _A(\alpha )\wedge \lambda _A(\beta ))={\cal L}(f)(\lambda _A(\alpha ))\wedge {\cal L}(f)(\lambda _A(\beta ))=\lambda _B(f^{\bullet }(\alpha ))\wedge \lambda _B(f^{\bullet }(\beta ))=\lambda _B(f^{\bullet }(\alpha )\cap f^{\bullet }(\beta ))$, so that $f^{\bullet }(\alpha \cap \beta )=f^{\bullet }(\alpha )\cap f^{\bullet }(\beta )$ by Lemma \ref{boolinj}. Therefore $f^{\bullet }\mid _{{\cal B}({\rm Con}(A))}:{\cal B}({\rm Con}(A))\rightarrow {\cal B}({\rm Con}(B))$ is a Boolean morphism.\end{proof}

\begin{corollary} If:\begin{itemize}
\item $\C $ is semi--degenerate,
\item $f^{\bullet }(\nabla _A)=\nabla _B$ and $f^{\bullet }(\alpha \cap \beta )=f^{\bullet }(\alpha )\cap f^{\bullet }(\beta )$ for all $\alpha ,\beta \in {\cal B}({\rm Con}(A))$,
\item $\C $ is congruence--modular and the commutators of $A$ and $B$ coincide to the intersection, in particular if $\C $ is congruence--distributive,\end{itemize}

\noindent then $f$ fulfills FBC.\end{corollary}

\begin{proposition}\begin{itemize}
\item FRet does not imply FBC, not even in congruence--distributive varieties.
\item FBC does not imply FRet.\end{itemize}\end{proposition}

\begin{proof}The lattice morphism $g$ in Example \ref{pentagon} fulfills the FRet, but fails the FBC.

The morphism $h$ in Example \ref{tip20} satisfies FBC, but fails the FRet.\end{proof}

\begin{remark} If $f$ fulfills FBC and $f^{\bullet }(\nabla _A)=\nabla _B$, in particular if $f$ fulfills FBC and FRet, then ${\cal L}(f)$ preserves the ${\bf 1}$, but, as shown by the case of the bounded lattice morphism $k$ in Example \ref{pentagon}, ${\cal L}(f)$ does not necessarily preserve the meet.\end{remark}

\begin{remark} If the commutators of $A$ and $B$ coincide to the intersection and the lattices ${\rm Con}(A)$ and ${\rm Con}(B)$ are Boolean, then the following are equivalent:\begin{itemize}
\item $f$ fulfills FBC;
\item $f$ fulfills FRet and ${\cal L}(f)$ preserves the meet and the ${\bf 1}$.\end{itemize}\end{remark}

\begin{remark} If $f$ fulfills FRet and FBC, then ${\cal L}(f)\mid _{{\cal B}({\cal L}(A))}:{\cal B}({\cal L}(A))\rightarrow {\cal B}({\cal L}(B))$ is a Boolean morphism.

\begin{center}
\begin{picture}(120,40)(0,0)
\put(-26,30){${\cal B}({\rm Con}(A))$}
\put(82,30){${\cal B}({\rm Con}(B))$}
\put(-14,0){${\cal B}({\cal L}(A))$}
\put(90,0){${\cal B}({\cal L}(B))$}
\put(23,33){\vector(1,0){57}}
\put(25,39){$f^{\bullet }\mid _{{\cal B}({\rm Con}(A))}$}
\put(24,3){\vector(1,0){65}}
\put(29,8){${\cal L}(f)\mid _{{\cal B}({\cal L}(A))}$}
\put(12,27){\vector(0,-1){18}}
\put(-49,17){$\lambda _A\mid _{{\cal B}({\rm Con}(A))}$}
\put(94,27){\vector(0,-1){18}}
\put(96,17){$\lambda _B\mid _{{\cal B}({\rm Con}(B))}$}
\end{picture}\end{center}\end{remark}

\begin{remark} Obviously, whenever ${\cal L}(f):{\cal L}(A)\rightarrow {\cal L}(B)$ is injective, it follows that ${\cal L}(f)\mid _{{\cal B}({\cal L}(A))}:{\cal B}({\cal L}(A))\rightarrow {\cal B}({\cal L}(B))$ is injective, as well.\label{linjbl}\end{remark}

\begin{corollary}\begin{itemize}\item If $\C $ has EDPC and $f$ is injective, then ${\cal L}(f)\mid _{{\cal B}({\cal L}(A))}:{\cal B}({\cal L}(A))\rightarrow {\cal B}({\cal L}(B))$ is injective.
\item If $\C $ is a variety with $\vec{0}$ and $\vec{1}$, EDPC and the CIP, then the functor ${\cal B}\circ {\cal L}$ preserves injectivity.\end{itemize}\label{edpcblinj}\end{corollary}

\begin{proof} By Remark \ref{linjbl} and Propositions \ref{edpclinj} and \ref{cipnedpc}.\end{proof}

\begin{proposition} If $f^{\bullet }(\nabla _A)=\nabla _B$ and $f^{\bullet }\mid _{{\cal B}({\rm Con}(A))}$ preserves the intersection, in particular if $f^{\bullet }$ preserves the commutator, then $f$ fulfills the FBC.\label{suf4fbc}\end{proposition}

\begin{proof}Let $\alpha \in {\cal B}({\rm Con}(A))$, so that, for some $\beta \in {\cal B}({\rm Con}(A))$, $\alpha \vee \beta =\nabla _A$ and $[\alpha ,\beta ]_A=\alpha \cap \beta =\Delta _A$. Then $f^{\bullet }(\alpha )\vee f^{\bullet }(\beta )=f^{\bullet }(\alpha \vee \beta )=f^{\bullet }(\nabla _A)=\nabla _B$ and thus $f^{\bullet }(\alpha )\cap f^{\bullet }(\beta )=[f^{\bullet }(\alpha ),f^{\bullet }(\beta )]_B=f^{\bullet }([\alpha ,\beta ]_A)=f^{\bullet }(\Delta _A)=\Delta _B$, hence $f^{\bullet }(\alpha )\in {\cal B}({\rm Con}(B))$, so $f$ fulfills ${\rm FBC1}$. Also, $f^{\bullet }(\Delta _A)=\Delta _B$, $f^{\bullet }(\nabla _A)=\nabla _B$ and $f^{\bullet }$ preserves the join and the commutator, that is the intersection on ${\cal B}({\rm Con}(A))$.\end{proof}

\begin{corollary} If $\C $ is congruence--modular and $f$ is surjective, then $f$ fulfills the FBC.\label{cgmodfsurj}\end{corollary}

\begin{definition} We say that a $\theta \in {\rm Con}(A)$ fulfills the {\em Congruence Boolean Lifting Property} (abbreviated {\em CBLP}) iff the map $p_{\theta }^{\bullet }\mid _{{\cal B}({\rm Con}(A))}=p_{\theta }\mid _{{\cal B}({\rm Con}(A))}:{\cal B}({\rm Con}(A))\rightarrow {\cal B}({\rm Con}(A/\theta ))$ is surjective. We say that $A$ fulfills the {\em Congruence Boolean Lifting Property} ({\em CBLP}) iff all congruences of $A$ satisfy the CBLP.\label{defcblp}\end{definition}

For instance, if $\theta \in {\rm Con}(A)$ such that $A/\theta $ is simple, so that ${\cal B}({\rm Con}(A/\theta ))={\rm Con}(A/\theta )\cong {\cal L}_2$, then $\theta $ satisfies the CBLP, so, in particular, any maximal congruence of $A$ has the CBLP.

Throughout the rest of this section, $\C $ will be congruence--modular.

\begin{remark} Let $\theta \in {\rm Con}(A)$. Then, by Lemma \ref{fcg}, $p_{\theta }^{\bullet }:{\rm Con}(A)\rightarrow {\rm Con}(A/\theta )$ is defined by $p_{\theta }^{\bullet }(\alpha )=(\alpha \vee \theta )/\theta $ for all $\alpha \in {\rm Con}(A)$, and, by Corollary \ref{cgmodfsurj}, the map $p_{\theta }^{\bullet }\mid _{{\cal B}({\rm Con}(A))}=p_{\theta }\mid _{{\cal B}({\rm Con}(A))}:{\cal B}({\rm Con}(A))\rightarrow {\cal B}({\rm Con}(A/\theta ))$ is well defined and it is a Boolean morphism.\end{remark}

\begin{lemma} Let $\alpha ,\beta \in {\rm Con}(A)$ with $\beta \subseteq \alpha $.\begin{enumerate}
\item\label{quocblp1} If $\beta $ and $\alpha /\beta $ have the CBLP, then $\alpha $ has the CBLP.
\item\label{quocblp2} If $\alpha $ has the CBLP, then $\alpha /\beta $ has the CBLP.\end{enumerate}\label{quocblp}\end{lemma}

\begin{proof} By the Second Isomorphism Theorem, the map $\varphi _{\alpha ,\beta }:A/\alpha \rightarrow (A/\beta )/(\alpha /\beta )$, defined by $\varphi _{\alpha ,\beta }(a/\alpha )=(a/\beta )/(\alpha /\beta )$ for all $a\in A$, is an isomorphism in $\C $, so that $\varphi _{\alpha ,\beta }^{\bullet }:{\rm Con}(A/\alpha )\rightarrow {\rm Con}((A/\beta )/(\alpha /\beta ))$ is a lattice isomorphism and thus ${\cal B}(\varphi _{\alpha ,\beta }^{\bullet }):{\cal B}({\rm Con}(A/\alpha ))\rightarrow {\cal B}({\rm Con}((A/\beta )/(\alpha /\beta )))$ is a Boolean isomorphism. For all $\theta \in {\rm Con}(A)$, $\varphi _{\alpha ,\beta }^{\bullet }(p_{\alpha }^{\bullet }(\theta ))=\varphi _{\alpha ,\beta }^{\bullet }((\theta \vee \alpha )/\alpha )=((\theta \vee \alpha )/\beta )/(\alpha /\beta )=((\theta \vee \beta \vee \alpha )/\beta )/(\alpha /\beta )=((\theta \vee \beta )/\beta \vee \alpha /\beta )/(\alpha /\beta )=p_{\alpha /\beta }^{\bullet }((\theta \vee \beta )/\beta )=p_{\alpha /\beta }^{\bullet }(p_{\beta }^{\bullet }(\theta ))$, hence the following leftmost diagram is commutative, thus so is the rightmost diagram below, hence the implications in the enunciation:

\begin{center}\begin{tabular}{cc}
\hspace*{-70pt}
\begin{picture}(120,40)(0,0)
\put(-7,30){${\rm Con}(A)$}
\put(82,30){${\rm Con}(A/\alpha )$}
\put(-12,0){${\rm Con}(A/\beta )$}
\put(70,0){${\rm Con}((A/\alpha )/(\alpha /\beta ))$}
\put(26,33){\vector(1,0){54}}
\put(45,37){$p_{\alpha }^{\bullet }$}
\put(32,3){\vector(1,0){36}}
\put(40,9){$p_{\alpha /\beta }^{\bullet }$}
\put(12,27){\vector(0,-1){18}}
\put(0,17){$p_{\beta }^{\bullet }$}
\put(94,27){\vector(0,-1){18}}
\put(96,17){$\varphi _{\alpha ,\beta }^{\bullet }$}
\end{picture}
&\hspace*{90pt}
\begin{picture}(120,40)(0,0)
\put(-26,30){${\cal B}({\rm Con}(A))$}
\put(82,30){${\cal B}({\rm Con}(A/\alpha ))$}
\put(-35,0){${\cal B}({\rm Con}(A/\beta ))$}
\put(95,0){${\cal B}({\rm Con}((A/\alpha )/(\alpha /\beta )))$}
\put(23,33){\vector(1,0){57}}
\put(25,39){$p_{\alpha }^{\bullet }\mid _{{\cal B}({\rm Con}(A))}$}
\put(24,3){\vector(1,0){70}}
\put(24,9){$p_{\alpha /\beta }^{\bullet }\mid _{{\cal B}({\rm Con}(A/\beta ))}$}
\put(12,27){\vector(0,-1){18}}
\put(-43,17){$p_{\beta }^{\bullet }\mid _{{\cal B}({\rm Con}(A))}$}
\put(103,27){\vector(0,-1){18}}
\put(105,17){${\cal B}(\varphi _{\alpha ,\beta }^{\bullet })$}
\end{picture}\end{tabular}\end{center}\vspace*{-15pt}\end{proof}

\begin{proposition} $A$ has the CBLP iff, for all $\theta \in {\rm Con}(A)$, $A/\theta $ has the CBLP.\label{quoprescblp}\end{proposition}

\begin{proof} By Lemma \ref{quocblp}, (\ref{quocblp2}), for the direct implication, and the fact that $A$ is isomorphic to $A/\Delta _A$, for the converse.\end{proof}

\begin{proposition} Let $\theta \in {\rm Con}(A)$. Then: $A/\theta $ is semiprime iff $\theta \in {\rm RCon}(A)$.\label{quosprime}\end{proposition}

\begin{proof} $\Delta _{A/\theta }=(\Delta _A\vee \theta )/\theta =\theta /\theta $ and $\rho _{A/\theta }(\Delta _{A/\theta })=\rho _A(\Delta _A\vee \theta )/\theta =\rho _A(\theta )/\theta $. Hence $A/\theta $ is semiprime iff $\rho _{A/\theta }(\Delta _{A/\theta })=\Delta _{A/\theta }$ iff $\rho _A(\theta )/\theta =\theta /\theta $ iff $\rho _A(\theta )=\theta $ iff $\theta \in {\rm RCon}(A)$.\end{proof}

\begin{corollary}\begin{itemize}
\item $A/\theta $ is semiprime for all $\theta \in {\rm Con}(A)$ iff ${\rm RCon}(A)={\rm Con}(A)$.
\item If the commutator of $A$ equals the intersection, then $A/\theta $ is semiprime for all $\theta \in {\rm Con}(A)$.
\end{itemize}\end{corollary}

Throughout the rest of this section, $\C $ will be congruence--modular and semi--degenerate.

Recall that an ideal $I$ of a bounded distributive lattice $L$ is said to have the {\em Id--BLP} iff the Boolean morphism ${\cal B}(\pi _I):{\cal B}(L)\rightarrow {\cal B}(L/I)$ is surjective {\rm \cite{dcggcm}}, and $L$ is said to have the {\em Id--BLP} iff all its ideals have the Id--BLP.

Recall from Section \ref{reticulatia} that, for any $\theta \in {\rm Con}(A)$, we have $\theta ^*\in {\rm Id}({\cal L}(A))$.

\begin{theorem}{\rm \cite[Theorem $7$]{retic}} For any $\theta \in {\rm Con}(A)$, the map $\varphi _{\theta }:{\cal L}(A/\theta )\rightarrow {\cal L}(A)/\theta ^*$ defined by $\varphi _{\theta }(\lambda _{A/\theta }((\alpha \vee \theta )/\theta ))=\lambda _A(\alpha )/\theta ^*$ for all $\alpha \in {\cal K}(A)$, is a lattice isomorphism.\label{presquo}\end{theorem}

\begin{lemma} Let $\theta \in {\rm Con}(A)$.\begin{itemize}
\item If $\lambda _{A/\theta }\mid _{{\cal B}({\rm Con}(A/\theta ))}:{\cal B}({\rm Con}(A/\theta ))\rightarrow {\cal B}({\cal L}(A/\theta ))$ is surjective and $\theta $ has the CBLP, then $\theta ^*$ has the Id--BLP.
\item If $\lambda _A\mid _{{\cal B}({\rm Con}(A))}:{\cal B}({\rm Con}(A))\rightarrow {\cal B}({\cal L}(A))$ is surjective and $\lambda _{A/\theta }\mid _{{\cal B}({\rm Con}(A/\theta ))}:{\cal B}({\rm Con}(A/\theta ))\rightarrow {\cal B}({\cal L}(A/\theta ))$ is bijective, then: $\theta $ has the CBLP iff $\theta ^*$ has the Id--BLP (in ${\cal L}(A)$).\end{itemize}\label{cblpidblp}\end{lemma}

\begin{proof} By the definitions, $\theta $ has the CBLP iff the Boolean morphism $p_{\theta }^{\bullet }\mid _{{\cal B}({\rm Con}(A))}:{\cal B}({\rm Con}(A))\rightarrow {\cal B}({\rm Con}(A/\theta ))$ is surjective, while $\theta ^*$ has the Id--BLP iff the Boolean morphism ${\cal B}(\pi _{\theta ^*}):{\cal B}({\cal L}(A))\rightarrow {\cal B}({\cal L}(A)/\theta ^*)$ is surjective.

The definition of the lattice isomorphism $\varphi _{\theta }$ from Theorem \ref{presquo} shows that the following leftmost diagram is commutative, hence, by considering the restrictions of the maps in this diagram to the Boolean centers, we obtain the commutative rightmost diagram below:

\begin{center}\begin{tabular}{cc}
\hspace*{-50pt}
\begin{picture}(120,40)(0,0)
\put(0,30){${\cal K}(A)$}
\put(82,30){${\cal K}(A/\theta )$}
\put(0,0){${\cal L}(A)$}
\put(82,0){${\cal L}(A/\theta )$}
\put(25,33){\vector(1,0){55}}
\put(36,37){$p_{\theta }^{\bullet }\mid _{{\cal K}(A)}$}
\put(25,3){\vector(1,0){55}}
\put(40,6){${\cal L}(p_{\theta })$}
\put(12,27){\vector(0,-1){18}}
\put(-1,17){$\lambda _A$}
\put(94,27){\vector(0,-1){18}}
\put(96,17){$\lambda _{A/\theta }$}
\put(34,-30){${\cal L}(A)/\theta ^*$}
\put(13,-3){\vector(1,-1){21}}
\put(95,-3){\vector(-1,-1){24}}
\put(11,-18){$\pi _{\theta ^*}$}
\put(85,-19){$\varphi _{\theta }$}
\end{picture}
&\hspace*{50pt}
\begin{picture}(120,40)(0,0)
\put(-26,30){${\cal B}({\rm Con}(A))$}
\put(82,30){${\cal B}({\rm Con}(A/\theta ))$}
\put(-14,0){${\cal B}({\cal L}(A))$}
\put(90,0){${\cal B}({\cal L}(A/\theta ))$}
\put(23,33){\vector(1,0){57}}
\put(25,39){$p_{\theta }^{\bullet }\mid _{{\cal B}({\rm Con}(A))}$}
\put(24,3){\vector(1,0){65}}
\put(27,8){${\cal L}(p_{\theta })\mid _{{\cal B}({\cal L}(A))}$}
\put(12,27){\vector(0,-1){18}}
\put(-46,17){$\lambda _A\mid _{{\cal B}({\rm Con}(A))}$}
\put(103,27){\vector(0,-1){18}}
\put(105,17){$\lambda _{A/\theta }\mid _{{\cal B}({\rm Con}(A/\theta ))}$}
\put(34,-30){${\cal B}({\cal L}(A)/\theta ^*)$}
\put(13,-3){\vector(1,-1){21}}
\put(110,-3){\vector(-1,-1){24}}
\put(-8,-18){${\cal B}(\pi _{\theta ^*})$}
\put(100,-20){${\cal B}(\varphi _{\theta })$}
\end{picture}\end{tabular}\end{center}\vspace*{25pt}Thus ${\cal L}(p_{\theta })\mid _{{\cal B}({\cal L}(A))}\circ \lambda _A\mid _{{\cal B}({\rm Con}(A))}=\lambda _{A/\theta }\mid _{{\cal B}({\rm Con}(A/\theta ))}\circ p_{\theta }^{\bullet }\mid _{{\cal B}({\rm Con}(A))}$, hence the statements in the enunciation.\end{proof}

\begin{proposition} Let $\theta \in {\rm Con}(A)$.\begin{itemize}
\item If $\theta \in {\rm RCon}(A)$ and $\theta $ has CBLP, then $\theta ^*$ has the Id--BLP. 
\item If $\Delta _A,\theta \in {\rm RCon}(A)$, then: $\theta $ has CBLP iff $\theta ^*$ has the Id--BLP.
\item If the commutator of $A/\theta $ is associative and $\theta $ has CBLP, then $\theta ^*$ has the Id--BLP.
\item If the commutators of $A$ and $A/\theta $ are associative, then: $\theta $ has CBLP iff $\theta ^*$ has the Id--BLP.\end{itemize}\label{cblpvsidblp}\end{proposition}

\begin{proof} By Lemmas \ref{cblpidblp} and Lemma \ref{boolinj} and Proposition \ref{quosprime}.\end{proof}

\begin{theorem}\begin{itemize}
\item If ${\rm RCon}(A)={\rm Con}(A)$, then: $A$ has the CBLP iff ${\cal L}(A)$ has the Id--BLP.
\item If the commutator in $\C $ is associative, then: $A$ has the CBLP iff ${\cal L}(A)$ has the Id--BLP.\end{itemize}\end{theorem}

\begin{proof} By Propositions \ref{cblpvsidblp} and \ref{mapstarsurj}.\end{proof}

\begin{proposition} Let $n\in \N ^*$, $M_1,\ldots ,M_n$ be members of $\C $ and $\theta _1\in {\rm Con}(M_1),\ldots ,\theta _n\in {\rm Con}(M_n)$. Then:\begin{enumerate}
\item\label{prodcblp1} $\theta _1\times \ldots \times \theta _n$ has the CBLP iff $\theta _1,\ldots ,\theta _n$ have the CBLP;
\item\label{prodcblp2} $M_1\times \ldots \times M_n$ has the CBLP iff $M_1,\ldots ,M_n$ have the CBLP.\end{enumerate}\label{prodcblp}\end{proposition}

\begin{proof} (\ref{prodcblp1}) Let $M=M_1\times \ldots \times M_n$ and $\theta =\theta _1\times \ldots \times \theta _n\in {\rm Con}(M)$, and note that $M/\theta =M_1/\theta _1\times \ldots \times M_n/\theta _n$. Since $\C $ is congruence--modular and semi--degenerate, the direct products $M_1\times \ldots \times M_n$ and $M_1/\theta _1\times \ldots \times M_n/\theta _n$ have no skew congruences, hence ${\cal B}({\rm Con}(M))={\cal B}({\rm Con}(M_1)\times \ldots \times {\rm Con}(M_n))={\cal B}({\rm Con}(M_1))\times \ldots \times {\cal B}({\rm Con}(M_n))$ and ${\cal B}({\rm Con}(M/\theta ))={\cal B}({\rm Con}(M_1/\theta _1)\times \ldots \times {\rm Con}(M_n/\theta _n))={\cal B}({\rm Con}(M_1/\theta _1))\times \ldots \times {\cal B}({\rm Con}(M_n/\theta _n))$. For all $\alpha _1\in {\rm Con}(M_1),\ldots ,\alpha _n\in {\rm Con}(M_n)$, $p_{\theta }^{\bullet }(\alpha )=(\alpha \vee \theta )/\theta =((\alpha _1\vee \theta _1)/\theta _1,\ldots ,(\alpha _n\vee \theta _n)/\theta _n)=(p_{\theta _1}^{\bullet }(\alpha _1),\ldots ,p_{\theta _n}^{\bullet }(\alpha _n))$, thus $p_{\theta }^{\bullet }=p_{\theta _1}^{\bullet }\times \ldots \times p_{\theta _n}^{\bullet }$. Hence $p_{\theta }^{\bullet }\mid _{{\cal B}({\rm Con}(M))}:{\cal B}({\rm Con}(M))\rightarrow {\cal B}({\rm Con}(M/\theta ))$ is surjective iff $p_{\theta _1}^{\bullet }\mid _{{\cal B}({\rm Con}(M_1))}:{\cal B}({\rm Con}(M_1))\rightarrow {\cal B}({\rm Con}(M_1/\theta _1)),\ldots ,p_{\theta _n}^{\bullet }\mid _{{\cal B}({\rm Con}(M_n))}:{\cal B}({\rm Con}(M_n))\rightarrow {\cal B}({\rm Con}(M_n/\theta _n))$ are surjective.

\noindent (\ref{prodcblp2}) By (\ref{prodcblp1}).\end{proof}

\begin{remark} In Proposition \ref{prodcblp}, (\ref{prodcblp1}), instead of $\C $ being congruence--modular and semi--degenerate, it suffices for $\C $ to be congruence--modular and the direct product $M_1\times \ldots \times M_n$ to have no skew congruences.\end{remark}

Recall that a bounded distributive lattice $L$ is said to be {\em B--normal} iff, for all $x,y\in L$ such that $x\vee y=1$, there exist $a,b\in {\cal B}(L)$ such that $x\vee a=y\vee b=1$ and $a\wedge b=0$. $L$ is said to be {\em B--conormal} iff its dual is B--normal.

\begin{definition} We say that the algebra $A$ is {\em congruence B--normal} iff, for all $\phi ,\psi \in {\rm Con}(A)$ such that $\phi \vee \psi =\nabla _A$, there exist $\alpha ,\beta \in {\cal B}({\rm Con}(A))$ such that $\phi \vee \alpha =\psi \vee \beta =\nabla _A$ and $[\alpha ,\beta ]_A=\Delta _A$.\end{definition}

\begin{remark} If $A$ is congruence--distributive, then $A$ is congruence B--normal iff its congruence lattice is B--normal. More generally, if $A$ is semiprime, then $A$ is congruence B--normal iff its congruence lattice satisfies the B--normality condition excepting distributivity.

Congruence B--normal algebras generalize commutative exchange rings \cite[Theorem $1.7$]{nic}, quasi--local residuated lattices \cite{eu3,eu4} and congruence--distributive B--normal algebras \cite{cblp}.\end{remark}

The following proofs are very similar to those of the analogous statements from \cite[Theorem $4.28$]{cblp}, but we introduce them here for the sake of completeness.

\begin{lemma} The following are equivalent:\begin{enumerate}
\item\label{charbnorm1} $A$ is congruence B--normal;
\item\label{charbnorm2} for all $\phi ,\psi \in {\cal K}(A)$ such that $\phi \vee \psi =\nabla _A$, there exist $\alpha ,\beta \in {\cal B}({\rm Con}(A))$ such that $\phi \vee \alpha =\psi \vee \beta =\nabla _A$ and $[\alpha ,\beta ]_A=\Delta _A$.\end{enumerate}\label{charbnorm}\end{lemma}

\begin{proof} (\ref{charbnorm1})$\Rightarrow $(\ref{charbnorm2}): Trivial.

\noindent (\ref{charbnorm2})$\Rightarrow $(\ref{charbnorm1}): Let $\phi ,\psi \in {\rm Con}(A)$ such that $\phi \vee \psi =\nabla _A$, that is $\nabla _A=\bigvee \{Cg_A(a,b)\ |\ (a,b)\in \phi \cup \psi \}$. But $\nabla _A\in {\cal K}(A)$, thus, for some $n,k\in \N ^*$, there exist $(a_1,b_1),\ldots ,(a_n,b_n)\in \phi $ and $(c_1,d_1),\ldots ,(c_k,d_k)\in \psi $ such that $\nabla =\varepsilon \vee \xi $, where $\displaystyle \varepsilon =\bigvee _{i=1}^nCg_A(a_i,b_i)\in {\cal K}(A)$ and $\displaystyle \xi =\bigvee _{j=1}^kCg_A(c_j,d_j)\in {\cal K}(A)$. Hence there exist $\alpha ,\beta \in {\rm Con}(A)$ such that $[\alpha ,\beta ]_A=\Delta _A$ and $\varepsilon \vee \alpha =\xi \vee \beta =\nabla _A$, so that $\phi \vee \alpha =\psi \vee \beta =\nabla _A$ since $\varepsilon \subseteq \phi $ and $\xi \subseteq \psi $.\end{proof}

\begin{proposition}\begin{enumerate}
\item\label{bnormretic1} If $A$ is congruence B--normal, then ${\cal L}(A)$ is B--normal.
\item\label{bnormretic2} If ${\cal C}$ is congruence--modular and semi--degenerate and the Boolean morphism $\lambda _A\mid _{{\cal B}({\rm Con}(A))}:{\cal B}({\rm Con}(A))\rightarrow {\cal B}({\cal L}(A))$ is surjective, then: $A$ is congruence B--normal iff ${\cal L}(A)$ is B--normal.
\end{enumerate}\label{bnormretic}\end{proposition}

\begin{proof} (\ref{bnormretic1}) Assume that $A$ is congruence B--normal and let $\theta ,\zeta \in {\cal K}(A)$ such that $\lambda _A(\theta )\vee \lambda _A(\zeta )={\bf 1}$, that is $\lambda _A(\theta \vee \zeta )={\bf 1}$, so that $\theta \vee \zeta =\nabla _A$, hence there exist $\alpha ,\beta \in {\cal B}({\rm Con}(A))$ such that $\theta \vee \alpha =\zeta \vee \beta =\nabla _A$ and $[\alpha ,\beta ]_A=\Delta _A$, thus $\lambda _A(\alpha ),\lambda _A(\beta )\in {\cal B}({\cal L}(A))$, $\lambda _A(\theta )\vee \lambda _A(\alpha )=\lambda _A(\theta \vee \alpha )={\bf 1}=\lambda _A(\zeta \vee \beta )=\lambda _A(\zeta )\vee \lambda _A(\beta )$ and $\lambda _A(\alpha )\wedge \lambda _A(\beta )=\lambda _A([\alpha ,\beta ]_A)={\bf 0}$. Therefore ${\cal L}(A)$ is B--normal.

\noindent (\ref{bnormretic2}) Assume that ${\cal C}$ is congruence--modular and semi--degenerate and that this Boolean morphism is surjective, so that it is a Boolean isomorphism by Lemma \ref{boolinj}. By (\ref{bnormretic1}), it suffices to prove the converse implication, so assume that ${\cal L}(A)$ is B--normal, and let $\phi ,\psi \in {\cal K}(A)$ such that $\phi \vee \psi =\nabla _A$. Then $\lambda _A(\phi )\vee \lambda _A(\psi )=\lambda _A(\phi \vee \psi )={\bf 1}$, hence, by the surjectivity of $\lambda _A$ restricted to the Boolean centers, there exist $\alpha ,\beta \in {\cal B}({\rm Con}(A))$ such that $\lambda _A(\phi \vee \alpha )=\lambda _A(\phi )\vee \lambda _A(\alpha )={\bf 1}=\lambda _A(\nabla _A)=\lambda _A(\psi )\vee \lambda _A(\beta )=\lambda _A(\psi \vee \beta )$ and $\lambda _A([\alpha ,\beta ]_A)\lambda _A(\alpha )\wedge \lambda _A(\beta )={\bf 0}=\lambda _A(\Delta _A)$, therefore, by the injectivity of this Boolean morphism, $\phi \vee \alpha =\psi \vee \beta =\nabla _A$ and $[\alpha ,\beta ]_A=\Delta _A$. By Lemma \ref{charbnorm}, it follows that $A$ is congruence B--normal.\end{proof}

\begin{theorem} If ${\cal C}$ is congruence--modular and semi--degenerate and the Boolean morphism $\lambda _A\mid _{{\cal B}({\rm Con}(A))}:{\cal B}({\rm Con}(A))\rightarrow {\cal B}({\cal L}(A))$ is surjective, then the following are equivalent:\begin{enumerate}
\item\label{charcblp0} $A$ has the CBLP;
\item\label{charcblp2} ${\cal L}(A)$ has the Id--BLP;
\item\label{charcblp3} ${\cal L}(A)$ is B--normal;
\item\label{charcblp1} $A$ is congruence B--normal;
\item\label{charcblp4} the topological space $({\rm Spec}(A),\{D_A(\theta )\ |\ \theta \in {\rm Con}(A)\})$ is strongly zero--dimensional.\end{enumerate}\label{charcblp}\end{theorem}

\begin{proof} By Lemma \ref{boolinj}, $\lambda _A\mid _{{\cal B}({\rm Con}(A))}:{\cal B}({\rm Con}(A))\rightarrow {\cal B}({\cal L}(A))$ is a Boolean isomorphism.

\noindent (\ref{charcblp0})$\Leftrightarrow $(\ref{charcblp2}): By Lemma \ref{cblpidblp} and Proposition \ref{mapstarsurj}.

\noindent (\ref{charcblp2})$\Leftrightarrow $(\ref{charcblp3}): By \cite[Proposition $13$]{dcggcm}.

\noindent (\ref{charcblp3})$\Leftrightarrow $(\ref{charcblp1}): By Proposition \ref{bnormretic}, (\ref{bnormretic2}).

\noindent (\ref{charcblp1})$\Leftrightarrow $(\ref{charcblp4}): Analogously to the proof of the similar equivalence from \cite[Theorem $4.28$]{cblp}.\end{proof}

\begin{remark} By \cite{dcggcm}, ${\cal L}(A)$ is B--normal iff ${\rm Id}({\cal L}(A))$ is B--normal iff ${\rm Filt}({\cal L}(A))$ is B--conormal.\end{remark}

\begin{corollary} If ${\cal C}$ is congruence--modular and semi--degenerate and either $A$ is semiprime or its commutator is associative, then: $A$ has the CBLP iff ${\cal L}(A)$ has the Id--BLP iff ${\cal L}(A)$ is B--normal iff $A$ is congruence B--normal iff the topological space $({\rm Spec}(A),\{D_A(\theta )\ |\ \theta \in {\rm Con}(A)\})$ is strongly zero--dimensional.\end{corollary}

\begin{proof} By Theorem \ref{charcblp} and Lemma \ref{boolinj}.\end{proof}

\begin{remark} Theorem \ref{charcblp} extends results such as: commutative unitary rings with the lifting property are exactly exchange rings \cite{nic}, residuated lattices with the Boolean Lifting Property are exactly quasi--local  residuated lattices \cite{ggcm}, in semi--degenerate congruence--distributive varieties, algebras with CBLP are exactly B--normal algebras \cite[Theorem $4.28$]{cblp}.\end{remark}

\section{Particular Cases and Examples}
\label{examples}

\begin{remark} By \cite[Theorem $8.11$, p.$126$]{blyth}, the variety of distributive lattices has the PIP, thus also the CIP, since it is congruence--distributive. Therefore, by Proposition \ref{cipnedpc}, ${\cal L}$ is a functor from the variety of distributive lattices to itself, as well as from the variety of bounded distributive lattices to itself.\end{remark}

\begin{remark} $\bullet \quad $Any Boolean algebra $A$ is isomorphic to its reticulation, since ${\rm Id}(A)\cong {\rm Con}(A)$ and thus ${\rm Spec}_{\rm Id}(A)$ and ${\rm Spec}(A)$, endowed with the Stone topologies, are homeomorphic, $A$ is a bounded distributive lattice and ${\cal L}(A)$ is unique up to a lattice isomorphism.

$\bullet \quad $A finite modular lattice $L$ is isomorphic to its reticulation iff $L$ is a Boolean algebra. Indeed, the converse implication follows from the above, while, for the direct implication, we may notice that, since $L$ is congruence--distributive and finite, we have ${\cal L}(L)\cong {\cal K}(L)={\rm Con}(L)$, which is a Boolean algebra \cite{blyth,gratzer,cwdw}.

$\bullet \quad $By Remark \ref{reticacc}, a lattice without ACC can not be isomorphic to its reticulation.

$\bullet \quad $If $A$ and $B$ are algebras with the CIP and the commutators equalling the intersection having ${\rm Con}(A)\cong {\rm Con}(B)$, then ${\cal K}(A)={\rm Cp}({\rm Con}(A))$ and ${\cal K}(B)={\rm Cp}({\rm Con}(B))$ are sublattices of ${\rm Con}(A)$ and ${\rm Con}(B)$, respectively, so we have ${\cal L}(A)\cong {\cal K}(A)\cong {\cal K}(B)\cong {\cal L}(B)$.

In particular, any lattice with the CIP, thus any finite or distributive lattice, has its reticulation isomorphic to the reticulation of its dual.\end{remark}

In the following examples, we have calculated the commutators using the method from \cite{meo}. Note that, in each of these examples, the commutator is distributive w.r.t. the join, hence, by \cite[Proposition $1.2$]{agl}, the prime congruences of $A$ are the meet--irreducible elements $\phi $ of ${\rm Con}(A)$ with the property that $[\alpha ,\alpha ]_A\subseteq \phi $ implies $\alpha \subseteq \phi $ for all $\alpha \in {\rm Con}(A)$.

\begin{example} By Lemma \ref{distribsuffret}, all the algebras in this example are semiprime and all the morphisms in this example fulfill FRet, since we are in the congruence--distributive variety of lattices and the following algebras are finite, thus all their congruences are compact, so these algebras trivially satisfy the CIP. Bounded lattices form a congruence--distributive variety with $\vec{0}$ and $\vec{1}$, thus all bounded lattice morphisms in this example also satisfy the FBC, according to Proposition \ref{suffbc}.

Let us consider the congruence--distributive variety of lattices, ${\cal L}_2^2=\{0,a,b,1\}$, ${\cal L}_2=\{0,a\}$ and let us consider the lattice embedding $i_{{\cal L}_2,{\cal L}_2^2}:{\cal L}_2\rightarrow {\cal L}_2^2$. Then we may take ${\cal L}({\cal L}_2)={\cal K}({\cal L}_2)={\rm Con}({\cal L}_2)=\{\Delta _{{\cal L}_2},\nabla _{{\cal L}_2}\}\cong {\cal L}_2$ and ${\cal L}({\cal L}_2^2)={\cal K}({\cal L}_2^2)={\rm Con}({\cal L}_2^2)=\{\Delta _{{\cal L}_2^2},\phi ,\psi ,\nabla _{{\cal L}_2^2}\}\cong {\cal L}_2^2$, where ${\cal L}_2^2/\phi =\{\{0,a\},\{b,1\}\}$ and ${\cal L}_2^2/\psi =\{\{0,b\},\{a,1\}\}$. Then $i_{{\cal L}_2,{\cal L}_2^2}$ fulfills FRet, with ${\cal L}(i_{{\cal L}_2,{\cal L}_2^2})=i_{{\cal L}_2,{\cal L}_2^2}^{\bullet }$, which preserves the meet, but does not preserve the ${\bf 1}$, since $i_{{\cal L}_2,{\cal L}_2^2}^{\bullet }(\nabla _{{\cal L}_2})=Cg_{{\cal L}_2^2}(i_{{\cal L}_2,{\cal L}_2^2}(\nabla _{{\cal L}_2}))=\alpha \neq \nabla _{{\cal L}_2^2}$. Recall that, since we are in a congruence--distributive variety, $\rho _{{\cal L}_2^2}=id_{{\rm Con}({\cal L}_2^2)}$.

Here is an example of a morphism $k$ in the congruence--dis\-tri\-bu\-tive semi--degenerate variety of bounded lattices ${\cal L}(k)$ does not preserve the meet, or, equivalently, such that $k^{\bullet }$ does not preserve the intersection of congruences. Let $k:{\cal N}_5\rightarrow {\cal N}_5$ be the bounded lattice morphism defined by the table below:\vspace*{15pt}\begin{center}\begin{tabular}{cccccc}
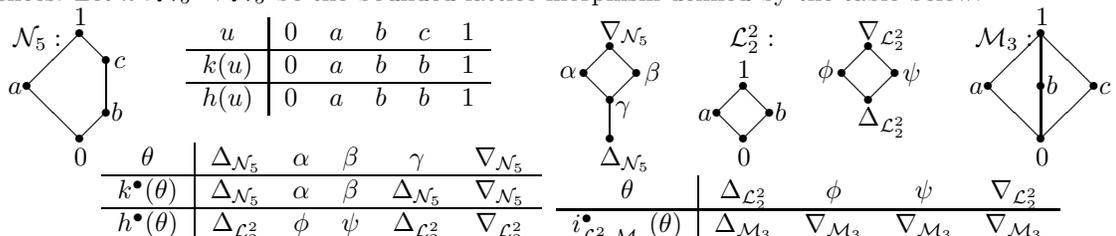
\begin{picture}(80,50)(0,0)
\put(38,10){$0$}
\put(13,37){$a$}
\put(52,27){$b$}
\put(53,47){$c$}
\put(40,20){\circle*{3}}
\put(50,30){\circle*{3}}
\put(50,50){\circle*{3}}
\put(40,60){\circle*{3}}
\put(20,40){\circle*{3}}
\put(40,20){\line(-1,1){20}}
\put(40,20){\line(1,1){10}}
\put(40,60){\line(-1,-1){20}}
\put(40,60){\line(1,-1){10}}
\put(50,30){\line(0,1){20}}
\put(38,63){$1$}
\put(15,55){${\cal N}_5:$}
\end{picture}
&\hspace*{-20pt}
\begin{picture}(160,50)(0,0)
\put(5,45){\begin{tabular}{c|ccccc}
$u$ & $0$ & $a$ & $b$ & $c$ & $1$\\ \hline 
$k(u)$ & $0$ & $a$ & $b$ & $b$ & $1$\\ \hline 
$h(u)$ & $0$ & $a$ & $b$ & $b$ & $1$\end{tabular}}
\put(-27,-3){\begin{tabular}{c|ccccc}
$\theta $ & $\Delta _{{\cal N}_5}$ & $\alpha $ & $\beta $ & $\gamma $ & $\nabla _{{\cal N}_5}$\\ \hline 
$k^{\bullet }(\theta )$ & $\Delta _{{\cal N}_5}$ & $\alpha $ & $\beta $ & $\Delta _{{\cal N}_5}$ & $\nabla _{{\cal N}_5}$\\ \hline 
$h^{\bullet }(\theta )$ & $\Delta _{{\cal L}_2^2}$ & $\phi $ & $\psi $ & $\Delta _{{\cal L}_2^2}$ & $\nabla _{{\cal L}_2^2}$\end{tabular}}
\put(145,-10){\begin{tabular}{c|cccc}
$\theta $ & $\Delta _{{\cal L}_2^2}$ & $\phi $ & $\psi $ & $\nabla _{{\cal L}_2^2}$\\ \hline 
$i_{{\cal L}_2^2,{\cal M}_3}^{\bullet }(\theta )$ & $\Delta _{{\cal M}_3}$ & $\nabla _{{\cal M}_3}$ & $\nabla _{{\cal M}_3}$ & $\nabla _{{\cal M}_3}$\end{tabular}}
\end{picture}
&\hspace*{-50pt}
\begin{picture}(80,50)(0,0)
\put(36,10){$\Delta _{{\cal N}_5}$}
\put(42,30){$\gamma $}
\put(40,20){\circle*{3}}
\put(40,35){\circle*{3}}
\put(30,45){\circle*{3}}
\put(50,45){\circle*{3}}
\put(40,55){\circle*{3}}
\put(40,20){\line(0,1){15}}
\put(40,35){\line(-1,1){10}}
\put(40,35){\line(1,1){10}}
\put(40,55){\line(-1,-1){10}}
\put(40,55){\line(1,-1){10}}
\put(21,43){$\alpha $}
\put(53,42){$\beta $}
\put(36,58){$\nabla _{{\cal N}_5}$}
\end{picture}
&\hspace*{-45pt}
\begin{picture}(80,50)(0,0)
\put(38,10){$0$}
\put(23,27){$a$}
\put(52,27){$b$}
\put(40,20){\circle*{3}}
\put(40,40){\circle*{3}}
\put(30,30){\circle*{3}}
\put(50,30){\circle*{3}}
\put(40,20){\line(-1,1){10}}
\put(40,20){\line(1,1){10}}
\put(40,40){\line(-1,-1){10}}
\put(40,40){\line(1,-1){10}}
\put(38,43){$1$}
\put(35,55){${\cal L}_2^2:$}
\end{picture}
&\hspace*{-48pt}
\begin{picture}(80,50)(0,0)
\put(36,25){$\Delta _{{\cal L}_2^2}$}
\put(40,35){\circle*{3}}
\put(30,45){\circle*{3}}
\put(50,45){\circle*{3}}
\put(40,55){\circle*{3}}
\put(40,35){\line(-1,1){10}}
\put(40,35){\line(1,1){10}}
\put(40,55){\line(-1,-1){10}}
\put(40,55){\line(1,-1){10}}
\put(21,43){$\phi $}
\put(53,42){$\psi $}
\put(36,58){$\nabla _{{\cal L}_2^2}$}
\end{picture}
&\hspace*{-30pt}
\begin{picture}(80,50)(0,0)
\put(38,10){$0$}
\put(13,37){$a$}
\put(42,37){$b$}
\put(62,37){$c$}
\put(40,20){\circle*{3}}
\put(60,40){\circle*{3}}
\put(40,60){\circle*{3}}
\put(40,40){\circle*{3}}
\put(20,40){\circle*{3}}
\put(40,20){\line(-1,1){20}}
\put(40,20){\line(1,1){20}}
\put(40,20){\line(0,1){40}}
\put(40,60){\line(-1,-1){20}}
\put(40,60){\line(1,-1){20}}
\put(38,63){$1$}
\put(15,55){${\cal M}_3:$}
\end{picture}
\end{tabular}\end{center}\vspace*{14pt}

${\cal N}_5$ has the congruence lattice above, where ${\cal N}_5/\alpha =\{\{0,b,c\},\{a,1\}\}$, ${\cal N}_5/\beta =\{\{0,a\},\{b,c,1\}\}$ and ${\cal N}_5/\gamma =\{\{0\},\{a\},\{b,c\},\{1\}\}$. We have $k^{\bullet }(\alpha )\cap k^{\bullet }(\beta )=\alpha \cap \beta =\gamma \neq \Delta _{{\cal N}_5}=k^{\bullet }(\gamma )=k^{\bullet }(\alpha \cap \beta )$.

Let us also consider ${\cal M}_3$ with the elements denoted as above and the bounded lattice embedding $i_{{\cal L}_2^2,{\cal M}_3}:{\cal L}_2^2\rightarrow {\cal M}_3$. ${\cal B}({\rm Con}({\cal M}_3))={\rm Con}({\cal M}_3)=\{\Delta _{{\cal M}_3},\nabla _{{\cal M}_3}\}\cong {\cal L}_2$. $i_{{\cal L}_2^2,{\cal M}_3}$ is injective and not surjective, but, as shown by the table above, $i_{{\cal L}_2^2,{\cal M}_3}^{\bullet }$ is surjective and not injective, hence so is ${\cal L}(i_{{\cal L}_2^2,{\cal M}_3})$, since we are in a congruence--distributive variety.

Let $h:{\cal N}_5\rightarrow {\cal L}_2^2$ be the surjective lattice morphism defined by the table above. Then $h^{\bullet }:{\rm Con}({\cal N}_5)={\cal K}({\cal N}_5)\rightarrow {\rm Con}({\cal L}_2^2)={\cal K}({\cal L}_2^2)$ is surjective, thus so is ${\cal L}(h):{\cal L}({\cal N}_5)\rightarrow {\cal L}({\cal L}_2^2)$, and $h$ fulfills the FBC, as announced above, but $h^{\bullet }\mid _{{\cal B}({\rm Con}({\cal N}_5))}:{\cal B}({\rm Con}({\cal N}_5))=\{\Delta _{{\cal N}_5},\nabla _{{\cal N}_5}\}\rightarrow {\cal B}({\rm Con}({\cal L}_2^2))={\rm Con}({\cal L}_2^2)$ is not surjective, thus neither is ${\cal B}({\cal L}(h)):{\cal B}({\cal L}({\cal N}_5))\rightarrow {\cal B}({\cal L}({\cal L}_2^2))$, since we are in a congruence--distributive variety and ${\cal N}_5$ and ${\cal L}_2^2$ are finite, so that we may take ${\cal L}({\cal N}_5)={\cal K}({\cal N}_5)={\rm Con}({\cal N}_5)$, ${\cal L}({\cal L}_2^2)={\cal K}({\cal L}_2^2)={\rm Con}({\cal L}_2^2)$ and ${\cal L}(h)=h^{\bullet }:{\rm Con}({\cal N}_5)\rightarrow {\rm Con}({\cal L}_2^2)$.

The bounded lattice embedding $i_{{\cal L}_2,{\cal N}_5}$ fulfills the FBC, as announced above, and, here as well, we may take ${\cal L}({\cal L}_2)={\cal K}({\cal L}_2)={\rm Con}({\cal L}_2)=\{\Delta _{{\cal L}_2},\nabla _{{\cal L}_2}\}={\cal B}({\rm Con}({\cal L}_2))$ and ${\cal L}(i_{{\cal L}_2,{\cal N}_5})=i_{{\cal L}_2,{\cal N}_5}^{\bullet }:{\rm Con}({\cal L}_2)\rightarrow {\rm Con}({\cal N}_5)$, so that ${\cal B}({\cal L}(i_{{\cal L}_2,{\cal N}_5}))={\cal L}(i_{{\cal L}_2,{\cal N}_5}^{\bullet })=i_{{\cal L}_2,{\cal N}_5}^{\bullet }\mid _{{\cal B}({\rm Con}({\cal L}_2))}:{\cal B}({\rm Con}({\cal L}_2))\rightarrow {\cal B}({\rm Con}({\cal N}_5))$. Since $i_{{\cal L}_2,{\cal N}_5}^{\bullet }({\cal B}({\rm Con}({\cal L}_2)))=i_{{\cal L}_2,{\cal N}_5}^{\bullet }({\rm Con}({\cal L}_2))=\{\Delta _{{\cal N}_5},\nabla _{{\cal N}_5}\}={\cal B}({\rm Con}({\cal N}_5))\subsetneq {\rm Con}({\cal N}_5)$, it follows that ${\cal B}({\cal L}(i_{{\cal L}_2,{\cal N}_5}))$ is surjective, while ${\cal L}(i_{{\cal L}_2,{\cal N}_5})$ is not surjective.

Here is a lattice morphism that fails FBC, and, since it is a morphism between finite lattices, it satisfies FRet, as all morphisms above: let $g:{\cal L}_2^2\rightarrow {\cal N}_5$ be defined by the following table, so that $g^{\bullet }$ has this definition:

\begin{center}\begin{tabular}{cc}
\begin{tabular}{c|cccc}
$u$ & $0$ & $a$ & $b$ & $1$\\ \hline 
$g(u)$ & $0$ & $0$ & $b$ & $b$\end{tabular}
&
\begin{tabular}{c|cccc}
$\theta $ & $\Delta _{{\cal L}_2^2}$ & $\phi $ & $\psi $ & $\nabla _{{\cal L}_2^2}$\\ \hline 
$g^{\bullet }(\theta )$ & $\Delta _{{\cal N}_5}$ & $\Delta _{{\cal N}_5}$ & $\alpha $ & $\alpha $\end{tabular}
\end{tabular}\end{center}

We have $g^{\bullet }({\cal B}({\rm Con}({\cal L}_2^2)))=g^{\bullet }({\rm Con}({\cal L}_2^2))=\{\Delta _{{\cal N}_5},\alpha \}\nsubseteq \{\Delta _{{\cal N}_5},\nabla _{{\cal N}_5}\}={\cal B}({\rm Con}({\cal N}_5))$, thus $g$ fails (FBC1).\label{pentagon}\end{example}

\begin{example} Let $\tau =(2)$ and let us consider the following $\tau $--algebra from \cite[Example $4$]{retic}: $N=(\{a,b,c,x,y\},$\linebreak $+^N)$, with $+^N:N^2\rightarrow N$ defined by the following table. Note that some of the congruences of $N$, as well as of the algebra $M$ from the same example, have been omitted in \cite{retic}; here is the correct Hasse diagram of ${\rm Con}(N)$, where: $N/\delta =\{\{a,b\},\{c\},\{x\},\{y\}\}$, $N/\eta _1=\{\{a\},\{b,c\},\{x\},\{y\}\}$, $N/\eta =\{\{a,b,c\},\{x\},\{y\}\}$, $N/\omega _1=\{\{a\},\{b\},\{c\},\{x,y\}\}$, $N/\omega _1=\{\{a,b\},\{c\},\{x,y\}\}$, $N/\zeta _1=\{\{a\},\{b,c\},\{x,y\}\}$, $N/\zeta =\{\{a,b,c\},\{x,y\}\}$, $N/\varepsilon =\{\{a,b,c,x\},\{y\}\}$ and $N/\xi =\{\{a,b,c,y\},\{x\}\}$.

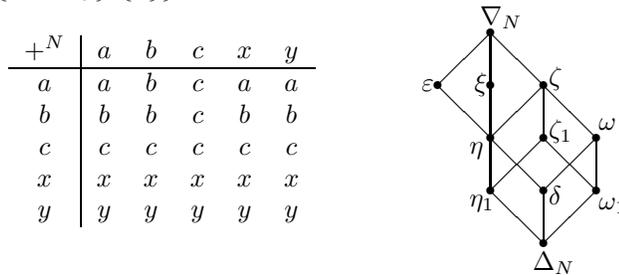
\begin{center}
\begin{tabular}{cc}
\hspace*{-60pt}\begin{picture}(207,85)(0,0)
\put(90,40){\begin{tabular}{c|ccccc}
$+^N$ & $a$ & $b$ & $c$ & $x$ & $y$\\ \hline  
$a$ & $a$ & $b$ & $c$ & $a$ & $a$\\ 
$b$ & $b$ & $b$ & $c$ & $b$ & $b$\\ 
$c$ & $c$ & $c$ & $c$ & $c$ & $c$\\ 
$x$ & $x$ & $x$ & $x$ & $x$ & $x$\\ 
$y$ & $y$ & $y$ & $y$ & $y$ & $y$\end{tabular}}\end{picture}
&\hspace*{30pt}
\begin{picture}(80,85)(0,0)
\put(40,0){\circle*{3}}
\put(36,-10){$\Delta _N$}
\put(40,0){\line(1,1){20}}
\put(40,0){\line(-1,1){20}}
\put(40,0){\line(0,1){20}}
\put(20,20){\circle*{3}}
\put(40,20){\circle*{3}}
\put(60,20){\circle*{3}}
\put(40,40){\circle*{3}}
\put(60,40){\circle*{3}}
\put(20,40){\circle*{3}}
\put(40,60){\circle*{3}}
\put(20,20){\line(0,1){20}}
\put(20,20){\line(1,1){20}}
\put(60,20){\line(-1,1){20}}
\put(60,20){\line(0,1){20}}
\put(40,20){\line(1,1){20}}
\put(40,20){\line(-1,1){20}}
\put(40,60){\line(0,-1){20}}
\put(40,60){\line(-1,-1){20}}
\put(40,60){\line(1,-1){20}}
\put(20,40){\line(-1,1){20}}
\put(40,60){\line(-1,1){20}}
\put(20,80){\line(-1,-1){20}}
\put(20,80){\line(0,-1){40}}
\put(20,80){\circle*{3}}
\put(20,60){\circle*{3}}
\put(0,60){\circle*{3}}
\put(12,14){$\eta _1$}
\put(12,34){$\eta $}
\put(42,15){$\delta $}
\put(61,13){$\omega _1$}
\put(14,58){$\xi $}
\put(-6,58){$\varepsilon $}
\put(42,40){$\zeta _1$}
\put(42,60){$\zeta $}
\put(61,43){$\omega $}
\put(16,83){$\nabla _N$}
\end{picture}\end{tabular}\end{center}\vspace*{5pt}

$[\cdot ,\cdot ]_N$ is given by the following table, so that ${\rm Spec}(N)=\{\omega \}$, thus ${\rm RCon}(N)=\{\omega ,\nabla _N\}$, hence ${\cal L}(N)={\cal K}(N)/\!\equiv _N={\rm Con}(N)/\!\equiv _N=\{(\omega ],[\omega )\}=\{{\bf 0},{\bf 1}\}\cong {\cal L}_2$. By Proposition \ref{quosprime}, since $\Delta _N\notin {\rm RCon}(N)$, while $\omega \in {\rm RCon}(N)$, $N$ is not semiprime, but $N/\omega $ is semiprime.\vspace*{-5pt}

\begin{center}
\begin{tabular}{c|ccccccccccc}
$[\cdot ,\cdot ]_N$ & $\Delta _N$ & $\delta $ & $\eta _1$ & $\eta $ & $\omega _1$ & $\omega $ & $\zeta _1$ & $\zeta $ & $\varepsilon $ & $\xi $ & $\nabla _N$\\ \hline 
$\Delta _N$ & $\Delta _N$ & $\Delta _N$ & $\Delta _N$ & $\Delta _N$ & $\Delta _N$ & $\Delta _N$ & $\Delta _N$ & $\Delta _N$ & $\Delta _N$ & $\Delta _N$ & $\Delta _N$\\ 
$\delta $ & $\Delta _N$ & $\delta $ & $\Delta _N$ & $\delta $ & $\Delta _N$ & $\delta $ & $\Delta _N$ & $\delta $ & $\delta $ & $\delta $ & $\delta $\\ 
$\eta _1$ & $\Delta _N$ & $\Delta _N$ & $\eta _1$ & $\eta _1$ & $\Delta _N$ & $\Delta _N$ & $\eta _1$ & $\eta _1$ & $\eta _1$ & $\eta _1$ & $\eta _1$\\ 
$\eta $ & $\Delta _N$ & $\delta $ & $\eta _1$ & $\eta $ & $\Delta _N$ & $\delta $ & $\eta _1$ & $\eta $ & $\eta $ & $\eta $ & $\eta $\\ 
$\omega _1$ & $\Delta _N$ & $\Delta _N$ & $\Delta _N$ & $\Delta _N$ & $\Delta _N$ & $\Delta _N$ & $\Delta _N$ & $\Delta _N$ & $\Delta _N$ & $\Delta _N$ & $\Delta _N$\\ 
$\omega $ & $\Delta _N$ & $\delta $ & $\Delta _N$ & $\delta $ & $\Delta _N$ & $\delta $ & $\Delta _N$ & $\delta $ & $\delta $ & $\delta $ & $\delta $\\ 
$\zeta _1$ & $\Delta _N$ & $\Delta _N$ & $\eta _1$ & $\eta _1$ & $\Delta _N$ & $\Delta _N$ & $\eta _1$ & $\eta _1$ & $\eta _1$ & $\eta _1$ & $\eta _1$\\ 
$\zeta $ & $\Delta _N$ & $\delta $ & $\eta _1$ & $\eta $ & $\Delta _N$ & $\delta $ & $\eta _1$ & $\eta $ & $\eta $ & $\eta $ & $\eta $\\ 
$\varepsilon $ & $\Delta _N$ & $\delta $ & $\eta _1$ & $\eta $ & $\Delta _N$ & $\delta $ & $\eta _1$ & $\eta $ & $\eta $ & $\eta $ & $\eta $\\ 
$\xi $ & $\Delta _N$ & $\delta $ & $\eta _1$ & $\eta $ & $\Delta _N$ & $\delta $ & $\eta _1$ & $\eta $ & $\eta $ & $\eta $ & $\eta $\\ 
$\nabla _N$ & $\Delta _N$ & $\delta $ & $\eta _1$ & $\eta $ & $\Delta _N$ & $\delta $ & $\eta _1$ & $\eta $ & $\eta $ & $\eta $ & $\eta $\end{tabular}\end{center}\vspace*{-3pt}

Note that ${\cal B}({\rm Con}(N))=\{\Delta _N,\omega _1,\varepsilon ,\xi ,\nabla _N\}$, which is not a sublattice of ${\rm Con}(N)$, since it is not closed w.r.t. the intersection. Note, also, that $\{a\}$ is a subalgebra of $N$, thus the variety generated by $N$ is not semi--degenerate; the same holds for all the algebras in this example, as well as those in the following example, because each of these algebras has trivial subalgebras.

Let $(P,+^P)$ be the following $\tau $--algebra: $P=\{a,b,x,y\}$, with $+^P:P^2\rightarrow P$ defined by the table that follows:\vspace*{5pt}\begin{center}\hspace*{-60pt}
\begin{tabular}{ccc}
\begin{picture}(50,95)(0,0)
\put(0,38){
\begin{tabular}{c|cccc}
$+^P$ & $a$ & $b$ & $x$ & $y$\\ \hline  
$a$ & $a$ & $b$ & $y$ & $y$\\ 
$b$ & $b$ & $b$ & $y$ & $y$\\ 
$x$ & $x$ & $x$ & $x$ & $x$\\ 
$y$ & $y$ & $y$ & $y$ & $y$\end{tabular}}\end{picture}
&\hspace*{45pt}
\begin{picture}(80,95)(0,0)
\put(40,20){\circle*{3}}
\put(36,10){$\Delta _P$}
\put(40,20){\line(1,1){20}}
\put(40,20){\line(-1,1){20}}
\put(40,20){\line(0,1){20}}
\put(20,40){\circle*{3}}
\put(40,40){\circle*{3}}
\put(60,40){\circle*{3}}
\put(20,60){\circle*{3}}
\put(40,60){\circle*{3}}
\put(60,60){\circle*{3}}
\put(40,80){\circle*{3}}
\put(20,40){\line(0,1){20}}
\put(20,40){\line(1,1){20}}
\put(60,40){\line(-1,1){20}}
\put(60,40){\line(0,1){20}}
\put(40,40){\line(1,1){20}}
\put(40,40){\line(-1,1){20}}
\put(40,80){\line(0,-1){20}}
\put(40,80){\line(-1,-1){20}}
\put(40,80){\line(1,-1){20}}
\put(12,37){$\mu $}
\put(42,34){$\chi $}
\put(62,37){$\nu $}
\put(12,59){$\phi $}
\put(42,62){$\iota $}
\put(63,59){$\psi $}
\put(36,83){$\nabla _P$}
\end{picture}
&\hspace*{-15pt}
\begin{picture}(200,95)(0,0)
\put(0,45){
\begin{tabular}{c|cccccccc}
$[\cdot ,\cdot ]_P$ & $\Delta _P$ & $\chi $ & $\phi $ & $\mu $ & $\psi $ & $\nu $ & $\iota $ & $\nabla _P$\\ \hline 
$\Delta _P$ & $\Delta _P$ & $\Delta _P$ & $\Delta _P$ & $\Delta _P$ & $\Delta _P$ & $\Delta _P$ & $\Delta _P$ & $\Delta _P$\\ 
$\chi $ & $\Delta _P$ & $\Delta _P$ & $\Delta _P$ & $\Delta _P$ & $\Delta _P$ & $\Delta _P$ & $\Delta _P$ & $\Delta _P$\\ 
$\phi $ & $\Delta _P$ & $\Delta _P$ & $\mu $ & $\mu $ & $\Delta _P$ & $\Delta _P$ & $\mu $ & $\mu $\\ 
$\mu $ & $\Delta _P$ & $\Delta _P$ & $\mu $ & $\mu $ & $\Delta _P$ & $\Delta _P$ & $\mu $ & $\mu $\\ 
$\psi $ & $\Delta _P$ & $\Delta _P$ & $\Delta _P$ & $\Delta _P$ & $\nu $ & $\nu $ & $\nu $ & $\nu $\\ 
$\nu $ & $\Delta _P$ & $\Delta _P$ & $\Delta _P$ & $\Delta _P$ & $\nu $ & $\nu $ & $\nu $ & $\nu $\\ 
$\iota $ & $\Delta _P$ & $\Delta _P$ & $\mu $ & $\mu $ & $\nu $ & $\nu $ & $\iota $ & $\iota $\\ 
$\nabla _P$ & $\Delta _P$ & $\Delta _P$ & $\mu $ & $\mu $ & $\nu $ & $\nu $ & $\iota $ & $\iota $\end{tabular}}\end{picture}\end{tabular}
\end{center}\vspace*{5pt}

${\rm Con}(P)={\cal B}({\rm Con}(P))=\{\Delta _P,\chi ,\phi ,\psi ,\mu ,\nu ,\iota ,\nabla _P\}\cong {\cal L}_2^3$, where $P/\chi =\{\{a\},\{b\},\{x,y\}\}$, $P/\phi =\{\{a,b\},$\linebreak $\{x,y\}\}$, $P/\psi =\{\{a\},\{b,x,y\}\}$, $P/\mu =\{\{a,b\},\{x\},\{y\}\}$, $P/\nu =\{\{a\},\{x\},\{b,y\}\}$ and $P\iota =\{\{a,b,y\},\{x\}\}$, as in the diagram above.  The commutator of $P$ has the table above, hence ${\rm Spec}(P)=\{\phi ,\psi \}$, thus $\Delta _P\notin \{\phi ,\psi ,\chi ,\nabla _P\}={\rm RCon}(P)$, so $P$ is not semiprime, and ${\cal L}(P)={\cal B}({\cal L}(P))={\cal B}({\cal K}(P)/\!\equiv _P)={\cal B}({\rm Con}(P)/\!\equiv _P)={\rm Con}(P)/\!\equiv _P=\{\{\Delta _P,\chi \},\{\phi ,\mu\},\{\psi ,\nu\},\{\iota ,\nabla _P\}\}\cong {\cal L}_2^2$, hence $\lambda _P\mid _{{\cal B}({\rm Con}(P))}:{\cal B}({\rm Con}(P))={\rm Con}(P)\rightarrow {\cal B}({\cal L}(P))={\cal L}(P)$ is a surjective Boolean morphism.

Let $g:P\rightarrow N$ and $h:N\rightarrow P$ be the following $\tau $--morphisms:\vspace*{-4pt}\begin{center}\begin{tabular}{ll}\begin{tabular}{c|cccc}
$u$ & $a$ & $b$ & $x$ & $y$\\ \hline 
$g(u)$ & $a$ & $a$ & $y$ & $a$\end{tabular}
&
\begin{tabular}{c|cccccccc}
$\theta $ & $\Delta _P$ & $\chi $ & $\phi $ & $\mu $ & $\psi $ & $\nu $ & $\iota $ & $\nabla _P$\\ \hline
$g^{\bullet }(\theta )$ & $\Delta _N$ & $\xi $ & $\xi $ & $\Delta _N$ & $\xi $ & $\Delta _N$ & $\Delta _N$ & $\xi $ 
\end{tabular}\\
\begin{tabular}{c|ccccc}
$u$ & $a$ & $b$ & $c$ & $x$ & $y$\\ \hline 
$h(u)$ & $x$ & $x$ & $x$ & $y$ & $x$\end{tabular}
&
\begin{tabular}{c|ccccccccccc}
$\theta $ & $\Delta _N$ & $\delta $ & $\eta _1$ & $\eta $ & $\omega _1$ & $\omega $ & $\zeta _1$ & $\zeta $ & $\varepsilon $ & $\xi $ & $\nabla _N$\\ \hline 
$h^{\bullet }(\theta )$ & $\Delta _P$ & $\Delta _P$ & $\Delta _P$ & $\Delta _P$ & $\chi $ & $\chi $ & $\chi $ & $\chi $ & $\chi $ & $\Delta _P$ & $\chi $ \end{tabular}\end{tabular}\end{center}\vspace*{-2pt}

Then $g^{\bullet }$ and $h^{\bullet }$ have the tables above.

We have $\nabla _P\equiv _P\iota $, but $g^{\bullet }(\nabla _P)=\xi \equiv _N\!\!\!\!\!\!\!\!\!\!/\ \ \, \Delta _N=g^{\bullet }(\iota )$, hence $g$ fails FRet. Note that $g^{\bullet }$ preserves the intersection, but not the commutator, since $g^{\bullet }([\psi ,\psi ]_P)=g^{\bullet }(\nu )=\Delta _N\neq \eta =[\xi ,\xi ]_N=[g^{\bullet }(\psi ),g^{\bullet }(\psi )]_N$.

Since $h^{\bullet }({\rm Con}(N))=\{\Delta _P,\chi \}=\lambda _P(\Delta _P)$ and $[\chi ,\chi ]_P=\Delta _P$, $h$ satisfies FRet and $h^{\bullet }$ preserves the commutator. $h^{\bullet }(\varepsilon )\cap h^{\bullet }(\zeta )=\chi \cap \chi =\chi \neq \Delta _P=h^{\bullet }(\eta )=h^{\bullet }(\varepsilon \cap \zeta )$, thus $h^{\bullet }$ does not preserve the intersection, and ${\cal L}(h)({\bf 1})={\cal L}(h)(\lambda _N(\nabla _N))=\lambda _P(h^{\bullet }(\nabla _N))=\lambda _P(\chi )\neq \lambda _P(\nabla _P)={\bf 1}$.

Let $(Q,+^Q)$ be the following $\tau $--algebra: $Q=\{a,b,x,y\}$, with $+^Q:Q^2\rightarrow Q$ defined by the table below:\vspace*{10pt}\begin{center}\begin{tabular}{ccc}
\hspace*{-25pt}
\begin{picture}(100,65)(0,0)
\put(0,37){
\begin{tabular}{c|cccc}
$+^Q$ & $a$ & $b$ & $x$ & $y$\\ \hline  
$a$ & $a$ & $b$ & $x$ & $x$\\ 
$b$ & $b$ & $b$ & $y$ & $y$\\ 
$x$ & $x$ & $x$ & $x$ & $x$\\ 
$y$ & $y$ & $y$ & $y$ & $y$\end{tabular}}\end{picture}
&\hspace*{-30pt}
\begin{picture}(80,65)(0,0)
\put(36,10){$\Delta _Q$}
\put(42,30){$\gamma $}
\put(40,20){\circle*{3}}
\put(40,35){\circle*{3}}
\put(30,45){\circle*{3}}
\put(50,45){\circle*{3}}
\put(40,55){\circle*{3}}
\put(40,20){\line(0,1){15}}
\put(40,35){\line(-1,1){10}}
\put(40,35){\line(1,1){10}}
\put(40,55){\line(-1,-1){10}}
\put(40,55){\line(1,-1){10}}
\put(21,43){$\alpha $}
\put(53,42){$\beta $}
\put(36,58){$\nabla _Q$}
\end{picture}
&\hspace*{-35pt}
\begin{picture}(200,65)(0,0)
\put(0,40){
\begin{tabular}{c|ccccc|c}
$[\cdot ,\cdot ]_Q$ & $\Delta _Q$ & $\alpha $ & $\beta $ & $\gamma $ & $\nabla _Q$ & $\rho _Q(\cdot )$\\ \hline 
$\Delta _Q$ & $\Delta _Q$ & $\Delta _Q$ & $\Delta _Q$ & $\Delta _Q$ & $\Delta _Q$ & $\nabla _Q$\\ 
$\alpha $ & $\Delta _Q$ & $\alpha $ & $\gamma $ & $\Delta _Q$ & $\alpha $ & $\alpha $\\ 
$\beta $ & $\Delta _Q$ & $\gamma $ & $\beta $ & $\Delta _Q$ & $\beta $ & $\beta $\\ 
$\gamma $ & $\Delta _Q$ & $\Delta _Q$ & $\Delta _Q$ & $\Delta _Q$ & $\Delta _Q$ & $\gamma $\\ 
$\nabla _Q$ & $\Delta _Q$ & $\alpha $ & $\beta $ & $\Delta _Q$ & $\nabla _Q$ & $\gamma $\end{tabular}}\end{picture}\end{tabular}
\end{center}\vspace*{-15pt}

Then $Q$ has the congruence lattice represented above, with $Q/\alpha =\{\{a,b\},\{x,y\}\}$, $Q/\beta =\{\{a\},\{b,x,y\}\}$ and $Q/\gamma =\{\{a\},\{b\},\{x,y\}\}$. The commutator of $Q$ has the table above, hence ${\rm Spec}(Q)=\{\alpha ,\beta \}$, so $\rho _Q$ is as above and thus ${\cal L}(Q)={\cal K}(Q)/\!\equiv _Q={\rm Con}(Q)/\!\equiv _Q=\{\{\Delta _Q,\gamma \},\{\alpha \},\{\beta \},\{\nabla _Q\}\}=\{{\bf 0},\lambda _Q(\alpha ),\lambda _Q(\beta ),{\bf 1}\}\cong {\cal L}_2^2$. ${\cal B}({\rm Con}(Q))=\{\Delta _Q,\nabla _Q\}\cong {\cal L}_2$, hence the Boolean morphism $\lambda _Q\mid _{{\cal B}({\rm Con}(Q))}:{\cal B}({\rm Con}(Q))\rightarrow {\cal B}({\cal L}(Q))={\cal L}(Q)$ is injective, but not surjective.

Let $k:Q\rightarrow N$ and $l:Q\rightarrow P$ be the following $\tau $--morphisms:\vspace*{-4pt}\begin{center}\begin{tabular}{cc}
\begin{tabular}{c|cccc}
$u$ & $a$ & $b$ & $x$ & $y$\\ \hline 
$k(u)$ & $a$ & $b$ & $c$ & $c$\\ \hline 
$l(u)$ & $a$ & $b$ & $y$ & $y$\end{tabular} &\hspace*{15pt}
\begin{tabular}{c|ccccc}
$\theta $ & $\Delta _Q$ & $\alpha $ & $\beta $ & $\gamma $ & $\nabla _Q$\\ \hline 
$k^{\bullet }(\theta )$ & $\Delta _N$ & $\xi _1$ & $\psi _1$ & $\Delta _N$ & $\chi _1$\\ \hline 
$l^{\bullet }(\theta )$ & $\Delta _P$ & $\mu $ & $\nu $ & $\Delta _P$ & $\iota $\end{tabular}\end{tabular}
\end{center}\vspace*{-2pt}

Then $h^{\bullet }$ has the table above, so $h$ fulfills FRet and ${\cal L}(h)$ preserves the ${\bf 1}$, although $h^{\bullet }(\nabla _Q)\neq \nabla _M$: ${\cal L}(h)({\bf 1})={\cal L}(h)(\lambda _Q(\nabla _Q))=\lambda _M(h^{\bullet }(\nabla _Q))=\lambda _M(\varepsilon )={\bf 1}$. But ${\cal L}(h)$ does not preserve the meet, because: ${\cal L}(h)(\lambda _Q(\alpha )\wedge \lambda _Q(\beta ))={\cal L}(h)(\lambda _Q([\alpha ,\beta ]_Q))={\cal L}(h)(\lambda _Q(\Delta _Q))={\cal L}(h)({\bf 0})={\bf 0}\neq {\bf 1}={\bf 1}\wedge {\bf 1}=\lambda _M(\varepsilon )\wedge \lambda _M(\varepsilon )=\lambda _M(h^{\bullet }(\alpha ))\wedge \lambda _M(h^{\bullet }(\beta ))={\cal L}(h)(\lambda _Q(\alpha ))\wedge {\cal L}(h)(\lambda _Q(\beta ))$. $h^{\bullet }$ preserves neither the intersection, nor the commutator: $h^{\bullet }(\alpha \cap \beta )=h^{\bullet }(\gamma )=\Delta _M\neq \varepsilon =\varepsilon \cap \varepsilon =h^{\bullet }(\alpha )\cap h^{\bullet }(\beta )$ and $h^{\bullet }([\alpha ,\beta ]_Q)=h^{\bullet }(\Delta _Q)=\Delta _M\neq \varepsilon =[\varepsilon ,\varepsilon ]_M=[h^{\bullet }(\alpha ),h^{\bullet }(\beta )]_M$.

$k^{\bullet }$ has the table above, so $k$ fulfills FRet and ${\cal L}(k)$ preserves the meet and the ${\bf 1}$, although $k^{\bullet }(\nabla _Q)\neq \nabla _N$, and $k^{\bullet }$ preserves both the intersection and the commutator.

$l^{\bullet }$ is defined as above, so $l$ fulfills FRet and ${\cal L}(l)$ preserves the meet and the ${\bf 1}$, although  $l^{\bullet }(\nabla _Q)\neq \nabla _P$, and $l^{\bullet }$ preserves both the intersection and the commutator. Note that $l^{\bullet }\mid _{{\cal B}({\rm Con}(Q))}:{\cal B}({\rm Con}(Q))=\{\Delta _Q,\nabla _Q\}\rightarrow {\cal B}({\rm Con}(P))=\{\Delta _P,\mu ,\nu ,\nabla _P\}$ an injective Boolean morphism, and that, while $l$ is neither injective, nor surjective, ${\cal L}(l):{\cal L}(Q)={\cal B}({\cal L}(Q))\rightarrow {\cal L}(P)={\cal B}({\cal L}(P))\cong {\cal L}_2^2$ is a Boolean isomorphism.

Now let $(R,+^R)$ be the $\tau $--algebra defined by $R=\{a,b,c\}$ and the following table for the operation $+^R$:\vspace*{22pt}\begin{center}\begin{tabular}{ccc}
\begin{picture}(80,50)(0,0)
\put(0,38){
\begin{tabular}{c|ccc}
$+^R$ & $a$ & $b$ & $c$\\ \hline  
$a$ & $a$ & $b$ & $b$\\ 
$b$ & $b$ & $b$ & $b$\\  
$c$ & $c$ & $c$ & $c$\end{tabular}}\end{picture}
&
\begin{picture}(80,50)(0,0)
\put(36,20){$\Delta _R$}
\put(40,30){\circle*{3}}
\put(30,40){\circle*{3}}
\put(50,40){\circle*{3}}
\put(40,50){\circle*{3}}
\put(40,30){\line(-1,1){10}}
\put(40,30){\line(1,1){10}}
\put(40,50){\line(-1,-1){10}}
\put(40,50){\line(1,-1){10}}
\put(22,37){$\sigma $}
\put(53,37){$\tau $}
\put(36,53){$\nabla _R$}
\end{picture}
&\hspace*{-5pt}
\begin{picture}(100,50)(0,0)
\put(0,45){
\begin{tabular}{c|cccc}
$[\cdot ,\cdot ]_R$ & $\Delta _R$ & $\sigma $ & $\tau $ & $\nabla _R$\\ \hline 
$\Delta _R$ & $\Delta _R$ & $\Delta _R$ & $\Delta _R$ & $\Delta _R$\\ 
$\sigma $ & $\Delta _R$ & $\sigma $ & $\Delta _R$ & $\sigma $\\ 
$\tau $ & $\Delta _R$ & $\Delta _R$ & $\Delta _R$ & $\Delta _R$\\ 
$\nabla _R$ & $\Delta _R$ & $\sigma $ & $\Delta _R$ & $\sigma $
\end{tabular}}\end{picture}\end{tabular}\end{center}\vspace*{-23pt}

Then $R$ has the congruence lattice above, with $R/\sigma =\{\{a,b\},\{c\}\}$ and $R/\tau =\{\{a\},\{b,c\}\}$, and the commutator of $R$ has the previous definition, so that ${\rm Spec}(R)=\{\tau \}$ and thus ${\rm RCon}(R)=\{\tau ,\nabla _R\}$, so ${\cal L}(R)={\cal K}(R)/\!\equiv _R={\rm Con}(R)/\!\equiv _R=\{\{\Delta _R,\tau \},\{\sigma ,\nabla _R\}\}=\{{\bf 0},{\bf 1}\}\cong {\cal L}_2$, hence the Boolean morphism $\lambda _R\mid _{{\cal B}({\rm Con}(R))}:{\cal B}({\rm Con}(R))={\rm Con}(R)\rightarrow {\cal B}({\cal L}(R))={\cal L}(R)$ is surjective, but not injective.

Let $d:R\rightarrow N$, $e:R\rightarrow N$, $j:R\rightarrow N$ and $m:R\rightarrow P$ be the $\tau $--morphisms defined as follows:\vspace*{-5pt}\begin{center}
\begin{tabular}{cc}
\begin{tabular}{c|ccc}
$u$ & $a$ & $b$ & $c$\\ \hline 
$d(u)$ & $a$ & $b$ & $b$\\ 
$e(u)$ & $a$ & $c$ & $c$\\ 
$j(u)$ & $y$ & $y$ & $a$\\ 
$m(u)$ & $a$ & $y$ & $x$\end{tabular}
&\hspace*{15pt}
\begin{tabular}{c|cccc}
$\theta $ & $\Delta _R$ & $\sigma $ & $\tau $ & $\nabla _R$\\ \hline 
$d^{\bullet }(\theta )$ & $\Delta _N$ & $\delta $ & $\Delta _N$ & $\delta $\\ 
$e^{\bullet }(\theta )$ & $\Delta _N$ & $\eta $ & $\Delta _N$ & $\eta $\\ 
$j^{\bullet }(\theta )$ & $\Delta _N$ & $\Delta _N$ & $\xi $ & $\xi $\\ 
$m^{\bullet }(\theta )$ & $\Delta _P$ & $\iota $ & $\chi $ & $\nabla _P$\end{tabular}\end{tabular}\end{center}\vspace*{-3pt}

Then $d^{\bullet }$, $e^{\bullet }$, $j^{\bullet }$ and $m^{\bullet }$ have the definitions above, so $d$, $e$ and $m$ fulfill FRet, while $j$ fails FRet, since $\Delta _R\equiv _R\tau $, but $j^{\bullet }(\Delta _R)=\Delta _N\equiv _N\!\!\!\!\!\!\!\!\!\!/\ \ \; \xi =j^{\bullet }(\tau )$. Note that ${\cal L}(d)$ preserves the meet and the intersection, but not the ${\bf 1}$. ${\cal L}(e)$ and ${\cal L}(m)$ preserve the ${\bf 1}$, $m^{\bullet }$ and $e^{\bullet }$ preserve the intersection and the commutator, while $j^{\bullet }$ preserves the intersection, but not the commutator, because $j^{\bullet }([\tau ,\tau ]_R)=j^{\bullet }(\Delta _R)=\Delta _N\neq \eta =[\xi ,\xi ]_N=[j^{\bullet }(\tau ),j^{\bullet }(\tau )]_N$.\label{mnex4}\end{example}

\begin{example} Let $\tau =(2)$ and let us consider the following $\tau $--algebra from \cite[Example 6.3]{urs2} and \cite[Example 4.2]{urs3}: $U=(\{0,a,b,c,d\},+^U)$, with $+^U$ defined by the following table, along with the subalgebra $T=\{0,a,b,c\}$ of $U$, the $\tau $--embedding $i_{T,U}:T\rightarrow U$ and the $\tau $--morphism $t:U\rightarrow T$ defined by the table below:\vspace*{-3pt}

\begin{center}\begin{tabular}{ccc}
\begin{picture}(125,77)(0,0)
\put(0,35){
\begin{tabular}{c|ccccc}
$+^U$ & $0$ & $a$ & $b$ & $c$ & $d$\\ \hline  
$0$ & $0$ & $a$ & $b$ & $c$ & $d$\\ 
$a$ & $a$ & $0$ & $c$ & $b$ & $b$\\ 
$b$ & $b$ & $c$ & $0$ & $a$ & $a$\\ 
$c$ & $c$ & $b$ & $a$ & $0$ & $0$\\ 
$d$ & $d$ & $b$ & $a$ & $0$ & $0$\end{tabular}}
\put(5,-20){\begin{tabular}{c|ccccc}
$\phi $ & $\Delta _T$ & $\theta $ & $\zeta $ & $\xi $ & $\nabla _T$\\ \hline 
$i_{T,U}^{\bullet }(\phi )$ & $\Delta _U$ & $\alpha $ & $\beta $ & $\gamma $ & $\nabla _U$\end{tabular}}\end{picture}
&
\begin{picture}(80,75)(0,0)
\put(36,0){$\Delta _U$}
\put(36,73){$\nabla _U$}
\put(40,10){\circle*{3}}
\put(40,30){\circle*{3}}
\put(20,50){\circle*{3}}
\put(40,50){\circle*{3}}
\put(60,50){\circle*{3}}
\put(40,70){\circle*{3}}
\put(40,30){\line(-1,1){20}}
\put(40,30){\line(1,1){20}}
\put(40,10){\line(0,1){60}}
\put(40,70){\line(-1,-1){20}}
\put(40,70){\line(1,-1){20}}
\put(12,47){$\alpha $}
\put(43,47){$\beta $}
\put(63,47){$\gamma $}
\put(43,26){$\delta $}
\end{picture}
&\hspace*{-20pt}
\begin{picture}(200,75)(0,0)
\put(0,30){
\begin{tabular}{c|cccccc}
$[\cdot ,\cdot ]_U$ & $\Delta _U$ & $\alpha $ & $\beta $ & $\gamma $ & $\delta $ & $\nabla _U$\\ \hline 
$\Delta _U$ & $\Delta _U$ & $\Delta _U$ & $\Delta _U$ & $\Delta _U$ & $\Delta _U$ & $\Delta _U$\\ 
$\alpha $ & $\Delta _U$ & $\delta $ & $\delta $ & $\delta $ & $\delta $ & $\delta $\\ 
$\beta $ & $\Delta _U$ & $\delta $ & $\delta $ & $\delta $ & $\delta $ & $\delta $\\ 
$\gamma $ & $\Delta _U$ & $\delta $ & $\delta $ & $\delta $ & $\delta $ & $\delta $\\ 
$\delta $ & $\Delta _U$ & $\delta $ & $\delta $ & $\delta $ & $\Delta _U$ & $\delta $\\ 
$\nabla _U$ & $\Delta _U$ & $\delta $ & $\delta $ & $\delta $ & $\delta $ & $\delta $\end{tabular}}\end{picture}\\ 
\begin{picture}(80,90)(0,0)
\put(0,40){\begin{tabular}{c|ccccc}
$u$ & $0$ & $a$ & $b$ & $c$ & $d$\\ \hline 
$t(u)$ & $0$ & $a$ & $a$ & $0$ & $0$\end{tabular}}
\put(-50,10){\begin{tabular}{c|cccccc}
$\phi $ & $\Delta _U$ & $\alpha $ & $\beta $ & $\gamma $ & $\delta $ & $\nabla _U$\\ \hline 
$t^{\bullet }(\phi )$ & $\Delta _T$ & $\theta $ & $\theta $ & $\Delta _T$ & $\Delta _T$ & $\theta $\end{tabular}}
\end{picture}
&
\begin{picture}(80,90)(0,0)
\put(36,0){$\Delta _T$}
\put(36,53){$\nabla _T$}
\put(40,10){\circle*{3}}
\put(20,30){\circle*{3}}
\put(40,30){\circle*{3}}
\put(60,30){\circle*{3}}
\put(40,50){\circle*{3}}
\put(40,10){\line(-1,1){20}}
\put(40,10){\line(1,1){20}}
\put(40,10){\line(0,1){40}}
\put(40,50){\line(-1,-1){20}}
\put(40,50){\line(1,-1){20}}
\put(12,27){$\theta $}
\put(43,27){$\zeta $}
\put(63,27){$\xi $}
\end{picture}
&\hspace*{-20pt}
\begin{picture}(200,90)(0,0)
\put(0,30){
\begin{tabular}{c|cccccc}
$[\cdot ,\cdot ]_T$ & $\Delta _T$ & $\theta $ & $\zeta $ & $\xi $ & $\nabla _T$\\ \hline 
$\Delta _T$ & $\Delta _T$ & $\Delta _T$ & $\Delta _T$ & $\Delta _T$ & $\Delta _T$\\ 
$\theta $ & $\Delta _T$ & $\Delta _T$ & $\Delta _T$ & $\Delta _T$ & $\Delta _T$\\ 
$\zeta $ & $\Delta _T$ & $\Delta _T$ & $\Delta _T$ & $\Delta _T$ & $\Delta _T$\\ 
$\xi $ & $\Delta _T$ & $\Delta _T$ & $\Delta _T$ & $\Delta _T$ & $\Delta _T$\\ 
$\nabla _T$ & $\Delta _T$ & $\Delta _T$ & $\Delta _T$ & $\Delta _T$ & $\Delta _T$\end{tabular}}\end{picture}\end{tabular}
\end{center}

${\rm Con}(T)=\{\Delta _T,\theta ,\zeta ,\xi ,\nabla _T\}\cong {\cal M}_3$, with the Hasse diagram above, where $T/\theta =\{\{0,a\},\{b,c\}\}$, $T/\zeta =\{\{0,b\},\{a,c\}\}$, $T/\xi =\{\{0,c\},\{a,b\}\}$. Note that ${\cal B}({\rm Con}(T))={\rm Con}(T)$, which is not a Boolean lattice. The commutator of $T$ has the value $\Delta _T$ for every pair of congruences of $T$, so ${\rm Spec}(T)=\emptyset $, thus ${\cal L}(T)=\{{\bf 0}\}\cong {\cal L}_1$, thus, trivially, $t$ satisfies FRet. As shown by the table of $t^{\bullet }$ above, $t^{\bullet }$ preserves the commutator, but not the intersection, since $t^{\bullet }(\alpha \cap \beta )=t^{\bullet }(\delta )=\Delta _T\neq \theta =\theta \cap \theta =t^{\bullet }(\alpha )\cap t^{\bullet }(\beta )$.

$U$ has the congruence lattice represented above, where $U/\alpha =\{\{0,a\},\{b,c,d\}\}$, $U/\beta =\{\{0,b\},\{a,c,d\}\}$, $U/\gamma =\{\{0,c,d\},\{a,b\}\}$ and $U/\delta =\{\{0\},\{a\},\{b\},\{c,d\}\}$. As shown by the table of $[\cdot , \cdot ]_U$ above, calculated in \cite[Example $3$]{retic}, we have ${\rm Spec}(U)=\emptyset $, thus $\rho _U(\sigma )=\nabla _U$ for all $\sigma \in {\rm Con}(U)$, and hence ${\cal L}(U)=\{{\bf 0}\}\cong {\cal L}_1$, therefore, trivially, $i_{T,U}$ fulfills FRet. Also, trivially, ${\cal L}(i_{T,U})$ and ${\cal L}(t)$ are lattice isomorphisms. $[i_{T,U}^{\bullet }(\theta ),i_{T,U}^{\bullet }(\theta )]_U=[\alpha ,\alpha ]_U=\delta \notin i_{T,U}^{\bullet }({\rm Con}(T))$, in par\-ti\-cu\-lar $[i_{T,U}^{\bullet }(\theta ),i_{T,U}^{\bullet }(\theta )]_U\neq i_{T,U}^{\bullet }([\theta ,\theta ]_T)$. So $i_{T,U}^{\bullet }$ does not preserve the commutator, and, despite $i_{T,U}$ being injective, $i_{T,U}^{\bullet }$ does not preserve the intersection, either, since $i_{T,U}^{\bullet }(\theta \cap \zeta )=i_{T,U}^{\bullet }(\Delta _T)=\Delta _U\neq \delta =\alpha \cap \beta =i_{T,U}^{\bullet }(\theta )\cap i_{T,U}^{\bullet }(\zeta )$.

${\cal B}({\rm Con}(U))=\{\Delta _U,\nabla _U\}\cong {\cal L}_2$, hence the Boolean morphism $\lambda _U\mid _{{\cal B}({\rm Con}(U))}:{\cal B}({\rm Con}(U))\rightarrow {\cal B}({\cal L}(U))={\cal L}(U)$ is surjective, but not injective. Note that $[\phi ,\nabla _U]_U=\phi $ for all $\phi \in {\rm Con}(U)$, which proves that the stronger assumption that $\C $ is congruence--modular and semi--degenerate is necessary for the properties of ${\cal B}({\rm Con}(U))$ and this restriction of $\lambda _U$ recalled above.

Let us also consider the $\tau $--algebra $(V,+^V)$, with $V=\{0,s,t\}$ and $+^V$ defined by the following table:\vspace*{-2pt}

\begin{center}\begin{tabular}{ccc}
\begin{tabular}{c|ccc}
$+^V$ & $0$ & $s$ & $t$\\ \hline 
$0$ & $0$ & $s$ & $t$\\ 
$s$ & $s$ & $0$ & $t$\\ 
$t$ & $t$ & $t$ & $0$
\end{tabular}
&
\begin{tabular}{c}
\begin{tabular}{c|ccccc}
$u$ & $0$ & $a$ & $b$ & $c$ & $d$\\ \hline 
$h(u)$ & $0$ & $0$ & $t$ & $t$ & $t$
\end{tabular}\\ 
\begin{tabular}{c|cccccc}
$\phi $ & $\Delta _U$ & $\alpha $ & $\beta $ & $\gamma $ & $\delta $ & $\nabla _U$\\ \hline 
$h^{\bullet }(\phi )$ & $\Delta _V$ & $\Delta _V$ & $\nabla _V$ & $\nabla _V$ & $\Delta _V$ & $\nabla _V$
\end{tabular}\end{tabular}
&
\begin{tabular}{c|ccc}
$[\cdot ,\cdot ]_V$ & $\Delta _V$ & $\sigma $ & $\nabla _V$\\ \hline 
$\Delta _V$ & $\Delta _V$ & $\Delta _V$ & $\Delta _V$\\ 
$\sigma $ & $\Delta _V$ & $\Delta _V$ & $\sigma $\\ 
$\nabla _V$ & $\Delta _V$ & $\sigma $ & $\sigma $\end{tabular}\end{tabular}\end{center}

Notice that ${\rm Con}(V)=\{\Delta _V,\sigma ,\nabla _V\}\cong {\cal L}_3$, with $\sigma =eq(\{0,s\},\{t\})$, and that the commutator of $V$ has the table above, so that ${\rm Spec}(V)=\{\Delta _V\}$ and hence ${\cal L}(V)=\{\{\Delta _V\},\{\sigma ,\nabla _V\}\}\cong {\cal L}_2$. The map $h:U\rightarrow V$ defined by the table above is a $\tau $--morphism and $h^{\bullet }$ is defined as above, hence $h^{\bullet }({\cal B}({\rm Con}(U)))=h^{\bullet }(\{\Delta _U,\nabla _U\})=\{\Delta _V,\nabla _V\}={\cal B}({\rm Con}(V))$ and $h^{\bullet }\mid _{{\cal B}({\rm Con}(U))}$ is a Boolean isomorphism between ${\cal B}({\rm Con}(U))$ and ${\cal B}({\rm Con}(V))$, thus $h$ satisfies the FBC, but $\Delta _U\equiv _U\nabla _U$, while $(h^{\bullet }(\Delta _U),h^{\bullet }(\nabla _U))=(\Delta _V,\nabla _V)\notin \; \equiv _V$, thus $h$ fails FRet.

Now let us consider the map $v:V\rightarrow V$ defined by the following table. Then $v^{\bullet }$ has the following definition, thus $v$ fails FRet since $\sigma \equiv _V\nabla _V$, but $v^{\bullet }(\sigma )=\Delta _V\equiv _V\!\!\!\!\!\!\!\!\!\!/\ \ \ \sigma =v^{\bullet }(\nabla _V)$, despite the fact that $v^{\bullet }$ preserves the commutator and the intersection and $v^{\bullet }(\nabla _V)\equiv _V\nabla _V$.

\begin{center}\begin{tabular}{cc}
\begin{tabular}{c|ccc}
$u$ & $0$ & $s$ & $t$\\ \hline 
$v(u)$ & $0$ & $0$ & $s$
\end{tabular}
&\hspace*{15pt}
\begin{tabular}{c|ccc}
$\phi $ & $\Delta _V$ & $\sigma $ & $\nabla _V$\\ \hline 
$v^{\bullet }(\phi )$ & $\Delta _V$ & $\Delta _V$ & $\sigma $
\end{tabular}
\end{tabular}\end{center}\label{tip20}\end{example}

\section*{Acknowledgements}

This work was supported by the research grant number IZSEZO\_186586/1, awarded to the project {\em Re\-ti\-cu\-la\-tions of Concept Algebras} by the Swiss National Science Foundation, within the programme Scientific Exchanges.

\end{document}